\documentclass[letterpaper,11pt]{amsart}

\usepackage{amsmath}
\usepackage{amsthm}
\usepackage{mathrsfs}
\usepackage{amssymb,amscd}
\usepackage[all]{xy}
\usepackage{mathtools}
\usepackage{fullpage}
\usepackage{comment}

\usepackage[dvipsnames]{xcolor}
\usepackage[colorlinks=true,linkcolor=blue,breaklinks,pagebackref]{hyperref}

\usepackage[lite,abbrev,msc-links,alphabetic]{amsrefs}

\newcommand{\fg}{\mathfrak{g}}
\newcommand{\fh}{\mathfrak{h}}

\newcommand{\fn}{\mathfrak{n}}
\newcommand{\fo}{\mathfrak{o}}

\newcommand{\fs}{\mathfrak{s}}

\newcommand{\ft}{\mathfrak{t}}
\newcommand{\fu}{\mathfrak{u}}

\newcommand{\bA}{\mathbb{A}}

\newcommand{\C}{\mathbb{C}}

\newcommand{\R}{\mathbb{R}}

\newcommand{\cA}{\mathcal{A}}
\newcommand{\cB}{\mathcal{B}}

\newcommand{\cG}{\mathcal{G}}
\newcommand{\cH}{\mathcal{H}}

\newcommand{\cK}{\mathcal{K}}

\newcommand{\cM}{\mathcal{M}}

\newcommand{\cO}{\mathcal{O}}

\newcommand{\cS}{\mathcal{S}}

\newcommand{\cU}{\mathcal{U}}
\newcommand{\cV}{\mathcal{V}}
\newcommand{\cW}{\mathcal{W}}

\newcommand{\cZ}{\mathcal{Z}}

\newcommand{\rd}{\mathrm{d}}

\newcommand{\bs}{\backslash}

\newcommand{\re}{\mathrm{Re}\,}

\newcommand{\tp}[1]{\prescript{t}{}{#1}}

 \DeclareMathOperator{\vol}{vol}
\DeclareMathOperator{\val}{val}

\providecommand{\abs}[1]{\lvert#1\rvert}
\providecommand{\aabs}[1]{\lVert#1\rVert}

\DeclareMathOperator{\Tr}{Tr}

\DeclareMathOperator{\GL}{GL}

\DeclareMathOperator{\GSp}{GSp}

\DeclareMathOperator{\GU}{GU}

\DeclareMathOperator{\Res}{Res}

\DeclareMathOperator{\Hom}{Hom}

\DeclareMathOperator{\id}{\mathbf{1}}

\DeclareMathOperator{\As}{As}

\DeclareMathOperator{\Ad}{Ad}

\DeclareMathOperator{\Sym}{Sym}

\DeclareMathOperator{\reg}{reg}

\DeclareMathOperator{\el}{ell}

\DeclareFontFamily{U}{mathx}{\hyphenchar\font45}
\DeclareFontShape{U}{mathx}{m}{n}{
      <5> <6> <7> <8> <9> <10>
      <10.95> <12> <14.4> <17.28> <20.74> <24.88>
      mathx10
      }{}
\DeclareSymbolFont{mathx}{U}{mathx}{m}{n}
\DeclareMathAccent{\widecheck}{\mathalpha}{mathx}{"71}

\makeatletter
\def\Ddots{\mathinner{\mkern1mu\raise\p@
\vbox{\kern7\p@\hbox{.}}\mkern2mu
\raise4\p@\hbox{.}\mkern2mu\raise7\p@\hbox{.}\mkern1mu}}
\makeatother

\theoremstyle{definition}
\newtheorem{definition}{Definition}[section]

\theoremstyle{plain}
\newtheorem{theorem}[definition]{Theorem}
\newtheorem{prop}[definition]{Proposition}
\newtheorem{lemma}[definition]{Lemma}

\newtheorem{conj}[definition]{Conjecture}

\theoremstyle{remark}
\newtheorem{remark}[definition]{Remark}

\numberwithin{equation}{section}

\newenvironment{altenumerate}
   {\begin{list}
      {(\theenumi) }
      {\usecounter{enumi}
       \setlength{\labelwidth}{0pt}
       \setlength{\labelsep}{0pt}
       \setlength{\leftmargin}{0pt}
       \setlength{\itemsep}{\the\smallskipamount}
       \renewcommand{\theenumi}{\roman{enumi}}
      }}
   {\end{list}}

\begin{document}

\title{Twisted linear periods and a new relative trace formula}

\author{Hang Xue}

\author{Wei Zhang}

\date{\today}

\address{Hang Xue\\
Department of Mathematics\\ The University of Arizona\\
Tucson, AZ 85721}

\email{xuehang@math.arizona.edu}

\address{Wei Zhang\\
Department of Mathematics\\
Massachusetts Institute of Technology\\
Cambridge, MA 02139 }

\email{weizhang@mit.edu}

\begin{abstract}
We study the linear periods on $\GL_{2n}$ twisted by a character using a
new relative trace formula. We establish the relative fundamental lemma and
the transfer of orbital integrals. Together with the spectral isolation
technique of Beuzart-Plessis--Liu--Zhang--Zhu, we are able to compare the
elliptic part of the relative trace formulae and to obtain new results
generalizing Waldspurger's theorem in the $n=1$ case.
\end{abstract}

\maketitle

\tableofcontents

\section{Introduction}

\subsection{Linear periods}
Let $F$ be a number field and $E$ a quadratic field extension of $F$, with
their rings of ad\`eles denoted by $\bA_F$ and $\bA_E$. Let $\eta:
\bA_F^\times/F^\times\to\{\pm 1\}$ be the quadratic character attached to
$E/F$ by class field theory. Let $\omega: \bA_F^\times/F^\times \to
\C^\times$ and $\chi: \bA_E^\times/E^\times \to \C^\times$ be two characters
with $\chi^n|_{\bA_F^\times} \omega = 1$. Let $\cA$ be a central simple
algebra over $F$ of dimension $4n^2$. Fix an embedding  $E\to \cA$ of
$F$-algebras, and let $\cB$ be the centralizer of $E$ in $\cA$. Let $G =
\cA^\times$ and $H = \cB^\times$, both regarded as algebraic groups over $F$.
Let $Z$ be the center of $G$. Let $\pi$ be an irreducible cuspidal
automorphic representation of $G(\bA_F)$ with central character $\omega$.
Take $\varphi \in \pi$ and define
    \begin{align}
    P_\chi(\varphi) = \int_{Z(\bA_F)H(F) \bs H(\bA_F)} \varphi(h) \chi(h) \rd h.
    \end{align}

We propose the following conjecture.

\begin{conj}    \label{conj:linear_periods}
\begin{altenumerate}
\item\label{item:conjt1} If $P_\chi$ is not identically zero, then
    $L(\frac{1}{2}, \pi_{0, E} \otimes \chi) \not=0$ and $L(s, \pi_0, \wedge^2
    \otimes \chi|_{\bA_F^\times})$ has a simple pole at $s= 1$ where $\pi_0$
    is the Jacquet--Langlands transfer of $\pi$ to $\GL_{2n}(\bA_F)$, and
    $\pi_{0, E}$ is the base change of $\pi_0$ to $\GL_{2n}(\bA_E)$.
    Moreover for all places $v$ of $F$, the Langlands parameter of $\pi_v$
    takes value in $\GSp_{2n}(\C)$ with similitude factor
    $\chi_v|_{F_v^\times}$, and $\epsilon(\pi_{0, E, v}
    \otimes \chi_v) \eta_v(-1)^n = (-1)^r$ if $v \in S$, where $r$ is the
    split rank of $G$ over $F_v$.
\item\label{item:conjt2}

If all archimedean places of $F$ split in $E$, or $n$ is odd, we also have a
converse. Let $\pi_0$ be an irreducible cuspidal automorphic representation
of $\GL_{2n}(\bA_F)$. Assume that $L(\frac{1}{2}, \pi_{0,E} \otimes \chi)
\not=0$ and $L(s, \pi_0, \wedge^2 \otimes \chi|_{\bA_F^\times})$ has a simple
pole at $s= 1$. Then there is a central simple algebra $\cA$ containing $E$
over $F$ and an irreducible cuspidal representation $\pi$ of $G(\bA_F)$, such
that $\pi$ is the Jacquet--Langlans transfer of $\pi_0$ and $P_{\chi}$ is not
identically zero on $\pi$. Moreover if $n$ is odd, we can take $\cA$ to be a
matrix algebra over a (possibly split) quaternion algebra.
\end{altenumerate}
\end{conj}

If $n =1$, then this is the celebrated result of Waldspurger~\cite{W}. If
$\chi$ is trivial, as $L(s, \pi_{0, E}) = L(s, \pi_0)L(s, \pi_0 \otimes
\eta)$, the conjecture reduces to the one proposed by Guo and Jacquet
in~\cite{Guo2}. The description of the local components $D_v$ and $\pi_v$,
regardless of $\chi$ being trivial or not, is (a consequence of) the
conjecture of Prasad and Takloo-Bighash~\cite{PTB}*{Conjecture~1}.

When $\chi$ is trivial, a relative trace formula approach was proposed
in~\cite{Guo2}, generalizing the work of~\cite{JacquetI}. The study of these
relative trace formulae yields both local and global results towards
Conjecture~\ref{conj:linear_periods} in the case $\chi$ being trivial,
cf.~\cites{FMW,Li1,Li2,Xue1}. However these relative trace formulae make
essential use of the fact that $L(s, \pi_{0, E})$ factors, so they cannot be
extended to the case of nontrivial $\chi$.

The goal of this paper is to propose a new relative trace formula towards
Conjecture~\ref{conj:linear_periods}, and to confirm many cases of the
conjecture by comparing the elliptic parts of the relative trace formula. We
do this by establishing the relevant fundamental lemma and transfer of
orbital integrals.

\begin{remark}
The assumption that $E/F$ splits at the archimedean places comes from the
fact that the only central simple algebras over $\R$ are matrix algebras over
quaternion algebras. In the case $n$ being even, they do not provide us with
enough orbits in our relative trace formula approach. A similar phenomenon
appears even when $\chi$ is trivial, cf.~\cite{Guo2}.
\end{remark}

\begin{remark}
The statement of Conjecture~\ref{conj:linear_periods} also makes sense when
$E = F \times F$ provided that the embedding $E\to \cA$ makes $\cA$ a free
$E$-module. If $\chi$ is trivial, some preliminary studies have been carried
out in~\cite{CZhang2}. The trace formula we proposed in this paper, with
obvious modifications, can be used to study the general case. We will return
to this in the future work.
\end{remark}

For simplicity we assume for the rest of this paper that $\cA = M_n(D)$ where
$D$ is a (possibly split) quaternion algebra over $F$ containing $E$. Thus $G
= \GL_n(D)$ and $H = \Res_{E/F} \GL_{n, E}$.

\begin{theorem} \label{thm:linear_periods}
Assume that $E/F$ is split at all archimedean places, $\pi_{0, E}$ is
cuspidal, and there is at least one place $v_1$ such that $\pi_{v_1}$ is {\em
elliptic}. Then part (\ref{item:conjt1}) of
Conjecture~\ref{conj:linear_periods} holds.
\end{theorem}

Here the ellipticity is relative to the subgroup $H$ and its precise meaning
can be found in Subsection~\ref{ss:sph char}. By Proposition
\ref{prop:elliptic_rep}, all supercuspidal representations are elliptic.

The local part of this theorem, i.e. the self-duality of $\pi$ and
determining the local root numbers, confirms the conjecture of Prasad and
Takloo-Bighash in many cases. The general case of the conjecture will be
treated in a subsequent paper, based on the results we obtain here.

In the converse direction, because of the lack of certain representation
theoretic results, our theorem is less general than the above one.

\begin{theorem} \label{thm:linear_periods_converse}
Assume that $n$ is odd, $E/F$ is split at all archimedean places and that
$\pi_0$ satisfy the conditions in part (\ref{item:conjt2}) of
Conjecture~\ref{conj:linear_periods}. Assume $\pi_{0,E}$ is cuspidal. Let
$\Sigma$ be a finite set of finite places of $F$ containing dyadic places,
such that if $v \not \in \Sigma$, then $E_v/F_v$, $\pi_{0, v}$ and $\chi_v$
are all unramified. Assume that
\begin{enumerate}
\item if $v \in \Sigma$, then either $v$ splits in $E$ or $\pi_{0, E_v}$ is
    supercuspidal.

\item  there is at least one place $v_1$ in $\Sigma$ such that $\pi_{0,
    E_{v_1}}$ is elliptic.
\end{enumerate}
Then there is a unique quaternion algebra $D$ over $F$ and an irreducible
cuspidal representation $\pi$ of $G(\bA_F)$, such that $\pi$ is the
Jacquet--Langlans transfer of $\pi_0$ and $P_{\chi}$ is not identically zero
on $\pi$, i.e. part (\ref{item:conjt2}) of
Conjecture~\ref{conj:linear_periods} holds.
\end{theorem}

Here the ellipticity is relative to subgroups of $\GL_{2n}(E_{v_1})$ and the
precise meaning will be explained in Subsection~\ref{subsec:involution}.

\subsection{The new relative trace formula}

Let $v$ be an archimedean place of $F$, we let $\cS(G(F_v))$ be the space of
Schwartz functions on $G(F_v)$, meaning the functions $f$ such that $Df$ is
bounded on $G(F_v)$ for all algebraic differential operators $D$ on $G(F_v)$.
Let $\cS(G(\bA_F))$ be the space of Schwartz functions on $G(\bA_F)$, i.e.
linear combinations of the functions of the form $\prod_v f_v$ where $f_v \in
C_c^\infty(G(F_v))$ if $v$ is nonarchimedean and $f_v \in \cS(G(F_v))$ is $v$
is archimedean. Similar definitions also applies to other groups.

To study the linear period $P_{\chi}$ we consider the following relative
trace formulae. Let $f \in \cS(G(\bA_F))$. We put
    \[
    K_f(g_1, g_2) = \int_{Z(F) \bs Z(\bA_F)} \sum_{y\in G(F)}
    f(z g_1^{-1} y g_2) \omega(z)^{-1} \rd z.
    \]
Define a distribution
    \begin{equation}    \label{eq:J_intro}
    J(f) = \iint_{(Z(\bA_F)H(F) \bs H(\bA_F))^2}
    K_f(h_1, h_2) \chi(h_1 h_2^{-1}) \rd h_1 \rd h_2.
    \end{equation}
This distribution at least formally unfolds geometrically and has a spectral
expansion. Then we obtain a relative trace formula on $G(\bA_F)$. This
relative trace formula is essentially the same as the one propose
in~\cite{Guo2}, except that a character $\chi$ is inserted.

We now propose a relative trace formula to study the $L$-function
$L(\frac{1}{2}, \pi_{0, E} \otimes \chi)$. This is the main innovation of
this paper.

Let us first recall the relative trace formula proposed by~\cite{Guo2} when
$\chi$ is trivial. In this case, let $f' \in \cS(\GL_{2n}(\bA_F))$ and the
relative trace formula results from the geometric and spectral expansion of
the distribution
    \[
    \iint
    \sum_{x \in \GL_{2n}(F)} f'(h_1^{-1} x h_2) \eta(h_1) \rd h_1 \rd h_2,
    \]
where the integration is over $h_1, h_2 \in \GL_n(\bA_F) \times
\GL_n(\bA_F)$. The spectral expansion gives both the periods
    \[
    \int \varphi(h) \rd h, \quad \int\varphi(h)\eta(h) \rd h,
    \]
where $\varphi \in \pi_{0}$ and the domain of the integration in both cases
are $\GL_n(\bA_F) \times \GL_n(\bA_F)$. Thus by the work of Friedberg and
Jacquet~\cite{FJ} on (split) linear periods these periods give rise to the
$L$-functions $L(s, \pi_0)L(s, \pi_{0} \otimes \eta)$.

It is clear that such an approach cannot be generalized to arbitrary $\chi$,
simply because $L(s, \pi_{0, E} \otimes \chi)$ does not factorize in general.
An alternative approach is needed. Assume the central character of $\pi_0$ is
$\omega$. In the case of $n = 1$, Jacquet~\cite{JacquetII} proposed the
following. Assume $n = 1$. Let $f' \in \cS(\GL_2(\bA_E))$ be a test function.
Then consider
    \[
    \iint \sum_{x \in \GL_{2}(E)} f'(h_1^{-1} x h_2) \chi(h_1)
    (\omega\eta)(\lambda(h_2)) \rd h_1 \rd h_2.
    \]
Here the integration is over $h_1 = \begin{pmatrix} a \\ & 1\end{pmatrix}$,
$a \in E^\times \bs \bA_E^\times$, $h_2$ is in $\GU(1, 1)$ where $\GU(1,1)$
stands for the quasisplit similitude unitary group in two variables, and
$\lambda$ is the similitude character. Jacquet's idea is as follows. The
integration over $h_1$ gives the central $L$-value, and the period over
$\GU(1, 1)$ ensures that the representations we consider on $\GL_{2}(\bA_E)$
are all base change from $\GL_2(\bA_F)$. Jacquet (re)proved
Conjecture~\ref{conj:linear_periods} in the case $n = 1$ based on this
relative trace formula.

Thus for general $n$ a natural idea is to extend the relative trace formulae
in~\cite{JacquetII}, i.e.
    \begin{equation}    \label{eq:wrong}
    \iint \sum_{x \in \GL_{2n}(E)} f'(h_1^{-1} x h_2) \widetilde{\chi}(h_1)
    (\omega\eta)(\lambda(h_2)) \rd h_1 \rd h_2,
    \end{equation}
with $h_1 = (h_{11}, h_{12}) \in \GL_n(\bA_E) \times \GL_n(\bA_E)$,
$\widetilde{\chi}(h_1) = \chi(h_{11} h_{12}^{-1})$, and $h_2 \in
\GU(n,n)(\bA_F)$. This is very natural and was indeed our first attempt. But
it does not seem to be the correct approach and we eventually abandoned it
for the following reason. The stabilizers in the geometric expansions of the
distributions~\eqref{eq:J_intro} and~\eqref{eq:wrong} are very different, one
being tori in $\GL_n(F)$ and the other being tori in the unitary groups. On
the philosophical level it is not expected that two trace formulae can be
compared unless the stabilizers from their geometric side are closely
related, e.g. they are isomorphic or at least one of them is trivial. After
all, in the comparison of the trace formulae, we need to equate the volume of
these stabilizers. Therefore we do not expect a nice comparison between the
geometric expansions of the distributions~\eqref{eq:J_intro}
and~\eqref{eq:wrong}.

We take an alternative approach in this paper. The starting point is the
following key observation. Let $\Pi$ be an irreducible cuspidal automorphic
representation of $\GL_{2n}(\bA_E)$ and $\varphi \in \Pi$. The integration of
$\varphi$ over $\GL_n(\bA_E) \times \GL_n(\bA_E)$ does not merely tell us
something about $L(\frac{1}{2}, \Pi \otimes \chi)$, but also about the
self-duality of $\Pi$. Let us introduction some notation. Let $H' =
\Res_{E/F}(\GL_n \times \GL_n)$, and $\chi_{H'}$ is a character of $H'$
sending $h_1 = (h_{11}, h_{12})$ to $\chi(h_{11} \overline{h_{12}})$. Indeed
if
    \[
    \int \varphi(h_1) \chi_{H'}(h_1) \rd h_1 \not=0
    \]
where the integration is over $H'$, by~\cite{FJ} we have $L(\frac{1}{2}, \Pi
\otimes \chi) \not=0$ and $L(s, \Pi, \wedge^2 \otimes \chi\chi^c)$ has a pole
at $s = 1$ where $\chi^c(g) = \chi(\overline{g})$. The later implies that
$\Pi^\vee \simeq \Pi \otimes \chi \chi^c$. What we need in the relative trace
formula is to use the second integral over $h_2$ to separate those $\Pi$ with
$\Pi \simeq \Pi^c$. Under the condition that $\Pi^\vee \simeq \Pi \otimes
\chi \chi^c$, this is equivalent to $\Pi^\vee \otimes \chi^{-1} \simeq \Pi^c
\otimes \chi^c$ where $\Pi^c(g) = \Pi(\overline{g})$, and this later
condition can be detected using the period integral of Flicker and Rallis,
i.e. the integration
    \[
    \int \varphi(h_2) \chi\eta(h_2) \rd h_2
    \]
where $h_2 \in \GL_{2n}(\bA_F)$. Thus our new distribution on
$\GL_{2n}(\bA_E)$ reads the following
    \[
    \iint \sum_{x \in \GL_{2n}(E)} f'(h_1^{-1} x h_2)\chi_{H'}(h_1)
    (\chi\eta)^{-1}(h_2) \rd h_1 \rd h_2,
    \]
where $h_1 \in \GL_n(\bA_E) \times \GL_{n}(\bA_E)$ and $h_2 \in
\GL_{2n}(\bA_F)$. The geometric and spectral expansions of this distribution
give the relative trace formula on $\GL_{2n}(\bA_E)$. The stabilizer of any
(relatively) regular semisimple orbit is a torus in $\GL_n(F)$ of the form
$\prod_{i}\Res_{F_i/F}  \GL_1$, and hence it matches the stabilizers arising
from the distribution~\eqref{eq:J_intro}.

The majority of this paper compares the elliptic part of this relative trace
formula with the one on $G(\bA_F)$. Our key local results  are  the relevant
fundamental lemma and transfer of orbital integrals. With the recent
technique from~\cite{BPLZZ} to isolate cuspidal spectra, these local results
lead to the main theorems. To remove the unnecessary conditions in those
theorems, one would need to compare the full relative trace formulae, not
just the elliptic part. Nevertheless the current comparison is sufficient for
the purpose of solving the local problems, i.e. the conjecture of Prasad and
Takloo-Bighash. The conjecture of Prasad and Takloo-Bighash in turn appears
to be an indispensable ingredient in the comparison of the full relative
trace formulae. We hope to address these  questions  in a future work.

\subsection{Notation and Convention}

Throughout this paper we keep the following notation and convention.

If $X$ is a set, we denote by $\id_X$ the characteristic function of it.

If $G$ is a group and $f$ is a function on $G$ then we put $f^\vee(g) =
f(g^{-1})$.

When a group $A$ acts on a set $X$ and $x \in X$ we always denote by $A_x$
the stabilizer of $x$ in $A$.

The $n \times n$ identity matrix is denoted by $1_n$, or simply $1$ when the
size of the matrix is clear.

If $F$ is a number field, we put $F_{\infty} = \prod_{v \mid \infty} F_v$.

Let $E/F$ is be quadratic field extension. The nontrivial Galois involution
is denoted by $\overline{\cdot}$. By a twisted conjugation by $g \in
\GL_n(E)$, we mean the map $x \mapsto g x \overline{g}^{-1}$. The stabilizer
of $x$ in $\GL_{n, E}$ under this twisted conjugation is denoted by $(\GL_{n,
E})_{x, \mathrm{twisted}}$. This is an algebraic group over $F$ and
    \[
    (\GL_{n, E})_{x, \mathrm{twisted}}(F) = \GL_n(E)_{x, \mathrm{twisted}} =
    \{ g \in \GL_n(E) \mid gx \overline{g}^{-1} =x \}.
    \]
We define $N: \GL_n(E) \to \GL_n(E)$ the norm map $Ng = g \overline{g}$. The
image of the norm map is denoted by $N \GL_n(E)$.

Let $D$ be a quaternion algebra over $F$ with a fixed embedding $E \to D$. We
fix an element $\epsilon \in NE^\times$ or $\epsilon \in F^\times \bs
NE^\times$ depending on whether $D$ split or ramifies. The group $\GL_n(D)$
is realized as a subgroup of $\GL_{2n}(E)$ consisting of elements of the form
    \[
    \begin{pmatrix} A & \epsilon B \\ \overline{B} & \overline{A} \end{pmatrix},
    \quad A, B \in M_n(E),
    \]

We let $\theta: \GL_{2n}(E) \to \GL_{2n}(E)$ be the involution
    \[
    g\mapsto \theta(g) = \begin{pmatrix} 1_n \\ &-1_n \end{pmatrix} g
    \begin{pmatrix} 1_n \\ &-1_n \end{pmatrix}.
    \]

\subsection{Acknowledgement}
HX is partially supported by the NSF grant DMS~\#1901862 and DMS \#2154352.
WZ is partially supported by the NSF grant DMS~\#1901642.

\section{Relative trace formulae: the geometric side}
\label{sec:geometric_side}

\subsection{Geometric side}
Let $E/F$ be a quadratic extension of number fields, and $D$ a (possibly
split) quaternion algebra over $F$ containing $E$. Let $G = \GL_n(D)$, $Z =
\GL_{1, F}$ the center of $G$, and $H = \Res_{E/F} \GL_{n, E}$.

Let us recall from the introduction that we have the following relative trace
formula on $G(\bA_F)$. Let $f \in \cS(G(\bA_F))$. We put
    \[
    K_f(g_1, g_2) = \int_{Z(F) \bs Z(\bA_F)} \sum_{y\in G(F)}
    f(z g_1^{-1} y g_2) \omega(z)^{-1} \rd z.
    \]
Define a distribution
    \[
    J(f) = \iint_{(Z(\bA_F)H(F) \bs H(\bA_F))^2}
    K_f(h_1, h_2) \chi(h_1 h_2^{-1}) \rd h_1 \rd h_2.
    \]

We consider the $H\times H$ action on $G$ by $(h_1,h_2)\cdot y=h_1^{-1} y
h_2$. An element $y \in G(F)$ is called regular semisimple if the stabilizer
$(H \times H)_y$ is a torus of dimension $n$ over $F$. It is called elliptic
if in addition that this torus is anisotropic modulo the center of $G$. Let
$G(F)_{\reg}$ and $G(F)_{\el}$ be the subsets of regular semisimple and
elliptic elements. These definitions also apply to elements in $G(F_v)$ where
$v$ is a place of $F$.

Assume that $f = \otimes f_v$ is decomposable and there is one nonsplit place
$v_1$ of $F$ such that $f_{v_1}$ is supported in the regular elliptic locus.
Then we have
    \begin{equation}    \label{eq:geometric_nonsplit}
    J(f) = \sum_{y \in H(F) \bs G(F)_{\el} /H(F)}
    \vol((H\times H)_y)  O^G(y, f),
    \end{equation}
where
    \[
    O^G(y, f) = \int_{(H \times H)_y(\bA_F) \bs (H \times H)(\bA_F)}
    f(h_1^{-1} y h_2) \chi(h_1^{-1} h_2)^{-1} \rd h_1 \rd h_2.
    \]
In these expressions we fix compatible measures on $Z(\bA_F) \bs (H \times
H)_y(\bA_F)$, $Z(\bA_F) \bs (H \times H)(\bA_F)$ and $(H \times H)_y(\bA_F)
\bs (H \times H)(\bA_F)$ for each $y \in G(F)_{\el}$ and $\vol((H\times
H)_y)$ stands for the volume of
    \[
    Z(\bA_F)(H\times H)_y(F) \bs (H \times H)(\bA_F).
    \]

This integral is absolutely convergent for all regular semisimple $y \in
G(F)$. Since the test function $f$ is not compactly supported, the absolute
convergence of~\eqref{eq:geometric_nonsplit} needs explanation. This will be
given in Appendix~\ref{s:A1}.

For any place $v$ of $F$, we define similarly the local orbital integrals,
except we integrate over $(H \times H)(F_v)$ instead.

The orbital integral can be simplified as follows. If $g \in \GL_{2n}$ we
define an involution
    \[
    \theta(g) = \begin{pmatrix} 1_n \\ &- 1_n \end{pmatrix} g
    \begin{pmatrix} 1_n \\ &- 1_n \end{pmatrix}.
    \]
Then $H$ is the group of fixed point of $\theta$. We introduce the symmetric
space
    \[
    S = \{ g \theta(g)^{-1} \mid g \in G\}
    \]
and then $H$ acts on $S$ by conjugation. Put
    \[
    \widetilde{f}(g \theta(g)^{-1}) = \int_{H(\bA_F)}
    f( gh) \chi (gh)^{-1} \rd h.
    \]
Then $\widetilde{f} \in \cS(S(\bA_F))$ and
    \[
    O^G(g, f) = O^S(s, \widetilde{f}) =
    \int_{H_s(\bA_F) \bs H(\bA_F)} \widetilde{f}(h^{-1} s h) \rd h.
    \]
For any place $v$ of $F$, the local orbital integrals can be simplified in a
similar way.

We introduce another relative trace formula with
which~\eqref{eq:geometric_nonsplit} will be compared. While the distribution
$J$ does not differ much from those in~\cite{Guo2}, this relative trace
formula is the main point of innovation of the present paper.

Let $G' = \Res_{E/F} \GL_{2n}$, $H' = \Res_{E/F} (\GL_n \times \GL_n)$
embedded in $G'$ as diagonal blocks, and $H'' = \GL_{2n, F}$. Let $Z' \simeq
\GL_{1, F}$ embedded in $G'$ diagonally. Let $f' \in \cS(G'(\bA_E))$ and
    \[
    K_{f'}(g_1, g_2) = \int_{Z'(E) \bs Z'(\bA_E)}
    \sum_{\gamma \in G'(F)} f'(g_1^{-1} z \gamma g_2)
    \omega(z \overline{z})^{-1} \rd z.
    \]

For $h = \begin{pmatrix} h_1 \\ & h_2 \end{pmatrix}\in H'(\bA_F)$ with
$h_1,h_2\in ( \Res_{E/F} \GL_n)(\bA_F)$, we put $\chi_{H'}(h) = \chi(h_1
\overline{h_2})$. Consider the distribution
    \[
    \begin{aligned}
    I(f') = \int_{Z'(\bA_E)H'(F) \bs H'(\bA_E)}
    \int_{Z'(\bA_F) H''(F) \bs H''(\bA_F)}
    K_{f'}\left( h, g \right)\chi_{H'}(h)
    (\chi\eta)^{-1}(g) \rd h \rd g.
    \end{aligned}
    \]
The motivation for introducing this distribution will be clear when we
discuss its spectral expansion in Subsection~\ref{subsec:functoriality}.

We say that an element $x \in G'(F)$ is regular semisimple if the stabilizer
$(H' \times H'')_x$ is a torus of dimension $n$ over $F$. It is elliptic if
in addition $(H' \times H'')_x$ is an elliptic torus. Let $G'(F)_{\el}$ be
the subset of elliptic elements. These definitions also apply to elements in
$G'(F_v)$ where $v$ is a place of $F$.

Assume that $f' = \otimes f'_v$ is decomposable and there is one nonsplit place $v_1$ of $F$ such
that $f_{v_1}'$ is supported in the regular elliptic locus. We fix a
character $\widetilde{\eta}: E^\times \bs \bA_E^\times \to \C^\times$ such
that $\widetilde{\eta}|_{\bA_F^\times} = \eta$. Then as usual the
distribution $I$ unfolds to orbital integrals, i.e. we have
    \begin{equation}    \label{eq:geometric_split}
    I(f') = \sum_{x \in H'(F) \bs G'(F)_{\el} / H''(F)}
    \vol((H'\times H'')_x) O^{G'}(x, f'),
    \end{equation}
where
    \[
    O^{G'}(x, f') = \int_{(H' \times H'')_x (\bA_F) \bs (H' \times H'')(\bA_F)}
    f'(h^{-1} x h'') (\chi_{H'}\chi^{-1} \widetilde{\eta}^{-1})(h)
    (\chi \widetilde{\eta})^{-1}(h^{-1} x h'') \rd h \rd h''.
    \]
We fix compatible measures on $Z'(\bA_F) \bs (H' \times H'')_x(\bA_F)$,
$Z'(\bA_F) \bs (H' \times H'')(\bA_F)$ and $(H' \times H'')_x(\bA_F) \bs (H'
\times H'')(\bA_F)$ for each $x \in G'(F)_{\el}$ and $\vol((H'\times H'')_x)$
stands for the volume of
    \[
    Z'(\bA_F)(H'\times H'')_x(F) \bs (H' \times H'')(\bA_F).
    \]

This integral is absolutely convergent for all regular semisimple $x$. Since
the test function $f'$ is not compactly supported, the absolute convergence
of~\eqref{eq:geometric_split} needs explanation. This will be given in the
Appendix~\ref{s:A1}.

If $v$ is a place of $F$, then we define similarly the local orbital
integral, except we integrate over $(H' \times H'')(F_v)$ instead.

The orbital integral can be simplified as follows. Introduce the symmetric
space
    \[
    S' = \{s\overline{ s}=1 \mid s\in G'\}\simeq G'/H'',
    \]
on which $H'$ acts by twisted conjugation. Put
    \begin{equation}    \label{eq:f'tilde}
    \widetilde{f'}(g \overline{g}^{-1}) =
    \int_{H''(\bA_F)} f'(gh) (\chi\widetilde{\eta})^{-1}(gh) \rd h.
    \end{equation}
Then $\widetilde{f'} \in \cS(S'(\bA_F))$. We have
    \begin{equation}    \label{eq:simplify_orbital_integral}
    O^{G'}(g, f') = O^{S'}(s', \widetilde{f'}) =
    \int_{H'_{s'}(\bA_F) \bs H'(\bA_F)}
    \widetilde{f'}(h^{-1} s' \overline{h})
    (\chi_{H'}\chi^{-1} \widetilde{\eta}^{-1})(h) \rd h, \quad
    s' = g \overline{g}^{-1}.
    \end{equation}
If $v$ is a place of $F$, the local orbital integral can be defined and
simplified in a similar way.

An element $s' \in S'(F)$ is regular semisimple if $H'_{s'}$ is a torus of
dimension $n$. It is in addition elliptic if $H'_{s'}$ is an anisotropic
torus modulo the split center of $G'$. If $s' = g\overline{g}^{-1}$, $g \in
G'(F)$, then $s'$ is regular semisimple or elliptic if $g$ is so in $G'(F)$.
Let $S'(F)_{\mathrm{reg}}$ and $S'(F)_{\el}$ be the subsets of regular
semisimple and elliptic elements respectively.

\subsection{Matching of test functions} \label{subsec:matching}
Recall that we fixed $\epsilon  \in NE^\times$ (resp. $F^\times \bs
NE^\times$) if $D$ splits (resp. ramifies) and the group $G$ is realized as a
subgroup of $\GL_{2n}(E)$ consisting of matrices of the form
    \[
    \begin{pmatrix} \alpha & \epsilon  \beta \\
    \overline{\beta} & \overline{\alpha} \end{pmatrix}, \quad
    \alpha, \beta \in M_n(E).
    \]
Then $H$ consists of matrices of the form $\begin{pmatrix} \alpha \\
& \overline{\alpha} \end{pmatrix}$, $\alpha \in \GL_n(E)$.

Let $x \in G'(F)$ and $y \in G(F)$ be regular semisimple elements. Write
    \[
    x \overline{x}^{-1} = \begin{pmatrix} \alpha_1 & \alpha_2 \\
    \alpha_3&\alpha_4 \end{pmatrix} \in S'(F),
    \quad
    y \theta(y)^{-1} = \begin{pmatrix} \beta_1 &\beta_2 \\
    \beta_3 & \beta_4 \end{pmatrix} \in S(F),
    \]
where $\alpha_i, \beta_i \in M_n(E)$. We will prove that $\alpha_i$'s and
$\beta_i$'s are all invertible. We say that $x$ and $y$ match if $2 \alpha_1
\overline{\alpha_1} - 1$ and $\beta_1$ have the same characteristic
polynomial. Matching of regular semisimple elements will be studied in detail
in Section~\ref{sec:transfer}, and it turns out that not all regular
semisimple $x \in G'(F)$ matches a regular semisimple $y \in G(F)$, and vice
versa. This is a new feature of the present relative trace formula at hand.
The element $x$ matches some $y$ (resp. $y$ matches some $x$) if
    \[
    1 - (\alpha_1 \overline{\alpha_1})^{-1} \in \epsilon N\GL_n(E), \quad
    \text{resp.}\quad \frac{1}{2}(\beta_1+1) \in N \GL_n(E).
    \]
The matching of regular semisimple elements also applies to the situation of
$F_v$ where $v$ is a place of $F$. We note that there is a neighbourhood of
$1 \in G(F_v)$ such that every regular semisimple $y$ in this neighbourhood
matches some $x \in G'(F_v)$. This is because there is a small neighbourhood
of $1 \in \GL_n(E_v)$ such that every element in this neighbourhood is a
norm.

If $x \in G'(F_v)$ and $y \in G(F_v)$ match, we will see in
Subsection~\ref{subsec:matching_def} that the stabilizers $(H' \times
H'')_x(F_v)$ and $(H \times H)_y(F_v)$ are isomorphic and we fix such an
isomorphism. We fix measures on these stabilizers which are identified under
the fixed isomorphism. If $x \in G'(F)$ and $y \in G(F)$ match, then they
match at every place of $F$. With the above fixed measures on their
stabilizers we have
    \[
    \vol((H' \times H'')_x) = \vol((H \times H)_y).
    \]

Fix an element $\tau \in E^\times$ such that $\overline{\tau} = -\tau$. For
each place $v$ of $F$, we define  transfer factors on $G'$ and $G$:
    \begin{align}\label{eq:tran fact}
    \kappa_v^{G'}(x) = \chi_v(\alpha_4) \widetilde{\eta_v}(\tau \alpha_2), \quad
    \kappa_v^G(y) = \chi_v(y_1), \quad y^{-1} =
    \begin{pmatrix} y_1 & \epsilon y_2 \\
    \overline{y_2} & \overline{y_1} \end{pmatrix}.
    \end{align}

For each place $v$ of $F$, we define
    \[
    \cS(G'(F_v))_0 = \{ f' \in \cS(G'(F_v)) \mid O^{G'}(x, f_v') = 0 \text{
    for all $x$ not matching any $y \in G(F_v)$}\},
    \]
and $\cS(G(F_v))_0$ in a similar way. By definition, two test functions $f'_v \in
\cS(G'(F_v))_0$ and $f_v \in \cS(G(F_v))_0$ match if and only if
    \[
    \kappa_v^{G'}(x) O^{G'}(x, f'_v) = \kappa_v^G(y) O^{G}(y, f_v)
    \]
for all matching regular semisimple $x \in G'(F_v)$ and $y \in G(F_v)$. Two
test functions $f' = \otimes f'_v \in \cS(G'(\bA_F))$ and $f = \otimes f_v
\in \cS(G(\bA_F))$ match if $f'_v \in \cS(G'(F_v))_0$ and $f_v \in
\cS(G(F_v))_0$ and they match for all place $v$ of $F$.

The following are the main theorems concerning the geometric side of the
trace formulae. They will be proved in Sections~\ref{sec:transfer}
and~\ref{sec:FL} respectively.

\begin{theorem} \label{thm:transfer_group}
Assume that $v$ is nonarchimedean and nonsplit. For any $f'_v \in
\cS(G'(F_v))_0$ there is an $f_v \in \cS(G(F_v))_0$ that matches it, and vice
versa.
\end{theorem}

\begin{theorem} \label{thm:FL_group}
Let $v$ be a nonsplit nonarchimedean odd place in $F$. Assume the quaternion
algebra $D$ splits at $v$ and $\chi_v$ is unramified at $v$. Let $\fo_{F_v}$
be the ring of integers of $F_v$. We pick the measure on $G'(F_v)$ and
$G(F_v)$ such that the volumes of $G'(\fo_v)$ and $G(\fo_v)$ are $1$. Then
$\id_{G'(\fo_{F_v})}$ and $\id_{G(\fo_{F_v})}$ match.
\end{theorem}

\subsection{Matching at split places}
The goal of this subsection is to explain the matching of test functions at
the split places, which can be made explicit.

If $v$ is a split place of $F$. Then $G(F_v) = \GL_{2n}(F_v)$ and $G'(F_v) =
\GL_{2n}(F_v) \times \GL_{2n}(F_v)$. We fix a measure on $\GL_{2n}(F_v)$ and
thus we have measures on $G'(F_v)$ and on $G(F_v)$ under this identification.
The character $\eta_v$ is trivial so $\widetilde{\eta}_v$ takes the form
$(\eta_0, \eta_0^{-1})$ where $\eta_0$ is a character of $F_v^\times$. The
character $\chi_v$ is of the form $(\chi_1, \chi_2)$ where $\chi_1, \chi_2$
are characters of $F_v^\times$. Regular semisimple elements $y \in G(F_v)$
and $(x_1, x_2) \in G'(F_v)$ matches if $y = x_1 x_2^{-1}$. Let $f' = (f_1',
f_2') \in \cS(G'(F_v)) \simeq \cS(G(F_v)) \otimes \cS(G(F_v))$ and put
    \begin{equation}    \label{eq:chi_convolution_split}
    f(g) = \int_{\GL_{2n}(F_v)}
    f_1'(g h) f_2'(h) \chi_1^{-1}(h) \chi_2^{-1}(h) \rd g,
    \quad g \in \GL_{2n}(F_v).
    \end{equation}

\begin{lemma}   \label{lemma:matching_split}
The functions $f'$ and $f$ match.
\end{lemma}

\begin{proof}
We write
    \[
    x_1 x_2^{-1} = \begin{pmatrix} A_1 & B_1\\ C_1 & D_1 \end{pmatrix}, \quad
    x_2 x_1^{-1} = \begin{pmatrix} A_2 & B_2\\ C_2 & D_2 \end{pmatrix}.
    \]
A little computation gives that the orbital integral $\kappa^{G'}((x_1, x_2))
O^{G'(F_v)}((x_1, x_2), f')$ equals
    \begin{equation}    \label{eq:orbital_integral_explicit}
    \begin{aligned}
    &\eta_0(x_1 x_2^{-1})^{-1} \eta_0(-B_1 B_2^{-1})
    \chi_1(D_1)\chi_2(D_2)\chi_1(x_1 x_2^{-1})^{-1}\\
    &\int f\left(\begin{pmatrix} a_1 \\ & b_1 \end{pmatrix}^{-1}
    x_1 x_2^{-1} \begin{pmatrix} a_2 \\ & b_2 \end{pmatrix} \right)
    \chi_1(a_1 a_2^{-1} )
    \chi_2(b_1 b_2^{-1})
    \rd a_1 \rd a_2 \rd b_1 \rd b_2.
    \end{aligned}
    \end{equation}
When $(x_1, x_2)$ is regular semisimple, $A_2$ is invertible, and thus
    \[
    \begin{pmatrix} A_2 & B_2\\ C_2 & D_2 \end{pmatrix} =
    \begin{pmatrix} 1 \\ C_2A_2^{-1} & 1 \end{pmatrix}
    \begin{pmatrix} A_2 &B_2 \\ & D_2 - C_2 A_2^{-1} B_2 \end{pmatrix}.
    \]
Thus $D_1 = (D_2 - C_2 A_2^{-1} B_2)^{-1}$ and $(\det x_1 x_2^{-1})^{-1} =
\det A_2 \det (D_2 - C_2 A_2^{-1} B_2)$. It follows that
$\chi_1(D_1)\chi_2(D_2)\chi_1(x_1 x_2^{-1})^{-1} = \chi_1(A_2) \chi_2(D_2)$.
Similarly we have $\eta_0(x_1 x_2^{-1})= \eta_0(-B_1 B_2^{-1})$. In
particular we note that $\kappa^{G'}((x_1, x_2)) O^{G'(F_v)}((x_1, x_2), f')$
is independent of the choice of $\widetilde{\eta}$. A straightforward
computation then gives that~\eqref{eq:orbital_integral_explicit} equals
    \[
    \kappa^G(y) O^G(y, f),
    \]
when $y = x_1 x_2^{-1}$.
\end{proof}

When we speak of the matching of test functions $f'$ and $f$ at split places,
we always mean that $f'$ and $f$ are related in this explicit way.

We observe that there is an involution on $\cS(G(F_v))$ given by $f \mapsto
f^\vee \chi_1\chi_2$. Define
    \[
    \cS(G(F_v))^+ = \{ f \in \cS(G(F_v)) \mid f^\vee \chi_1 \chi_2 = f\}, \quad
    \cS(G'(F_v))^+ = \cS(G(F_v))^+ \otimes \cS(G(F_v))^+.
    \]
There is a base change homomorphism
    \[
    \mathrm{bc}_v: \cS(G'(F_v)) \simeq  \cS(G(F_v)) \otimes \cS(G(F_v))
    \to
    \cS(G(F_v)),
    \]
which is given by the usual multiplication (i.e. convolution) on
$\cS(G(F_v))$. It maps $\cS(G'(F_v))^+$ to $\cS(G(F_v))^+$,
cf.~\cite{AC}*{Chapter~1, Section~5}.

\begin{lemma}   \label{lemma:+space}
If $f' \in \cS(G'(F_v))^+$ then $f'$ and $\mathrm{bc}_v(f')$ match.
\end{lemma}

\begin{proof}
This is a direct consequence of Lemma~\ref{lemma:matching_split} and the
definition of $\cS(G'(F_v))^+$.
\end{proof}

Let us assume further that $v$ is an odd unramified (split) finite place, and
$\chi_1, \chi_2$ are unramified characters of $F_v^\times$. Let $K_v =
\GL_{2n}(\fo_{F_v})$ and $\cH_v = \C[K_v \bs G(F_v)/K_v]$ be the spherical
Hecke algebra of $G(F_v)$. Then the involution $f \mapsto f^\vee\chi_1\chi_2$
reduces to an involution on $\cH_v$. Let $\cH_v^{\pm}$ be the $\pm
1$-eigenspaces of this involution. Similarly let $\cH_{E, v}$ be the
spherical Hecke algebra of $G'(F_v)$, which is identified with  $\cH_v
\otimes \cH_v$.

\begin{lemma}   \label{lem:van}
If
    \[
    f'_v = (f'_1, f'_2) \in
    (\cH_v^- \otimes \cH_v^+) \oplus (\cH_v^+ \otimes \cH_v^-) \oplus
    (\cH_v^- \otimes \cH_v^-),
    \]
then we have
    \[
    O^{G'(F_v)}((x_1, x_2), f') =0, \quad
    O^{G(F_v)}(y, \mathrm{bc}_v(f')) = 0.
    \]
\end{lemma}

\begin{proof}
From the explicit expressions~\eqref{eq:orbital_integral_explicit}, it
suffices to check that if $f \in \cH_v^-$, then
    \begin{equation}    \label{eq:vanishing_orbital_integral}
    \int_{\GL_n(F_v) \times \GL_n(F_v)}
    f\left( \begin{pmatrix} a \\ & b \end{pmatrix}g \right)
    \chi_1(a)^{-1}\chi_2(b)^{-1} \rd a \rd b = 0
    \end{equation}
for all $g \in G(F_v)$. By~\cite{Offen}*{Proposition~3.1}, every double coset
$H(F_v) \bs G(F_v)/K_v$ is represented by an element $g$ such that
    \[
    g^{-1} \theta(g) = \begin{pmatrix}
    & \Lambda \\ - \Lambda^{-1} \end{pmatrix}, \quad
    \Lambda =
    \begin{pmatrix} && \varpi^{\lambda_1} \\ & \Ddots \\
    \varpi^{\lambda_n} \end{pmatrix}, \quad
    \lambda_1 \geq \cdots \geq \lambda_n \geq 0.
    \]
where $\varpi$ is a uniformizer in $F_v$. Therefore to
see~\eqref{eq:vanishing_orbital_integral} it is enough to assume that $g$ is
of this form. We may further assume that $g$ take the shape $g =
\begin{pmatrix} A & B \\ C & D \end{pmatrix}$,
and $A, D$ are diagonal while $B, C$ consist of only anti-diagonal entries,
and $\det \begin{pmatrix} a & b \\ c & d \end{pmatrix} = 1$, where $a, b, c,
d$ are entries of $g$ at $(i, i)$, $(i, n-i+1)$, $(n-i+1, i)$ and $(n-i+1,
n-i+1)$ respectively.
With this $g$ it is straightforward to see that
    \begin{equation}    \label{eq:g}
    \begin{pmatrix} & 1 \\ -1
    \end{pmatrix}\tp{g}^{-1} = g \begin{pmatrix} & 1 \\ -1
    \end{pmatrix}.
    \end{equation}
Note that since $f \in \cH_v$ we have $f(\tp{g}) = f(g)$ for all $g$. Now
replacing $f$ by $-f^\vee \chi_1\chi_2$
in~\eqref{eq:vanishing_orbital_integral}, we see that
    \[
    \begin{aligned}
    \eqref{eq:vanishing_orbital_integral} &= - \int
    f\left( \tp{\begin{pmatrix} a \\ & b \end{pmatrix}^{-1}}
    \tp{g}^{-1} \right) \chi_2(a) \chi_1(b) \rd a \rd b\\
    &=
    - \int
    f\left( \tp{\begin{pmatrix} a \\ & b \end{pmatrix}^{-1}}
    \begin{pmatrix} & -1 \\ 1 \end{pmatrix}
    g \begin{pmatrix} & 1 \\ -1 \end{pmatrix} \right)
    \chi_2(a) \chi_1(b) \rd a \rd b.
    \end{aligned}
    \]
Make change of variables $a \mapsto \tp{b}^{-1}$ and $b \mapsto \tp{a}^{-1}$,
and use the fact that $f \in \cH_v$ we see that the last integral equals
$-\eqref{eq:vanishing_orbital_integral}$, which
implies~\eqref{eq:vanishing_orbital_integral} equals zero.
\end{proof}

\begin{prop}    \label{prop:FL_split}
The function $f' \in \cH_{E, v}$ and $f = \mathrm{bc}_v(f') \in \cH_v$ match.
\end{prop}

\begin{proof}
By Lemma~\ref{lem:van} above, the assertion holds if $f'_v = (f'_1, f'_2) \in
(\cH_v^- \otimes \cH_v^+) \oplus (\cH_v^+ \otimes \cH_v^-) \oplus (\cH_v^-
\otimes \cH_v^-)$. If $f'_v \in \cH_v^+ \otimes \cH_v^+$ then the map
$\mathrm{bc}_v$ and the map~\eqref{eq:chi_convolution_split} coincide.
\end{proof}

\section{Relative trace formulae: the spectral side}
\subsection{Linear periods and Shalika models}  \label{subsec:FJ}
We recall some results on (split) linear periods and Shalika models on
$\GL_{2n}(\bA_E)$ from~\cites{FJ,JS}. In this subsection we temporarily use
the following notation. Let $Q$ be the Shalika subgroup of $\GL_{2n}(E)$
consisting of matrices of the form
    \[
    \begin{pmatrix} g \\ & g \end{pmatrix}
    \begin{pmatrix} 1 & x \\ & 1 \end{pmatrix}, \quad
    g \in \GL_n(E), \quad x \in M_n(E).
    \]
Let $\psi: E \bs \bA_E \to \C^\times$ be a nontrivial additive character and
$\chi, \xi: E^\times \bs \bA_E^\times \to \C^\times$ be two characters.
Define a character $\theta: Q(E) \bs Q(\bA_E) \to \C^\times$ by
    \[
    \theta \left( \begin{pmatrix} g \\ & g \end{pmatrix}
    \begin{pmatrix} 1 & x \\ & 1 \end{pmatrix} \right) =
    \psi(\Tr x) \xi(\det g).
    \]
Let $\Pi$ be an irreducible cuspidal automorphic representation of
$\GL_{2n}(\bA_E)$ with central character $\omega$, and assume that $\xi^n
\omega = 1$.

We define global Shalika functional on $\Pi$
    \[
    \lambda(\varphi) = \int_{\bA_E^\times Q(E) \bs Q(\bA_E)}
    \varphi(s) \theta(s) \rd s, \quad \varphi \in \Pi.
    \]
By~\cite{JS}*{Section~8, Theorem~1}, the Shalika functional on $\Pi$ is not
identically zero if and only if $L(s, \Pi, \wedge^2 \otimes \xi)$ has a
simple pole at $s = 1$. If this is the case, for $\varphi \in \Pi$ we put
    \[
    V_{\varphi}(g) = \lambda(\Pi(g) \varphi), \quad g \in \GL_{2n}(\bA_E).
    \]

Following~\cite{FJ} we consider the integral
    \[
    Z^{\mathrm{FJ}}(s, \phi, \chi, \xi) =
    \int
    \varphi \left( \begin{pmatrix} h_1 \\ & h_2 \end{pmatrix} \right)
    \abs{\det h_1 h_2^{-1}}^{s-\frac{1}{2}}
    \chi(\det h_1 h_2^{-1}) \xi(\det h_2) \rd h_1 \rd h_2,
    \]
Here the integration is over
    \[
    (h_1, h_2) \in \bA_E^\times (\GL_n(E) \times \GL_n(E)) \bs
    (\GL_n(\bA_E) \times \GL_n(\bA_E))
    \]
and $\bA_E^\times$ stands for the center of $\GL_{2n}(\bA_E)$. The integral
is convergent for all $s \in \C$. By~\cite{FJ}*{Proposition~2.2} this
integral is not identically zero only if the Shalika function on $\Pi$ is not
identically zero. Assume that this is the case. Then
by~\cite{FJ}*{Proposition~2.3} when $\re s$ is sufficiently large this
integral unfolds to
    \[
    \int_{\GL_n(\bA_E)} V_{\varphi}
    \left( \begin{pmatrix} g \\ & 1 \end{pmatrix} \right)
    \chi(\det g) \abs{\det g}^{s-\frac{1}{2}} \rd g.
    \]
By~\cite{FJ}*{Theorem~4.1} this integral equals a holomorphic multiple of
$L(s, \Pi \otimes \chi)$, and there is a choice of $\varphi$ such that it
equals this $L$-function. Specializing to the point $s = \frac{1}{2}$, we
obtain the following proposition.

\begin{prop}    \label{prop:FJ}
We have $Z^{\mathrm{FJ}}(\frac{1}{2}, \varphi, \chi, \xi) \not=0$ for some
$\varphi \in \Pi$ if and only if $L(s, \Pi, \wedge^2 \otimes \xi)$ has a
simple pole at $s = 1$ and $L(\frac{1}{2}, \Pi \otimes \chi) \not=0$.
\end{prop}

Let us now switch to the local situation. Let $v$ be a place of $E$. Define
the character $\theta_v: Q(E_v) \to \C^\times$ the same way as the global
case. By~\cite{CS}, the space
    \[
    \Hom_{Q(E_v)} (\Pi_v \otimes \theta_v, \C)
    \]
is at most one dimensional. Assume that it is one dimensional and $\lambda_v$
is a nonzero element. For any $\varphi_v \in \Pi_v$ put
    \[
    V_{\varphi_v}(g) = \lambda(\Pi_v(g) \varphi_v), \quad g \in \GL_{2n}(E_v),
    \]
and
    \[
    \cV_v = \{ V_{\varphi_v} \mid \varphi_v \in \Pi_v\}.
    \]
The space $\cV_v$ is called the local Shalika model of $\Pi_v$. As
in~\cite{FJ}*{Section~3} we define
    \[
    Z_v^{\mathrm{FJ}}(s, V_v, \chi_v) =
    \int_{\GL_n(F_v)}
    V_v \left( \begin{pmatrix} g \\ & 1 \end{pmatrix} \right)
    \chi_v(g) \abs{\det g}^{s-\frac{1}{2}} \rd g.
    \]
Here we note that this implicitly depends on $\psi_v$ and $\xi_v$ through the
Shalika model. Then by~\cite{FJ}*{Proposition~3.1}, it is a holomorphic
multiple of $L(s, \Pi_v \otimes \chi_v)$ and there is a $V_v \in \cV_v$ such
that it equals this local $L$-function. Moreover
by~\cite{FJ}*{Proposition~3.3} there is a functional equation
    \[
    \gamma(s, \Pi_v \otimes \chi_v, \psi_v)
    Z_v^{\mathrm{FJ}}(s, V_v, \chi_v) =
    Z_v^{\mathrm{FJ}}(1-s, \widetilde{V_v}, \chi_v^{-1}),
    \]
where
    \[
    \widetilde{V_v}(g) = V_v \left( \begin{pmatrix} & 1 \\ -1 \end{pmatrix}
    \tp{g}^{-1} \right),
    \]
and $\widetilde{V_v}$ belong to the Shalika model of the contragredient of
$\pi$ (defined using the characters $\chi_v^{-1}$ and $\psi_v^{-1}$).
Specializing to $s = \frac{1}{2}$ we conclude that
    \begin{equation}    \label{eq:FE_local_FJ}
    \epsilon(\Pi_v \otimes \chi_v, \psi_v)
    Z_v^{\mathrm{FJ}}(\frac{1}{2}, V_v, \chi_v) =
    Z_v^{\mathrm{FJ}}(\frac{1}{2}, \widetilde{V_v}, \chi_v^{-1}),
    \end{equation}
where $\epsilon(\Pi_v \otimes \chi_v, \psi_v)$ is the local root number of
$\Pi_v \otimes \chi_v$.

\subsection{Functoriality}  \label{subsec:functoriality}
We return to the setup and notation of Section~\ref{sec:geometric_side}. Let
$\Pi$ be an irreducible cuspidal automorphic representation of $G'(\bA_F)$.
We define $\Pi^\vee$ to be the dual of $\Pi$, and $\Pi^c$ to be the Galois
conjugate (relative to $E/F$) of $\Pi$, i.e. the automorphic representation
whose space is given by $\{ \varphi(\overline{g}) \mid \varphi \in \Pi\}$.
Then it is well-known that $\Pi^c \simeq \Pi$ if and only if $\Pi = \pi_E$
for some cuspidal automorphic representation $\pi$ of $\GL_{2n}(\bA_F)$,
cf.~\cite{AC}*{Chapter~3, Theorem~4.2 and~5.1}.

We apply the results recalled the previous subsection to the current
situation. Put
    \begin{equation}    \label{eq:split_linear}
    P'_{\chi}(\varphi) = \int_{Z'(\bA_F) H'(F) \bs H'(\bA_F)}
    \varphi\left( \begin{pmatrix} h_1 \\ & h_2 \end{pmatrix} \right)
    \chi(h_1 \overline{h_2}) \rd h_1 \rd h_2
    \end{equation}
Then by Proposition~\ref{prop:FJ}, applied to the case $\xi = \chi \chi^c$,
the linear form $P'_{\chi}$ is not identically zero if and only if
$L(\frac{1}{2}, \Pi \otimes \chi) \not=0$ and $L(s, \Pi, \wedge^2 \otimes
\chi\chi^c)$ has a pole at $s = 1$. The later implies that $\Pi^\vee \simeq
\Pi \otimes \chi \chi^c$.

We also consider the Flicker--Rallis period of $\Pi$ given by
    \[
    P''_{\chi\eta}(\varphi) =
    \int_{Z_{2n}(\bA_F) H''(F) \bs H''(\bA_F)} \varphi(h) (\chi\eta)(h) \rd h.
    \]
It is identically zero if and only if the Asai $L$-function $L(s, \Pi \otimes
\chi, \As^-)$ has a pole at $s = 1$, cf.~\cite{Kable}. Note that this implies
that $\Pi^\vee \otimes \chi^{-1} \simeq \Pi^c \otimes \chi^{c}$.

\begin{lemma}   \label{lemma:functoriality}
If neither of $P'_{\chi}$ and $P''_{\chi\eta}$ is identically zero, then
there is an irreducible cuspidal automorphic representation $\pi$ of
$\GL_{2n}(\bA_F)$ such that $\Pi = \pi_E$. Moreover, $L(\frac{1}{2}, \Pi
\otimes \chi) \not=0$, and $L(s, \pi, \wedge^2 \otimes \chi|_{\bA_F^\times})$
has a simple pole at $s = 1$.
\end{lemma}

\begin{proof}
The existence of $\pi$ follows from the above discussion. To prove the second
assertion, let us observe that
    \[
    L(s, \Pi \otimes \chi, \As^-) =
    L(s, \pi, \Sym^2 \otimes \chi|_{\bA_F^\times}\eta)
    L(s, \pi, \wedge^2 \otimes \chi|_{\bA_F^\times}).
    \]
Indeed this can be checked place by place for almost all places of $F$. If
$v$ is split in $E$, then the equality clearly holds. Assume $v$ is inert and
$\Pi_v$ (and hence $\pi_v$) is unramified. Let $\beta_1, \cdots, \beta_{2n}$
be the Satake parameters of $\pi_v$, and $\gamma = \chi_v(\varpi_v)$ where
$\varpi_v$ is the uniformizer at $v$. Then the Satake parameters of $\Pi_v$
are $\beta_1^2, \cdots, \beta_{2n}^2$. Thus the left hand side equals
    \[
    \prod_{1 \leq i \leq 2n} (1+ \beta_i^2 \gamma q_v^{-s})^{-1}
    \prod_{1 \leq i < j \leq 2n} (1- \beta_i^2 \beta_j^2 \gamma^2 q^{-2s})^{-1}.
    \]
The right hand side equals
    \[
    \prod_{1 \leq i \leq j \leq 2n} (1+ \beta_i \beta_j \gamma q_v^{-s})^{-1}
    \prod_{1 \leq i < j \leq 2n} (1- \beta_i \beta_j \gamma^2 q^{-s})^{-1}.
    \]
Thus the desired equality holds at the place $v$.

Similar we also have
    \[
    L(s, \Pi, \wedge^2 \otimes \chi \chi^c) =
    L(s, \pi, \wedge^2 \otimes \chi|_{\bA_F^\times})
    L(s, \pi, \wedge^2 \otimes \chi|_{\bA_F^\times} \eta).
    \]

If $L(s, \pi, \wedge^2 \otimes \chi|_{\bA_F^\times})$ is holomorphic at $s =
1$, then both
    \[
    L(s, \pi, \Sym^2 \otimes \chi|_{\bA_F^\times}\eta), \quad
    L(s, \pi, \wedge^2 \otimes \chi|_{\bA_F^\times} \eta)
    \]
have a pole at $s = 1$, which implies that
    \[
    L(s, \pi \times \pi \chi|_{\bA_F^\times}\eta)
    \]
has at least a double pole at $s = 1$, which is not possible. This proves the
final assertion.
\end{proof}

\subsection{Local spherical charaters}  \label{ss:sph char}
Let $v$ be a nonarchimedean nonsplit place. Let $\pi_v$ be an irreducible
unitary representation of $G(F_v)$. According to~\cite{Lu}, the space
$\Hom_{H(F_v)}(\pi_v \otimes \chi_v, \C)$ is at most one dimensional. We
assume that it is one dimensional and we fix a nonzero element $\ell$ in it.
Let $f \in \cS(G(F_v))$ and we consider the spherical character
    \[
    J_{\pi_v}(f) = \sum_{\phi} \ell(\pi_v(f) \phi) \overline{\ell(\phi)}
    \]
where the sum runs over an orthonormal basis of $\pi$. A test functions of
the form $f = f_1 * \overline{f_1^\vee}$ is called of positive type.

\begin{lemma}   \label{lemma:transferrable_test_function}
There is a positive type $f \in \cS(G(F_v))_0$ such that $J_{\pi_v}(f) > 0$.
\end{lemma}

\begin{proof}
Choose $\phi \in \pi_v$ such that $\ell(\phi) \not=0$. Let $f_1 \in
\cS(G(F_v))$ be the characteristic function of a small open compact subgroup,
then $\ell(\pi(f_1) \phi) \not=0$. Put $f = f_1 * \overline{f_1^\vee}$. Then
$f$ is supported in a sufficiently small neighbourhood of $1 \in G(F_v)$ and
thus $f \in \cS(G(F_v))_0$, because every regular semisimple element in the
support of $f$ matches some element in $G'(F_v)$. Moreover we have
    \[
    J_{\pi_v}(f) = \sum_\phi \ell(\pi_v(f_1) \phi) \overline{\ell(\pi_v(f_1) \phi)}.
    \]
Each term in the sum is nonnegative and there is at least one nonzero term.
Therefore $J_{\pi_v}(f) >0$.
\end{proof}

By~\cite{Guo3}, $J_{\pi_v}$ is represented by a locally integrable function
$\Theta_{\pi_v}$ on $G(F_v)$, which is locally constant on the regular
semisimple locus, and which satisfies the invariance property
$\Theta_{\pi_v}(h_1 g h_2) = \chi_v(h_1 h_2) \Theta_{\pi}(g)$. We say that
$\pi_v$ is $H(F_v)$-elliptic $\Theta_{\pi_v}(y) \not=0$ for some elliptic
regular semisimple $y \in G(F_v)$ and $y$ matches some $x \in G'(F_v)$.

\begin{prop}    \label{prop:elliptic_rep}
If $\pi_v$ is supercuspidal and $\Hom_{H(F_v)}(\pi_v \otimes \chi_v, \C)
\not=0$, then $\pi_v$ is $H(F_v)$-elliptic.
\end{prop}

The argument is standard and very close to the classical theory of
Harish-Chandra. The main part of the proof has also been worked out in a
slightly different setting in~\cite{Xue3}. A detailed proof is quite long,
and it produces a lot of repetitions from the literature and deviates
significantly from the goal of this paper. So we will merely sketch the
argument in Appendix~\ref{s:A2}.

\subsection{Involution} \label{subsec:involution}

We need to consider the local spherical characters arising from the
distribution $I_{\Pi}$. Let us fix a nonsplit nonarchimedean place $v$. Then
we have
    \[
    \Hom_{H'(F_v)}(\Pi_v \otimes \chi_{H',v}, \C) \not=0, \quad
    \Hom_{H''(F_v)}(\Pi_v \otimes \chi_v,
    \C) \not=0.
    \]
We let $\ell'$ and $\ell''$ be nonzero elements in these $\Hom$-spaces
respectively. Define the local spherical character
    \[
    I_{\Pi_v}(f') = \sum_{W} \ell'(\Pi_v(f')W) \overline{\ell''(W)}, \quad
    f' \in \cS(G'(F_v)),
    \]
where $W$ ranges over an orthonormal basis of $\Pi_v$.

We say that $I_{\Pi_v}$ is elliptic if there is an $f'$ supported in the
elliptic locus of $G'(F_v)$ such that $I_{\Pi_v}(f') \not=0$. We expect that
all supercuspidal representations are elliptic. We will address it in a
subsequent paper.

Let $f' \in \cS(G'(F_v))$. Put
    \begin{equation}    \label{eq:involution}
    f'^{\dag}(g) = f'(\tp{\overline{g}}^{-1}) (\chi_v\chi_v^c)(g), \quad
    g \in G'(F_v).
    \end{equation}

Let $\psi_v$ be a nontrivial additive character of $E_v$ trivial on $F_v$.
Since $\Pi_v^\vee \otimes \chi_v^{-1} \simeq \Pi_v^c \otimes \chi_v^{c}$, the
local root number $\epsilon(\Pi_v \otimes \chi_v, \psi_v) = \pm 1$ and is
independent of $\psi_v$. We denote it by $\epsilon(\Pi_v \otimes \chi_v)$.

\begin{lemma}   \label{lemma:involution}
For any $f' \in \cS(G'(F_v))$ we have
    \[
    I_{\Pi_v}(f'^\dag) = \epsilon(\Pi_v \otimes \chi_v) I_{\Pi_v}(f').
    \]
\end{lemma}

\begin{proof}
To prove this lemma, we need to choose $\ell'$ and $\ell''$ more carefully.
Since different choices differ only by a constant, the validity of the lemma
is independent of such choices.

Let $\cW$ be the Whittaker model of $\Pi_v$ (defined by the character
$\psi_v$) and $\cV$ be the Shalika model of $\Pi_v$ defined by the characters
$\psi_v$ and $\chi_v \chi_v^c$. We fix an isomorphism
    \[
    \cW \to \cV, \quad W \mapsto \phi_W.
    \]
We also denote by $\widetilde{\cW}$ and $\widetilde{\cV}$ the Whittaker model
and Shalika model of $\pi^\vee$, defined using the characters $\psi_v^{-1}$
and $\psi_v^{-1}$, $(\chi_v\chi_v^c)^{-1}$ respectively. We also fix an
isomorphism
    \[
    \widetilde{\cW} \to \widetilde{\cV}, \quad W \mapsto \phi_{W}.
    \]
For any function $\alpha$ on $G'(F_v)$ we temporarily put
    \[
    \widetilde{\alpha}(g) = \alpha \left( \begin{pmatrix} & 1 \\ -1 \end{pmatrix}
    \tp{g}^{-1} \right).
    \]
Then both
    \[
    W \mapsto \phi_W, \quad W \mapsto \widetilde{\phi_{\widetilde{W}}}
    \]
are isomorphisms between $\cW$ and $\cV$, and hence they differ by a
constant. By rescaling the isomorphisms that we have fixed, we may assume
that this constant equals one. It follows that for any $W \in \cW$,  we have
    \[
    \phi_{\widetilde{W}} = \widetilde{\phi_W}.
    \]
For any $W \in \cW$ we put $W^c(g) = W(\overline{g})$.

Recall from Subsection~\ref{subsec:FJ} that we have the integral
    \[
    Z^{\mathrm{FJ}}(s, \phi, \chi_v) =
    \int_{\GL_n(E)} \phi
    \left( \begin{pmatrix} a \\ & 1 \end{pmatrix} \right) \chi_v(a)
    \abs{\det a}^{s-\frac{1}{2}} \rd a, \quad \phi \in \cV.
    \]
We choose $\ell' \in \Hom_{H'(F_v)} (\Pi_v \otimes \chi_v, \C)$ to be the
    \[
    W \mapsto \ell'(W) =
    Z^{\mathrm{FJ}}(\frac{1}{2}, \phi_W, \chi_v).
    \]
For the linear form $\ell''$, we take it to be
    \[
    \ell''(W) = \int_{N' \cap H''(F_v) \bs P' \cap H''(F_v)}
    W(h) (\chi_v\eta_v)(h) \rd h.
    \]
where $P'$ is the mirabolic subgroup of $G'$ and $N'$ the standard upper
triangular unipotent subgroup of $G'$.

If $W \in \cW$, since $\Pi_v^\vee \simeq \Pi_v \otimes (\chi_v
\chi_v^c)^{-1}$, we have $\widetilde{W^c}(\chi_v \chi_v^c)^{-1} \in \cW$. We
claim that for any $g \in G'(F_v)$ we have
    \begin{equation}    \label{eq:FE_l'}
    \epsilon(\Pi_v \otimes \chi_v)
    \ell'(\Pi_v(\tp{\overline{g}}^{-1}) W) (\chi_v\chi_v^c)^{-1}(g) =
    \ell'(\Pi_v(g) (\widetilde{W}^c(\chi_v\chi_v^c)^{-1})).
    \end{equation}
The right hand side equals
    \[
    \int \phi_{\Pi_v(g) (\widetilde{W}^c(\chi_v\chi_v^c)^{-1})}
    \left( \begin{pmatrix} a \\ & 1_n \end{pmatrix} \right) \chi_v(a) \rd a
    \]
which simplifies to
    \[
    (\chi_v\chi_v^c)^{-1}(g) \int
    \phi_{\widetilde{\Pi_v(\tp{\overline{g}}^{-1})W}}
    \left( \begin{pmatrix} a \\ & 1_n \end{pmatrix}  \right)
    \chi_v^{-1}(a) \rd a
    \]
Using the functional equation~\eqref{eq:FE_local_FJ} of $Z^{\mathrm{FJ}}$ we
conclude that this equals
    \[
    \epsilon(\Pi_v \otimes \chi_v)
    (\chi_v\chi_v^c)^{-1}(g) \int
    \phi_{\Pi_v(\tp{\overline{g}}^{-1})W}
    \left( \begin{pmatrix} a \\ & 1_n \end{pmatrix}  \right)
    \chi_v(a) \rd a,
    \]
which is precisely the left hand side of~\eqref{eq:FE_l'}.

We now compute $I_{\Pi_v}(f'^\dag)$. By definition we have
    \[
    \begin{aligned}
    \ell'(\Pi_v(f'^\dag) W) &=
    \int_{G'(F)} f'(\tp{\overline{g}}^{-1}) \ell'(\Pi_v(g)W)
    (\chi_v\chi_v^c)(g) \rd g\\
    &= \int_{G'(F)} f'(g) \ell'(\Pi_v(\tp{\overline{g}}^{-1})W)
    (\chi_v\chi_v^c)(g)^{-1} \rd g\\
    & = \epsilon(\Pi_v \otimes \chi_v)
    \int_{G'(F)} f'(g)
    \ell'(\Pi_v(g) (\widetilde{W}^c(\chi_v\chi_v^c)^{-1})).
    \end{aligned}
    \]
Thus we have
    \[
    \ell'(\Pi_v(f'^\dag)W) =
    \ell'(\Pi_v(f') (\widetilde{W}^c (\chi_v\chi_v^c)^{-1})).
    \]

By~\cite{LM4}*{Lemma~1.1} (the main part of the proof is
from~\cite{Offen}*{Corollary~7.2}) we have
    \[
    \ell''(\widetilde{W}^c (\chi_v\chi_v^c)^{-1}) =
    \ell''(W).
    \]
Thus we conclude
    \[
    I_{\Pi_v}(f'^\dag) = \epsilon(\Pi_v \otimes \chi_v) \sum_W
    \ell'(\Pi_v(f') (\widetilde{W}^c (\chi_v\chi_v^c)^{-1}))
    \overline{\ell''(\widetilde{W}^c (\chi_v\chi_v^c)^{-1})}.
    \]
When $W$ ranges over an orthonormal basis in $\cW$, $\widetilde{W}^c
(\chi_v\chi_v^c)^{-1}$ ranges over an orthonormal basis in $\cW$ as well.
Thus
    \[
    I_{\Pi_v}(f'^\dag) = \epsilon(\Pi_v \otimes \chi_v)
    I_{\Pi_v}(f').
    \]
This proves the lemma.
\end{proof}

We now study the relation between matching of test functions and this
involution. Let $\epsilon_{D_v} = \eta_v(\epsilon) = \pm 1$.

\begin{lemma}   \label{lemma:matching_involution}
Let $f' \in \cS(G'(F_v))_0$ and $f \in \cS(G(F_v))_0$ be matching test
functions. Then $f'^\dag$ and $\eta_v(-1)^n \epsilon_{D_v}^n f$ match.
\end{lemma}

\begin{proof}
Let $x \in G'(F)$ and $s' = x \overline{x}^{-1}$. Then $O^{G'}(x, f') =
O^{S'}(s', \widetilde{f'})$ where $\widetilde{f'}$ is defined
by~\eqref{eq:f'tilde}. Moreover a little computation shows that
    \[
    O^{G'}(x, f'^\dag) = O^{S'}(\tp{s'}, \widetilde{f'}).
    \]
We will see below in Section~\ref{sec:orbit_S'} that $s'$ is in the
$H'$-orbit of the form
    \[
    \begin{pmatrix} \alpha & 1 \\
    1- \alpha \overline{\alpha} & - \overline{\alpha} \end{pmatrix}.
    \]
So we may assume that $s'$ equals it. Then
    \[
    \tp{s'} =
    \begin{pmatrix} \alpha & 1 - \alpha \overline{\alpha} \\
    1 & - \overline{\alpha} \end{pmatrix} =
    \begin{pmatrix} 1 - \alpha \overline{\alpha} \\ & 1 \end{pmatrix}
    \begin{pmatrix} \alpha & 1 \\ 1 - \alpha \overline{\alpha}
    & - \overline{\alpha} \end{pmatrix}
    \begin{pmatrix} (1- \alpha \overline{\alpha})^{-1}\\ & 1 \end{pmatrix}.
    \]
It follows that
    \[
    O^{S'}(\tp{s'}, \widetilde{f'}) =
    \eta_v( 1- \alpha \overline{\alpha})
    O^{S'}(s', \widetilde{f'}) =
    \eta_v( 1- \alpha \overline{\alpha}) O^{G'}(x, f').
    \]
This in particular implies that $f'^\dag \in \cS(G'(F_v))_0$.

Suppose $x$ matches some $y \in G(F_v)$. This is equivalent to $1 - (\alpha
\overline{\alpha})^{-1} \in \epsilon N \GL_n(E_v)$, which implies $\eta_v(1 -
\alpha \overline{\alpha}) = \eta_v(-1)^n \epsilon_{D_v}^n$. Then
    \[
    O^{G'}(x, f'^\dag) = O^{G'}(x, \eta_v(-1)^n \epsilon_{D_v}^n f')
    \]
and the lemma follows.
\end{proof}

\subsection{Multipliers}    \label{subsec:multiplier}
To analyze the spectral side without truncation, we need to make use of the
techniques from~\cite{BPLZZ}. Recall that for any (complex) algebra $\cA$, a
multiplier is linear map $\mu\star: \cA \to \cA$ that commutes with both the
left and right multiplication in $\cA$. The space of multipliers of $\cA$ is
denoted by $Mul(\cA)$.

Let $v$ be an archimedean place of $F$, $\ft_v$ the (complexified) Cartan
subalgebra of $G(F_{v})$, $\ft_v^*$ the dual space, and $\cZ_{G(F_v)} \simeq
\C[\ft_v]^{W_v}$ the center of the universal enveloping algebra of $G(F_v)$.
Write $\chi_v = (\chi_1, \chi_2)$ then $\chi_1\chi_2$ as a character of
$F_v^\times$ defines an element $\underline{a}_v = (a_v, \cdots, a_v) \in
\ft_v^*$.

In~\cite{BPLZZ}*{Definition~2.8(3)}, an algebra $\cM_v$ of holomorphic
functions on $\ft_v^*$ is introduced (the notation in~\cite{BPLZZ} is
$\cM_{\theta}^{\sharp}(\fh_{\C}^*)^{\mathsf{W}}$). We do not need the precise
definition of this space, but only the following property,
cf.~\cite{BPLZZ}*{Theorem~2.14}. There is an algebra homomorphism
    \[
    \cM_v \to Mul(\cS(G(F_v))), \quad \mu \mapsto \mu\star
    \]
such that if $\sigma$ is an irreducible representation of $G(F_v)$, and $f
\in \cS(G(F_v))$ then
    \[
    \sigma(\mu \star f) = \mu(\lambda_{\sigma}) \sigma(f)
    \]
where $\lambda_{\sigma}$ is the infinitesimal character of $\sigma$. Let
$\iota_v$ be the involution on $\cM_v$ such that $\iota_v(\mu)(z) = \mu(-
\underline{a}_v - z)$ for all $z \in \ft_v^*$ and $\cM_v^+$ be the subspace
consisting of elements invariant under this involution. Recall that we have
an involution $f \mapsto f^\vee \chi_1 \chi_2$ on $\cS(G(F_v))$ and the space
$\cS(G(F_v))^+$ consisting of functions invariant under this involution. If
$f \in \cS(G(F_v))^+$ and $\iota_v(\mu) = \mu$, then $\mu\star f \in
\cS(G(F_v))^+$, i.e.
    \begin{equation}    \label{eq:+space}
    (\mu\star f)^\vee \chi_1\chi_2 = \mu\star f.
    \end{equation}
To see this, we only need to check that the action of both sides on any
irreducible representation $\sigma$ of $G(F_v)$ coincide. The left hand side
equals
    \[
    (\sigma \otimes \chi_1 \chi_2) (\mu\star f)^\vee =
    \mu( - \lambda_{\sigma \otimes \chi_1 \chi_2})
    (\sigma \otimes \chi_1 \chi_2)^\vee (f)
    = \mu(-\underline{a}_v - \lambda_{\sigma})
    (\sigma \otimes \chi_1 \chi_2)^\vee (f).
    \]
Since $\iota_v(\mu) = \mu$ we have $\mu(-\underline{a}_v - \lambda_{\sigma})
= \mu(\lambda_{\sigma})$. Moreover $f^\vee \chi_1\chi_2 = f$ implies $(\sigma
\otimes \chi_1 \chi_2)^\vee (f)  = \sigma(f)$. This proves~\eqref{eq:+space}.

Put $\cZ_G = \prod_{v \mid \infty} \cZ_{G(F_v)}$, $\lambda = \otimes_{v \mid
\infty} \lambda_v$ the character of $\cZ_G$, $\cM = \prod_{v \mid \infty}
\cM_v$ and $\cM^+ = \prod_{v \mid \infty} \cM_v^+$.

Let $\mathtt{S}$ be a finite set of finite places of $F$ such that if $v
\not\in \mathtt{S}$ then $E_v/F_v$, $\pi_v$ and $\chi_v$ are all unramified.
Let $\mathtt{T}_0$ be the set of nonsplit finite places of $F$ and
$\mathtt{T} = \mathtt{T}_0 \cup \mathtt{S}$. Let
    \[
    \cH_G^{\mathtt{T}} =  \bigotimes_{v \not\in \mathtt{T}}
    \cH_{v} =
     \bigotimes_{v \not\in \mathtt{T}}
    C_c^\infty (G(\fo_{F_v}) \bs G(F_v) /G(\fo_{F_v}))
    \]
be the spherical Hecke algebra away from $\mathtt{T}$. We fix an open compact
subgroup $K = \prod_{v \nmid \infty} K_v$ such that if $v \not\in \mathtt{S}$
then $K_v = G(\fo_{F_v})$.

All the above objects for $G$ have their counterparts for $G'$. For each
archimedean place $v$ we have an algebra of holomorphic function $\cM'_v$
which is identified with $\cM_v \otimes \cM_v$ and we put $\cM'^+_v = \cM_v^+
\otimes \cM_v^+$. Moreover put $\cM' = \prod_{v \mid \infty} \cM_v$ and
$\cM'^+ = \prod_{v \mid \infty} \cM_v^+$. We have $\cS(G'(F_v))^+ =
\cS(G(F_v))^+ \otimes \cS(G(F_v))^+$, and we put $\cS(G'(F_\infty))^+ =
\prod_{v \mid \infty} \cS(G'(F_v))^+$. We also have the center of the
universal enveloping algebra $\cZ_{G'}$ which is identified with $\cZ_{G}
\otimes \cZ_G$, and the spherical Hecke algbra away from $\mathtt{T}$
    \[
    \cH_{G'}^{\mathtt{T}} = \bigotimes_{v \not\in \mathtt{T}}
    \cH_{G', v} \simeq
    \cH^{\mathtt{T}} \otimes \cH^{\mathtt{T}}.
    \]
There is a base change homomorphism
    \[
    \mathrm{bc}: \cZ_{G'} \otimes \cH_{G'}^{\mathtt{T}}
    \to \cZ_G \otimes \cH^{\mathtt{T}}
    \]
which is given by the usual multiplication in $\cZ_G$ and $\cH^{\mathtt{T}}$.
We fix an open compact subgroup $K' = \prod_{v \nmid \infty} K'_v$ such that
if $v \not\in \mathtt{S}$ then $K_v' = G'(\fo_{F_v})$.

Let $\lambda = (\lambda_\infty, \lambda^{\infty, \mathtt{T}})$ be the
character of $\cZ_G \otimes \cH_G^{\mathtt{T}}$ attached to $\pi$, and
$L^2_0(G(F) \bs G(\bA_F)/K, \omega)[\lambda]$ be the maximal quotient of
$L^2_0(G(F) \bs G(\bA_F)/K, \omega)$ on which $\cZ_G \otimes
\cH_G^{\mathtt{T}}$ acts by $\lambda$. Then $\lambda' = \lambda \circ
\mathrm{bc} =  (\lambda, \lambda)$ is the character of $\cZ_{G'} \otimes
\cH_{G'}^{\mathtt{T}}$ attached $\pi_E$, cf.~\cite{AC}*{Chapter~1,
Section~5}. We let $L^2_0(G'(F) \bs G'(\bA_F)/K', \omega)[\lambda']$ be the
maximal quotient of $L^2_0(G'(F) \bs G'(\bA_F)/K', \omega)$ on which
$\cZ_{G'} \otimes \cH_{G'}^{\mathtt{T}}$ acts by $\lambda'$. We have
    \[
    L^2_0(G(F) \bs G(\bA_F), \omega)[\lambda] =
    \pi \oplus (\pi \otimes \eta), \quad
    L^2_0(G'(F) \bs G'(\bA_F), \omega')[\lambda'] = \pi_E.
    \]
The second equality is a direct consequence of~\cite{Ram}. The first needs a
little explanation. Indeed let $\sigma$ be an irreducible component of
$L^2_0(G(F) \bs G(\bA_F), \omega)[\lambda]$. Then the base change $\pi_E$ and
$\sigma_E$ to $\GL_{2n}(\bA_E)$ are isobaric automorphic representations and
they agree on almost all places of $E$ of degree one over $F$. Therefore
$\pi_E = \sigma_E$ by~\cite{Ram} and in particular $\sigma_E$ is cuspidal.
Let $\sigma_0$ be the Jacquet--Langlands transfer $\sigma$ to
$\GL_{2n}(\bA_F)$. Then by~\cite{AC}*{Chapter~3, Theorem~4.2(d)}, either
$\pi_0 = \sigma_0$ or $\pi_0 = \sigma_0 \otimes \eta$. Since
Jacquet--Langlands transfer is an injective map, we conclude that either $\pi
= \sigma$ or $\pi = \sigma \otimes \eta$. Our claim then follows from the
fact that $L^2_0(G(F) \bs G(\bA_F), \omega)$ is of multiplicity one.

Let
    \[
    \mathrm{bc}: \cM' \otimes \cH_{G'}^{\mathtt{T}} = (\cM \otimes
    \cH_{G}^{\mathtt{T}}) \otimes (\cM \otimes \cH_{G}^{\mathtt{T}})
    \to \cM \otimes \cH_{G}^{\mathtt{T}}
    \]
be the usual multiplication map.

\begin{prop}    \label{prop:multiplier}
There are elements $\mu' \in \cM'^+ \otimes \cH_{G'}^{\mathtt{T}}$ and $\mu =
\mathrm{bc}(\mu') \in \cM^+ \otimes \cH_G^{\mathtt{T}}$ such that for all $f'
\in \cS(G'(\bA_F))$ and $f \in \cS(G(\bA_F))$ we have
\begin{itemize}
\item $R(\mu' \star f')$ maps $L^2(G'(F) \bs G'(\bA_F)/K', \omega')$ into
    $\pi_E$,

\item $\mu'(\lambda') = \mu'(\lambda, \lambda) = 1$, which is equivalent to
    $\pi_E(\mu' \star f') = \pi_E(f')$ for all $f' \in
    \cS(G'(\bA_F))_{K'}$,
\end{itemize}
and
\begin{itemize}
\item $R(\mu \star f)$ maps $L^2(G(F) \bs G(\bA_F)/K, \omega)$ into $\pi
    \oplus (\pi \otimes \eta)$,

\item $\pi(\mu \star f) = \pi(f)$ and $(\pi \otimes \eta)(\mu \star f) =
    (\pi \otimes \eta)(f)$ for all $f \in \cS(G(\bA_F))_{K}$.
\end{itemize}
\end{prop}

\begin{proof}
Since $G' = \Res_{E/F}\GL_{2n}$, there is no CAP automorphic representation
of $G'(\bA_F)$ in the sense of~\cite{BPLZZ}*{Definition~3.4}.
By~\cite{BPLZZ}*{Theorem~3.6} there is an element $\mu'' \in \cM' \otimes
\cH_{G'}^{\mathtt{T}}$ satisfying the condition required for $\mu'$ in the
proposition.

As $\cM' = \prod_{w \mid \infty} \cM_w$, for $\mu \in \cM' \otimes
\cH_{G'}^{\mathtt{T}}$ we have the element $\iota_w(\mu)$ for each
archimedean place $w$ of $E$, by applying the involution $\iota_w$ only to
the place $w$. Put
    \[
    \mu' = \prod_{w \mid \infty} \mu'' \iota_w (\mu'') \in
    \cM'^+ \otimes \cH_{G'}^{\mathtt{T}}.
    \]
We claim that $\mu'$ again satisfies conditions in the proposition. Indeed,
the first condition is clear since $\mu''$ already maps $L^2(G'(F) \bs
G'(\bA_F)/K', \omega')$ into $\pi_E$, and moreover we have
    \[
    \mu'(\lambda) = \prod_{w \mid \infty} \mu''(\lambda')
    \iota_w(\mu'')(\lambda')
    = 1,
    \]
since $\pi_{E_w}$ satisfies $\pi_{E_w}^\vee = \pi_{E_w} \otimes \chi_{1}
\chi_{2}$ if $w$ is above the place $v$ of $F$ and we write $\chi_v =
(\chi_1, \chi_2)$, which implies $\iota_w(\mu'')(\lambda') = \mu''(\lambda')
= 1$.

Put $\mu = \mathrm{bc}(\mu')$. We claim that $\mu$ satisfies the conditions
in the proposition. This is equivalent to $\mu(\lambda) = 1$ and if $\nu =
(\nu_\infty, \nu^{\infty, \mathtt{T}})$ arises from an irreducible component
$\sigma$ of $L^2(G(F) \bs G(\bA_F)/K, \omega)$ and $\mu(\nu) \not=0$, then
$\sigma$ is cuspidal and $\nu = \lambda$. By the definition of $\mu$ we have
$\mu(\nu) = \mu'(\nu , \nu)$ where $(\nu, \nu) = \nu \circ \mathrm{bc}$ is
the character of the algebra $\cZ_{G'} \otimes \cH_{G'}^{\mathtt{T}}$
obtained by pulling back the character $\nu$, which is a character arising
from the automorphic representation $\sigma_E$. Then $\mu'(\nu \circ
\mathrm{bc}) \not=0$ means that $\sigma_E$ is cuspidal and $\nu \circ
\mathrm{bc} = \lambda'$. Moreover $\mu(\lambda) = \mu'(\lambda') = 1$.
\end{proof}

\section{Proofs of the main theorems}

We are now ready to prove Theorems~\ref{thm:linear_periods}
and~\ref{thm:linear_periods_converse}, assuming
Theorem~\ref{thm:transfer_group} and Theorem~\ref{thm:FL_group}.

\subsection{Proof of Theorem~\ref{thm:linear_periods}}
We keep the notation from the theorem and from
Subsection~\ref{subsec:multiplier}. First we define the global spherical
characters as follows. Let $\pi$ be an irreducible cuspidal automorphic
representation of $G(\bA_F)$ and $f \in \cS(G(\bA_F))$. Put
    \[
    J_{\pi}(f) = \sum_{\varphi} P_{\chi}(\pi(f) \varphi)
    \overline{P_{\chi}(\varphi)},
    \]
where $\varphi$ runs through an orthonormal basis of $\pi$. Let $\Pi$ be an
irreducible cuspidal automorphic representation of $G'(\bA_F)$ and $f' \in
\cS(G'(\bA_F))$. Put
    \[
    I_{\Pi}(f') = \sum_{\varphi} P'_{\chi}(\Pi(f') \varphi)
    \overline{P''_{\chi\eta}(\varphi)},
    \]
where again $\varphi$ runs through an orthonormal basis of $\Pi$.

Recall that $\mathtt{S}$ is a finite set of finite places of $F$ such that if
$v \not\in \mathtt{S}$ then $E_v/F_v$, $\pi_v$ and $\chi_v$ are all
unramified, $\mathtt{T}_0$ is the set of nonsplit finite places of $F$ and
$\mathtt{T} = \mathtt{T}_0 \cup \mathtt{S}$. We require the finite set
$\mathtt{S}$ to contain the place $v_1$. Let us now take test function $f =
\otimes f_v \in \cS(G(\bA_F))$ and $f' = \otimes f_v' \in \cS(G'(\bA_F))$ as
follows. If $v \not\in \mathtt{S}$ and is finite, we take $f_v =
\id_{G(\fo_{F_v})}$ and $f_v' = \id_{G'(\fo_{F_v})}$. If $v \in \mathtt{S}$
and $v \not= v_1$, or $v$ is an archimedean place, we take a positive type
test function $f_v \in \cS(G(F_v))_0$ supported sufficiently close to $1$,
such that $J_{\pi_v}(f_v) \not=0$. The existence of such a test function is
given by Lemma~\ref{lemma:transferrable_test_function} if $v$ is not split,
and is obvious if $v$ is split. We choose $f'_v \in \cS(G'(F_v))$ such that
$f'_v$ and $f_v$ match in the sense of Lemma~\ref{lemma:matching_split}. If
$v$ is archimedean, since $\pi_v^\vee \otimes \chi_1 \chi_2 \simeq \pi_v$ if
$\chi_v = (\chi_1, \chi_2)$, we may further assume that $f'_v \in
\cS(G'(F_v))^+$ and hence $f_v \in \cS(G(F_v))^+$. With this additional
condition, $f_v = \mathrm{bc}_v(f'_v)$. If $v = v_1$, we may assume that
$f_{v_1}$ is supported in the elliptic locus, and this is possible by
Proposition~\ref{prop:elliptic_rep}.

We have constructed multipliers $\mu' \in \cM'^{+} \otimes
\cH_{G'}^{\mathtt{T}}$ and $\mu \in \cM^+ \otimes \cH_G^{\mathtt{T}}$ in
Proposition~\ref{prop:multiplier}. The key point is that these multipliers
are in the ``plus'' subspaces. We now use the test functions $\mu' \star f'$
and $\mu \star f$. We claim that they still match. To see this we write $f' =
f'^{\mathtt{T}} \otimes f'_{\mathtt{T}}$ where
    \[
    f'^{\mathtt{T}} \in \cS(G'(F_\infty))^+
    \otimes \cH_{G'}^{\mathtt{T}}, \quad
    f_{\mathtt{T}} \in
    \cS(G(F_\infty))^+ \otimes \cH_{G}^{\mathtt{T}},
    \]
and $f = f^{\mathtt{T}} \otimes f_{\mathtt{T}}$ similarly. The function $\mu'
\star f'$ (resp. $\mu \star f$) is obtained from $f'$ (resp. $f$) by
modifying only finite places not in $\mathtt{T}$ and the archimedean places.
Thus $f'_{\mathtt{T}}$ and $f_{\mathtt{T}}$ match (i.e. each local components
match). Moreover we have
    \[
    \mathrm{bc}(\mu' \star f'^{\mathtt{T}}) =
    \mathrm{bc}(\mu') \star \mathrm{bc}(f'^{\mathtt{T}}) =
    \mu \star f^{\mathtt{T}}.
    \]
The last equality follows from the definition of $\mu$ and $f^{\mathtt{T}}$.
It follows from Lemma~\ref{lemma:+space} and Proposition~\ref{prop:FL_split}
that $\mu' \star f'^{\mathtt{T}}$ and $\mu \star f^{\mathtt{T}}$ match. This
proves the claim.

With this choice of the test functions, we conclude that
    \begin{equation}    \label{eq:rtf}
    I_{\pi_E}(f') =
    J_{\pi}(f) + J_{\pi \otimes \eta}(f),
    \end{equation}
Since the functions are of positive type we conclude that $J_{\pi}(f) >0$ and
$J_{\pi \otimes \eta}(f) \geq 0$. Therefore $I_{\pi_E}(f') \not=0$ and the
assertion on the $L$-functions follows from Lemma~\ref{lemma:functoriality}.

We now move to the assertions on the local components of $\pi$. Let us fix a
nonsplit place $v$ of $F$. Since $L(s, \pi_0, \wedge^2 \otimes
\chi_v|_{F_v^\times})$ has simple pole at $s = 1$, we conclude that the
Langlands parameter of $\pi_{0, v}$ takes values in $\GSp_{2n}(\C)$ with
similitude character $\chi_v|_{F_v^\times}$.

It remains to calculate the local root number $\epsilon(\pi_{0, E, v} \otimes
\chi_v)$. By Lemma~\ref{lemma:matching_involution}, $\epsilon_{D_v}^n
\eta_v(-1)^n f_v$ and $f'^\dag_v$ also match. We define
    \[
    f'^\dag = \otimes_{w \not= v} f_w' \otimes f_v'^\dag.
    \]
Then by Lemma~\ref{lemma:involution} we have
    \[
    \epsilon(\pi_{0, E, v} \otimes \chi_v) I_{\pi_E}(f') =
    \epsilon_{D_v}^n \eta_v(-1)^n
    (J_{\pi}(f) + J_{\pi \otimes \eta}(f)) \not=0.
    \]
It follows that
    \[
    \epsilon(\pi_{0, E, v} \otimes \chi_v) = \epsilon_{D_v}^n \eta_v(-1)^n.
    \]
This finishes the proof of Theorem~\ref{thm:linear_periods}.

\subsection{Proof of  Theorem~\ref{thm:linear_periods_converse} }
We now move to Theorem~\ref{thm:linear_periods_converse}. The technical part
is, unlike the situation in Theorem~\ref{thm:linear_periods} where all test
functions supported in a small neighbourhood of $1 \in G(F_v)$ where $v$ is a
nonsplit place of $F$ can be transferred to $G'(F_v)$, not all test functions
$f'$ can be transferred to $G(F_v)$. In order for the test function on
$G'(F_v)$ to have a matching function $f$, there is a nontrivial vanishing
condition on the orbital integrals. Lacking a good understanding of
representation theory and harmonic analysis on $S'(F_v)$, the result we
obtain is unfortunately limited to the elliptic case.

We keep the assumptions of the theorem. Let us put $\Pi = \pi_{0, E}$. Then
$n$ is odd and $P_{\chi}'$ and $P_{\chi\eta}''$ are not identically zero on
$\Pi$. The later two conditions in particular implies that $\Pi^\vee \otimes
(\chi \chi^c)^{-1} = \Pi$ and $\Pi^c = \Pi$.

\begin{lemma}   \label{lemma:elliptic_support_split}
Assume that $\Pi_v$ is elliptic. Let $D$ be the quaternion algebra over
$F_v$, split (resp. nonsplit) if $\epsilon(\Pi_v \otimes \chi_v) =
\eta_v(-1)^n$ (resp. $-\eta_v(-1)^n$). Then there is an $f' \in
\cS(G'(F_v))_0$ (for the group $G(F_v)$ given by this quaternion algebra $D$)
such that $I_{\Pi_v}(f') \not=0$.
\end{lemma}

\begin{proof}
Let $f'$ be a test function supported in the elliptic locus such that
$I_{\Pi_v}(f') \not=0$. Let $f'' = f' + \epsilon(\Pi_v \otimes \chi_v)
f'^\dag$ where $-^\dag$ is the involution defined by~\eqref{eq:involution}
in Subsection~\ref{subsec:involution}. Then $I_{\Pi_v}(f'') = 2 I_{\Pi_v}(f')
\not=0$ by Lemma~\ref{lemma:involution}.

It remains to explain that $f'' \in \cS(G'(F_v))_0$. We need to prove that if
$x \in G'(F_v)$ is an elliptic element that does not match any $y \in
G(F_v)$, then $O^{G'}(x, f'') =0$.

The element $x$ matches $y\in G(F_v)$ if and only if $1 - (\alpha
\overline{\alpha})^{-1} \in \epsilon N \GL_n(E_v)$ where $\epsilon =
\epsilon(\Pi_v \otimes \chi_v) \eta_v(-1)^n$ in $F_v^\times/NE_v^\times$.
Since $x \in G'(F_v)$ is elliptic and $n$ is odd, we have $1 - (\alpha
\overline{\alpha})^{-1} \in \epsilon N \GL_n(E_v)$ if and only if $\eta_v(1 -
(\alpha \overline{\alpha})^{-1}) = \epsilon$, and this simplifies to $\eta_v(
1- \alpha \overline{\alpha}) = \epsilon(\Pi_v \otimes \chi_v)$.

Therefore if $x$ does not match and $y \in G(F_v)$, then $\eta_v( 1- \alpha
\overline{\alpha}) = - \epsilon(\Pi_v \otimes \chi_v)$. We have
    \[
    O^{G'}(x, f'^\dag) =
    O^{S'}(\tp{s'}, \widetilde{f'})
    = \eta_v( 1- \alpha \overline{\alpha})
    O^{S'}(s', \widetilde{f'})
    = - \epsilon(\Pi_v \otimes \chi_v) O^{G'}(x, f')
    \]
which implies $O^{G'}(x, f'') = 0$.
\end{proof}

A similar argument also proves the following lemma.

\begin{lemma}   \label{lemma:supercuspidal_matrix_coefficients}
Let $D$ be the quaternion algebra over $F_v$, split (resp. nonsplit) if
$\epsilon(\Pi_v \otimes \chi_v) = \eta_v(-1)^n$ (resp. $-\eta_v(-1)^n$).
Assume that $\Pi_v$ is supercuspidal. Let $f' \in \cS(G'(F_v))$ and
    \[
    f'_{\Pi_v}(g) =
    \sum_W \langle \Pi_v(f_v) W, \Pi_v(g)W \rangle
    \]
where $W$ runs through an orthonormal basis of $\Pi_v$. If $f''\in
\cS(G'(F_v))$ such that
    \[
    \int_{Z'(F_v)} f''(zg) \omega_{\Pi_v}(z) \rd z = f'_{\Pi_v}(g),
    \]
then $f'' \in \cS(G'(F_v))_0$ (for the group $G(F_v)$ given by this
quaternion algebra $D$).
\end{lemma}

\begin{proof}
We keep the notation from the proof of Lemma~\ref{lemma:involution}
and~\ref{lemma:elliptic_support_split}. To simplify notation, if $f'$ is a
matrix coefficient of $\Pi_v$, by the orbital integral $O(x, f')$ we mean
$O(x, f'_1)$ where $f'_1 \in \cS(G'(F_v))$ is a function such that
    \[
    \int_{Z'(F_v)} f''(zg) \omega_{\Pi_v}(z) \rd z = f'(g)
    \]
This is independent from the choice of $f''$.

The $\Hom$-spaces $\Hom_{H'(F)}(\Pi_v \otimes \chi_{H',v}, \C)$ and
$\Hom_{H''(F_v)}(\Pi_v \otimes \chi_v, \C)$ are both of dimension one, and
for a fixed regular semisimple $x \in G'(F_v)$, the linear form
    \[
    (W, W') \mapsto O^{G'}(x, \langle W, \Pi_v(\cdot) W' \rangle)
    \]
defines an element in
    \[
    \Hom_{H'(F)}(\Pi_v \otimes \chi_{H',v}, \C) \otimes
    \overline{\Hom_{H''(F_v)}(\Pi_v \otimes \chi_v, \C)}.
    \]
It follows that there is a function $A(x)$ on $G'(F_v)$ independent of $W$
and $W'$ such that
    \[
    O^{G'}(x, \langle W, \Pi_v(\cdot) W' \rangle) = A(x)
    \ell'(W) \overline{\ell''(W')}.
    \]
Thus
    \[
    O^{G'}(x, f'_{\Pi_v})= A(x) I_{\Pi_v}(f').
    \]

Let us now consider $O^{G'}(x, (f'^\dag)_{\Pi_v})$. By the proof of
Lemma~\ref{lemma:involution}, on the one hand, we have
    \[
    O^{G'}(x, (f'^\dag)_{\Pi_v}) = A(x) I_{\Pi_v}(f'^\dag) =
    \epsilon(\Pi_v \otimes \chi_v) A(x) I_{\Pi_v}(f') =
    \epsilon(\Pi_v \otimes \chi_v) O^{G'}(x, f'_{\Pi_v}).
    \]
On the other hand by the proof Lemma~\ref{lemma:elliptic_support_split} we
have
    \[
    O^{G'}(x, (f'^\dag)_{\Pi_v}) = O^{G'}(x, (f'_{\Pi_v})^\dag) =
    \eta_v(1 - \alpha \overline{\alpha}) O^{G'}(x, f'_{\Pi_v}).
    \]
Thus if $O^{G'}(x, f'_{\Pi_v}) \not=0$ we have $\epsilon(\Pi_v \otimes
\chi_v) = \eta_v(1-\alpha \overline{\alpha})$. As in the proof of
Lemma~\ref{lemma:elliptic_support_split}, this implies that $f'' \in
\cS(G'(F_v))_0$ (for the group $G$ given by this quaternion algebra $D$).
\end{proof}

\begin{proof}[Proof of Theorem~\ref{thm:linear_periods_converse}]
We just need to reverse the argument in the proof of
Theorem~\ref{thm:linear_periods}. Let $D$ be the quaternion algebra over $F$
that splits at all $v \not\in \Sigma$ and split places, and at a nonsplit
place $v \in \Sigma$, it is split (resp. nonsplit) if $\epsilon(\Pi_v \otimes
\chi_v) = \eta_v(-1)^n$ (resp. $-\eta_v(-1)^n$). Let $G = \GL_n(D)$.

By the assumption of the theorem we know that $P'_{\chi}$ and
$P''_{\chi\eta}$ are not identically zero on $\Pi$. This implies that if $v$
is a split place of $F$ and $w$ a place of $E$ above it, then $\Pi_{w} \simeq
\pi_{0, v}$ and $\Pi_w^\vee \otimes \chi_1\chi_2 \simeq \Pi_w$ if we write
$\chi_v = (\chi_1, \chi_2)$.

Again we have the sets of places $\mathtt{S}$ and $\mathtt{T}$ as in the
proof of Theorem~\ref{thm:linear_periods}. We require that the set
$\mathtt{S}$ contains $\Sigma$. We choose the test function $f'$ as follows.
Assume first that $v \not=v_1$. If $v \not\in \Sigma$, then we choose $f'_v =
\id_{G'(\fo_{F_v})}$ and $f_v = \id_{G(\fo_{F_v})}$. If $v$ is infinite then
we choose $f_v' \in \cS(G'(F_v))^+$ such that $I_{\Pi_v}(f'_v) \not=0$ and
let $f_v = \mathrm{bc}_v(f'_v) \in \cS(G(F_v))^+$ be the test function that
matches it. This is possible because $\Pi_w^\vee \otimes \chi_1\chi_2 \simeq
\Pi_w$. If $v \in \Sigma$ and is split we choose any $f'_v$ such that
$I_{\Pi}(f'_v) \not=0$ and let $f_v$ be the function that matches $f'_v$ in
the sense of Lemma~\ref{lemma:matching_split}. If $v \in \Sigma$ is nonsplit
and $\Pi_v$ is supercuspidal, then we take an $f''_v \in \cS(G'(F_v))$ and
choose $f'_v \in \cS(G'(F_v))$ such that
    \[
    \int_{Z'(F_v)} f'_v(zg) \omega_{\Pi_v}(z) \rd z = (f''_v)_{\Pi_v}(g),
    \]
as in Lemma~\ref{lemma:supercuspidal_matrix_coefficients}. We may assume that
$I_{\Pi'_v}(f'_v) \not=0$. By
Lemma~\ref{lemma:supercuspidal_matrix_coefficients}, $f'_v \in
\cS(G'(F_v))_0$ and we let $f_v \in \cS(G(F_v))$ be a test functions that
matches $f'_v$. If $v = v_1$, then we choose $f'_v$ to be a test function
supported in the elliptic locus such that $I_{\Pi_v}(f'_v) \not=0$ and $f'_v
\in \cS(G'(F_v))_0$. By Lemma~\ref{lemma:elliptic_support_split} such a test
function exists. We let $f_v \in \cS(G(F_v))$ be a test functions that
matches $f'_v$.

For this test function we have $I_{\Pi}(f') \not=0$. Now argue as in the
proof of Theorem~\ref{thm:linear_periods}, we conclude that
    \begin{equation}    \label{eq:rtf_converse}
    I_{\Pi}(f') = J_{\pi}(f) + J_{\pi \otimes \eta}(f).
    \end{equation}
Here $\pi$ is an irreducible cuspidal representation of $G(\bA_F)$ such that
$\pi_E = \Pi$. So $P_{\chi}$ is not identically on either on $\pi$ or $\pi
\otimes \eta$. But as a character of $G(\bA_F)$, $\eta$ is trivial when
restricted to $H(\bA_F)$. It follows that $P_{\chi}$ is identically zero on
$\pi$ if and only if it is so on $\pi \otimes \eta$. Thus it is not
identically zero on both. Finally $\pi_0$, $\pi$ and $\pi \otimes \eta$ agree
at all the split places. So either $\pi$ or $\pi \otimes \eta$ is the
Jacquet--Langlands transfer of $\pi_0$ to $G(\bA_F)$.

The uniqueness of $D$ follows from Theorem~\ref{thm:linear_periods}. This
finishes the proof of Theorem~\ref{thm:linear_periods_converse}.
\end{proof}

\section{Analysis of the orbits}

\subsection{Semisimple elements in $S'$}    \label{sec:orbit_S'}
In this section, $E/F$ is a quadratic field extension of either number fields
or local fields of characteristic zero. Recall that we have the groups $G' =
\Res_{E/F} \GL_{2n}$, $H' = \Res_{E/F}(\GL_n \times \GL_n)$ embedded in $G'$
as diagonal blocks, $H'' = \GL_{2n, F}$, and $Z' \simeq \GL_{1, F}$ the split
center of $G'$. We also have a symmetric space
    \[
    S' = \{ g \overline{g}^{-1} \mid g\in G'\}
    \]
on which $H'$ acts by twisted conjugation. An element in $S'(F)$ is called
semisimple if its $H'$-orbit is (Zariski) closed. It is called regular
semisimple if its stabilizer in $H'$ is a torus of dimension $n$. It is
called elliptic if further more its stabilizer in $H'$ is an elliptic torus
modulo $Z'$ (over $F$). An element $g \in G'(F)$ is semisimple (resp. regular
semisimple, resp. elliptic) if $g \overline{g}^{-1}$ is so in $S'(F)$.

In what follows we are going to make repeated use of the following lemma
without mentioning it, cf.~\cite{AC}*{Chapter~1, Lemma~1.1}.

\begin{lemma}
Let $g \in \GL_n(E)$. Then $N g$ is conjugate to an element in $\GL_n(F)$.
Moreover if $g_1, g_2 \in \GL_n(E)$, then $g_1$ and $g_2$ are twisted
conjugate if and only if $Ng_1$ and $N g_2$ are conjugate in $\GL_n(E)$.
\end{lemma}

We first classify all semisimple elements in $S'(F)$.

\begin{lemma}   \label{lemma:semisimple_S'}
Every semisimple element in $S'(F)$ is in the $H'(F)$-orbit of the form
$s'(\alpha, n_1, n_2, n_3)$, where $n_1+n_2+n_3 = n$ is a partition of $n$,
$\alpha \in \GL_{n_1}(E)$, $\alpha \overline{\alpha} \in \GL_{n_1}(F)$ is
semisimple in the usual sense, $\det (\alpha \overline{\alpha} - 1) \not=0$,
and
    \[
    s'(\alpha, n_1, n_2, n_3) = \begin{pmatrix}
    \alpha &&& 1_{n_1} \\& 0_{n_2} &&& 1_{n_2} \\
    && 1_{n_3} &&&0_{n_3} \\
    1_{n_1} - \alpha \overline{\alpha} &&& - \overline{\alpha}\\
    & 1_{n_2} &&& 0_{n_2} \\ &&0_{n_3}&&& 1_{n_3}
    \end{pmatrix}
    \]
\end{lemma}

\begin{proof}
Let $s' = \begin{pmatrix} A &B\\ C &D \end{pmatrix}$ be a semisimple element.
We claim that $A \overline{A}$, $D \overline{D}$, $\overline{B}C$ and $C
\overline{B}$ are all semisimple elements in $M_n(E)$ in the usual sense. To
see this we may assume that $F$ is algebraically closed as being semisimple
is a property that does not depend on the base field. Then $E = F \times F$,
$S'$ consists of elements of the form $(g, g^{-1})$, $g \in \GL_{2n}(F)$, and
it is identified with $\GL_n(F)$ by projection to the first factor. The group
$H'$ is identified with four copies of $\GL_n(F)$, and two copies of
$\GL_n(F) \times \GL_n(F)$ respectively acts on $\GL_{2n}(F)$ by left and
right translation. The claim then reduces to~\cite{JR}*{Lemma~4.2}.

We will use twisted conjugation by elements in $H'(F)$ to reduce $s'$ to an
element of the form $s'(\alpha, n_1, n_2, n_3)$. This takes several steps.
When we say that ``$s'$ or one of the blocks in $s'$ takes a particular
form'', or ``we may assume a block of $s'$ is of the form'', we mean that
after replacing $s'$ by its twisted conjugation by elements in $H'(F)$ which
do not change the particular shape of $s'$ we have achieved in the previous
steps, that block of $s'$ takes the shape that we want.

\emph{\underline{Step~1}: Simplifying $B$.} We may assume that $B$ is of the
form
    \[
    \begin{pmatrix} 1_{n - n_3} \\ & 0_{n_3} \end{pmatrix}.
    \]
Make a partition
    \[
    A = \begin{pmatrix} A_1 & A_2 \\ A_3 & A_4 \end{pmatrix}
    \]
of the matrix $A$ with $A_1 \in M_{n-n_3}(E)$, and similar partitions for $C$
and $D$. From the condition that $s' \overline{s'} = 1$, we conclude
    \[
    A \begin{pmatrix} 1_{n-n_3} \\ & 0_{n_3} \end{pmatrix} +
    \begin{pmatrix} 1_{n - n_3} \\ & 0_{n_3} \end{pmatrix} \overline{D} = 0, \quad
    C \overline{A} + D \overline{C} = 0,
    \]
and
    \[
    A \overline{A} + \begin{pmatrix} 1_{n-n_3} \\ & 0_{n_3} \end{pmatrix}
    \overline{C} = 1_n,
    \quad
    C \begin{pmatrix} 1_{n - n_3} \\ & 0_{n_3} \end{pmatrix} + D \overline{D} = 1_n.
    \]
It follows that $A_3 = 0$, $D_2 = 0$, $A_4 \overline{A_4} = D_4
\overline{D_4} = 1_{n_3}$. Thus we may assume $A_4 = D_4 = 1_{n_3}$. Thus
$s'$ takes the following form
    \[
    s' = \begin{pmatrix} A_1 & A_2 & 1 \\ & 1 && 0 \\
    C_1 & C_2 & -\overline{A_1} \\
    C_3 & C_4 & D_3 & 1 \end{pmatrix}.
    \]
Since $A \overline{A}$ is semisimple in the usual sense, $A_1 \overline{A_1}$
is so. We may further assume that $A_1 \overline{A_1}$ has entries in $F$. Then $C_1$ has entries in $F$.

\emph{\underline{Step~2}: $C_1$ is invertible.} To see this, first as
$\overline{B} C$ is semisimple in the usual sense, we conclude that $C_1 \in
M_{n-n_3}(F)$ is semisimple in the usual sense. If $C_1$ is not invertible we
may assume that $C_1$ is of the form $\begin{pmatrix} 0 \\ & C_1'
\end{pmatrix}$ where $C_1'$ is invertible. Using the fact that
$\overline{B}C$ and $C \overline{B}$ are semisimple in the usual sense, we
conclude that $C$ has to be of the form
    \[
    \begin{pmatrix} 0 \\ &C_1' & * \\ &* &* \end{pmatrix}.
    \]
Therefore
    \[
    \overline{A_1} A_1 = \begin{pmatrix} 1 \\ & 1 - C_1'\end{pmatrix}.
    \]
Since $1$ is not an eigenvalue of $1 - C_1'$, we may assume that $A_1$ takes the
form $\begin{pmatrix} 1 \\ & A_{12} \end{pmatrix}$. It follows that $A$ takes
the form
    \[
    A = \begin{pmatrix} 1 &&A_{21} \\& A_{12} &A_{22}\\ &&1 \end{pmatrix}, \quad
    A_2 = \begin{pmatrix} A_{21} \\ A_{22} \end{pmatrix}.
    \]
That the upper right corner of $C$ equals zero implies that the upper right
corner of $A$ is purely imaginary. Now twisted conjugate $s'$ by an element
in $H'(F)$ of the form
    \[
    \left( \begin{pmatrix} 1 & &* \\ &1\\ &&1 \end{pmatrix}, \quad
    \begin{pmatrix} 1  \\ &1 \\ * &&1 \end{pmatrix} \right)
    \]
we may assume that the upper right corner is zero.

 Then it follows that $s'$ is of the form
    \[
    \begin{pmatrix} 1 &&&1\\ & *&* &&1 \\
    &&1 &&&0 \\
    0 &&& 1 \\
    & C_1' &* &&* \\
    &* &* &&* &1 \end{pmatrix}.
    \]
For any $\lambda \in F^\times$, the element
    \[
    \begin{pmatrix} 1 &&&\lambda\\ & *&* &&1 \\
    &&1 &&&0 \\
    0 &&& 1 \\
    & C_1' &* &&* \\
    &* &* &&* &1 \end{pmatrix} = \begin{pmatrix} \lambda \\ &*\\&&*\\
    &&&1 \\ &&&&* \\&&&&&* \end{pmatrix}
    \begin{pmatrix} 1 &&&1\\ & *&* &&1 \\
    &&1 &&&0 \\
    0 &&& 1 \\
    & C_1' &* &&* \\
    &* &* &&* &1 \end{pmatrix}
    \begin{pmatrix} \lambda \\ &*\\&&*\\
    &&&1 \\ &&&&* \\&&&&&* \end{pmatrix}^{-1}
    \]
is in the same $H'(F)$-orbit of $s'$. Since $s'$ is semisimple, it follows that the limit as $\lambda\to 0$
    \[
    \begin{pmatrix} 1 &&&0\\ & *&* &&1 \\
    &&1 &&&0 \\
    0 &&& 1 \\
    & C_1' &* &&* \\
    &* &* &&* &1 \end{pmatrix},
    \]
is also in the orbit. It is a contradiction, since
the rank of $B$ is a constant in an $H'(F)$-orbit. This proves the claim that
$C_1$ is invertible.

\emph{\underline{Step~3}: final simplification.} Once we have that $C_1$ is
invertible, we may assume that $C_2$ and $C_3$ are both zero. Then by $C
\overline{A} + D \overline{C} = 0$, we conclude that $A_2$ and $D_3$ are
zero. Using semisimplicity of $s'$ again we conclude that $C_4$ should be
zero. So we arrive at the conclusion that $s'$ takes the form
    \[
    \begin{pmatrix} A_1 &  & 1 \\ & 1 && 0 \\
    1- A_1 \overline{A_1} &  & -\overline{A_1} \\
    & &  & 1 \end{pmatrix}.
    \]

Finally we may replace $A_1$ by its twisted conjugation so that $A_1 =
\begin{pmatrix} \alpha \\ & 0 \end{pmatrix}$ where $\alpha$ is invertible and
$\alpha \overline{\alpha} \in \GL_{n_1}(F)$. The lemma then follows.
\end{proof}

\begin{lemma}
Let $s' = s'(\alpha, n_1, n_2, n_3)$ be a semisimple element in $S'(F)$ as in
Lemma~\ref{lemma:semisimple_S'}. It is regular if and only if $n_1 = n$ and
$\alpha \overline{\alpha} \in \GL_n(F)$ is regular semisimple in the usual
sense. It is elliptic if $\alpha \overline{\alpha} \in \GL_n(F)$ in the usual
sense. Two regular semisimple $s'(\alpha_1, n, 0,0)$ and $s'(\alpha_2,
n,0,0)$ are in the same $H'(F)$-orbit if and only if $\alpha_1$ and
$\alpha_2$ are twisted conjugate in $\GL_n(E)$.
\end{lemma}

\begin{proof}
The stabilizer of $s'$ in $H'$ is isomorphic to
    \[
    (\GL_{n_1, E})_{\alpha, \text{twisted}}
    \times \GL_{n_2, E}
    \times (\GL_{n_3, F} \times \GL_{n_3, F}).
    \]
The twisted stabilizer $(\GL_{n_1, E})_{\alpha, \text{twisted}}$ is an inner
form of $(\GL_{n_1, F})_{\alpha \overline{\alpha}}$, whose dimension is at
least $n_1$. Thus the dimension of $H'_{s'}$ is at least
    \[
    n_1 + 2 n_2^2 + 2n_3^3 \geq n,
    \]
and the equality holds if and only if $n_2= n_3 = 0$ and $\dim (\GL_{n_1,
E})_{\alpha, \text{twisted}} = n$, which is equivalent to that $\alpha
\overline{\alpha}$ is regular semisimple in $\GL_n(F)$. The other assertions
of the lemma are obvious.
\end{proof}

To simplify notation, for any $\alpha \in M_n(E)$, we put $s'(\alpha) =
s'(\alpha, n, 0, 0)$.

Let $\mathbf{A}^n$ be the affine space of dimension $n$ over $F$ and
    \[
    q': S' \to \mathbf{A}^n
    \]
be the morphism
    \[
    \begin{pmatrix} A & B \\ C & D \end{pmatrix}
    \mapsto \Tr \wedge^i (2A \overline{A} - 1), \quad i = 1, \cdots, n.
    \]

\begin{lemma}   \label{lemma:categorical_quotient_split}
The map $q'$ is a categorical quotient.
\end{lemma}

\begin{proof}
This is a geometric statement, so we may assume that $F$ is algebraically
closed. Then as in the proof of Lemma~\ref{lemma:semisimple_S'} we are then
reduced to the case considered in~\cites{Guo2,CZhang}.
\end{proof}

\subsection{Explicit \'etale Luna slices} \label{subsec:etale_Luna}
Let us begin with some general discussion. Let $G$ be a reductive algebraic
group acting on an affine algebraic variety $X$. We denote by $X \to X//G$,
or simply $X//G$ the categorical quotient. Let $x \in X$ and $T_x$ be the
tangent space of $X$ at $x$. We fix a $G_x$-invariant inner product on $T_x$.
Let $TO_x$ be the tangent space of orbit $Gx$ at $x$, and $N_x  =
TO_x^{\perp}$ the orthogonal complement in $T_x$. The group $G_x$ then acts
on $N_x$, and this is called the sliced representation at $x$. An \'etale Luna
slice at $x$ is a locally closed subvariety $Z \subset X$, containing $x$ and
stable under the $G_x$ action, together with a strongly \'etale
$G_x$-equivariant morphism $\iota: Z \to N_x$ such that the morphism
    \[
    G \times_{G_x} Z \to X
    \]
is strongly \'etale. Here if $X$ and $Y$ are affine varieties with $G$ actions,
a morphism $X \to Y$ is called strongly \'etale if the induced morphism $X//G
\to Y//G$ is \'etale and the diagram
    \[
    \xymatrix{
    X \ar[d] \ar[r] & Y \ar[d] \\
    X//G \ar[r] & Y//G
    }
    \]
is Cartesian.

We now come back to the action of the group $H'$ on $S'$. Let
    \[
    \fs' = \left\{ \begin{pmatrix} &X \\ Y \end{pmatrix} \Big|\
    X, Y \in M_n(E)^- \right\},
    \]
where $M_n(E)^-$ stands for the matrices with purely imaginary entries. This
is viewed as an algebraic variety over $F$. It is isomorphic to the tangent
space of $S'$ at $1$. The stabilizer of $1$ in $H'$ is isomorphic to $\GL_{n,
F} \times \GL_{n, F}$, which acts on $\fs'$ by conjugation. Let $g \in G'(F)$
and $s' = g \overline{g}^{-1} \in S'(F)$. The tangent space of $S'$ at $s'$
is identified with
    \[
    T_{s'} = \{ gY \overline{g}^{-1} \mid Y \in \fs'\}.
    \]
The tanget space of the $H'$-orbit of $s'$ is identified with a subspace
    \[
    TO_{s'} = \{ Xs' - s' \overline{X} \mid X \in \fh'\}.
    \]
We fix an inner product on $T_s$ by
    \[
    \langle gY_1 \overline{g}^{-1}, gY_2 \overline{g}^{-1} \rangle =
    \Tr Y_1 Y_2.
    \]
This inner product is $H'_{s'}$-invariant. Put
    \[
    N_{s'} = \{ Ys' \mid \theta'(Y) = -Y, \ Ys' = -s' \overline{Y}\}.
    \]
Then we have an orthogonal decomposition
    \[
    T_{s'} = TO_{s'} \oplus^{\perp} N_{s'}.
    \]
Thus $N_{s'}$ is the sliced representation.

Put
        \[
        Z' = \left\{ x s' \mid x \theta(x) = 1,
        \ x s' = s' \overline{x}^{-1}, \ \det(1+x) \not=0,\
        \det\left((1 - \Ad(xs'))|_{\fg'^{\perp}_{s'}} \right) \not=0
        \right\},
        \]
where $\fg'_{s'}$ stands for the Lie algebra of the centralizer of $s'$ in
$G'$, and the orthogonal complement is taken with respect to an
$H_{s'}'$-invariant inner product. Then $Z'$ is a locally closed subscheme of
$S'$. Put also
        \[
        \iota': Z' \to N_{s'}, \quad xs' \mapsto (1-x)(1+x)^{-1}s'.
        \]

\begin{lemma}   \label{lemma:Luna}
The maps $\iota': Z' \to N_{s'}$ and $(H' \times Z')//H'_{s'} \to S'$ are
strongly \'etale. Therefore $Z'$ is an \'etale Luna slice at $s'$.
\end{lemma}

\begin{proof}
Since this is a geometric statement we may assume that the base field $F$ is
algebraically closed. As in the proof of
Lemma~\ref{lemma:categorical_quotient_split}, this lemma is then reduced to
the explicit construction of \'etale Luna slices in the case of relative trace
formula of Guo--Jacquet, and this is explained
in~\cite{CZhang}*{Subsection~5.3} (which in turn is based
on~\cite{JR}*{Section~5.2}).
\end{proof}

We will need more concrete descriptions of the sliced representations. Let
$s' = s'(\alpha, n_1, n_2, n_3)$ be a semisimple element in $S'(F)$. Its
stabilizer in $H'$ equals
    \[
    H_1' \times H_2' \times H_3' =
    (\GL_{n_1, E})_{\alpha, \text{twisted}}
    \times \GL_{n_2, E}
    \times (\GL_{n_3, F} \times \GL_{n_3, F}).
    \]
The sliced representation at $s'(\alpha, n_1, n_2, n_3)$ is isomorphic to
$V_1' \oplus V_2' \oplus V_3'$ where
    \[
    V_1' = \{A \in M_{n_1}(E) \mid \alpha \overline{A} = A \overline{\alpha} \},
    \quad V_2' = M_{n_2}(E), \quad V_3' = M_{n_3}(E)^- \oplus M_{n_3}(E)^-.
    \]
In the three extreme cases where $n_1 = n$, $n_2 = n$ and $n_3 = n$, we have
the following descriptions.

\begin{enumerate}
\item Assume $n_1 = n$, $n_2 = n_3 = 0$. The embedding of $(\GL_{n,
    E})_{\alpha, \text{twisted}}$ in $H'$ is given by
    \[
    h \mapsto \begin{pmatrix} h \\ & \overline{h}\end{pmatrix}.
    \]
    The embedding of $V_1'$ in $N_{s'}$ is given by
        \[
        A \mapsto \begin{pmatrix} & A\\
        -\overline{A}(1 - \alpha \overline{\alpha})\end{pmatrix} s'.
        \]

\item Assume $n_2 = 0$, $n_1 = n_3 = 0$. Then $s' = \begin{pmatrix} & 1 \\
    1
    \end{pmatrix}$. The embedding of $\GL_{n ,E}$ in $H'$ is given by
        \[
        h \mapsto \begin{pmatrix} h \\ & \overline{h}\end{pmatrix}.
        \]
    The embedding of $V_2' = M_n(E)$ into $N_{s'}$ is given by
        \[
        A \mapsto \begin{pmatrix} & A \\ - \overline{A} \end{pmatrix}s' =
        \begin{pmatrix} A \\ & - \overline{A} \end{pmatrix}
        \]

\item Assume $n_1 = n_2 = 0$. Then $s' = 1$. The embedding of $\GL_{n, F}
    \times \GL_{n, F}$ in $H'$ is given by
        \[
        (h_1, h_2) \mapsto \begin{pmatrix} h_1 \\ & h_2 \end{pmatrix}.
        \]
    The embedding of $V_3' = M_{n_3}(E)^- \oplus M_{n_3}(E)^-$ into
    $N_{s'}$ is given by
        \[
        (A, B) \mapsto \begin{pmatrix} &A \\ B \end{pmatrix} s' =
        \begin{pmatrix} &A \\ B \end{pmatrix}.
        \]
\end{enumerate}

In general, in obvious notation, we have $Z' = Z_1' \times Z_2' \times Z_3'$
according to the decomposition $N_{s'} = V_1' \times V_2' \times V_3'$. We
also have the $H'_{s'}$-equivariant morphism $\iota' = \iota_1' \times
\iota_2' \times \iota_3'$.

\subsection{Semisimple elements in $G$}
Recall that $D$ is a quaternion algebra over $F$ with a fixed embedding $E
\to D$, $G = \GL_n(D)$, $H = \Res_{E/F} \GL_{n}$ the centralizer of
$E^\times$ in $G$, $Z \simeq \GL_{1, F}$ the split center of $G$. Recall that
$\theta$ is the automorphism of $G$ given by conjugation by
$\begin{pmatrix} 1_n \\
&-1_n
\end{pmatrix}$. Then $H$ is the stabilizer of $\theta$.

Put $S = \{ g \theta(g)^{-1} \mid g \in G\}$ and $H$ acts on $S$ by
conjugation. Similarly to $S'$, an element in $S(F)$ is called semisimple if
its $H$-orbit is (Zariski) closed. It is called regular semisimple its
stabilizer in $H$ is a torus of dimension $n$. It is called elliptic if
further more its stabilizer in $H$ is an elliptic torus modulo $Z$ (over
$F$). An element $g \in G(F)$ is semisimple (resp. regular semisimple, resp.
elliptic) if $g \theta(g)^{-1}$ is so in $S(F)$.

The following lemma summarizes~\cite{Guo2}*{Proposition~1.2}.

\begin{lemma}   \label{lemma:semisimple_S}

Every semisimple $g \in G(F)$ is in the $(H \times H)(F)$-orbit of
    \[
    g(\beta, n_1, n_2, n_3) = \begin{pmatrix} 1_{n_1} &&& \epsilon \beta \\
    & 0_{n_2} &&& \epsilon 1_{n_2} \\ && 1_{n_3} &&& 0_{n_3}\\
    \overline{\beta} &&& 1_{n_1} \\
    &1_{n_2} &&& 0_{n_2} \\ && 0_{n_3} &&& 1_{n_3}
    \end{pmatrix},
    \]
where $\beta \in \GL_{n_1}(E)$, $\beta \overline{\beta} \in \GL_n(F)$ and
$\det(1- \epsilon \beta \overline{\beta}) \not=0$. It is regular semisimple
(resp. elliptic) if $n_2 = n_3 = 0$ and $\beta\overline{\beta}$ is regular
semisimple (resp. elliptic) in $\GL_n(F)$ in the usual sense.
\end{lemma}

If $n_2 = n_3 = 0$, we write $s(\beta, n, 0, 0) = s(\beta)$ and $g(\beta) =
g(\beta, n, 0, 0)$.

Let $g \in G$ and $s = g\theta(g)^{-1} = \begin{pmatrix} A & B \\ C &D
\end{pmatrix} \in S$. Define
    \[
    q: G \to \mathbf{A}^n, \quad g \mapsto \Tr \wedge^i A, \quad
    i = 1,\cdots n.
    \]

\begin{lemma}
The map $q$ is a categorical quotient.
\end{lemma}

\begin{proof}
This is again a geometric statement, so we may assume that $F$ is
algebraically closed. Then the lemma reduces to the case considered
in~\cites{Guo2,CZhang}.
\end{proof}

At each semisimple $g \in G(F)$, we construct an explicit \'etale Luna slice as
follows. Put $s= g \theta(g)^{-1} \in S(F)$. We note $(H \times H)_g \simeq
H_s$. The normal space of the $H\times H$-orbit of $g$ at the point $g$ is
identified with
    \[
    N_{g} = \{ Yg \mid Y \in \fs, Ys = sY\}.
    \]
The group $(H \times H)_g$ acts on $N_g$ and $N_{g}$ is the sliced
representation. Let
    \[
    Z = \{ y g \in G \mid y \theta(y) = 1, \ y s = s y,\ \det(1+y) \not=0,\
    \det\left((1 - \Ad(ys))|_{\fg^{\perp}_{s}} \right) \not=0\},
    \]
where $\fg_{s}$ stands for the Lie algebra of the centralizer of $s$ in $G$,
and the orthogonal complement is taken with respect to an $H_{s}$-invariant
inner product. There is a map $\iota: Z \to N_{s}$ given by
    \[
    yg \mapsto (1 - y)(1+y)^{-1}g.
    \]

\begin{lemma}   \label{lemma:Luna_S}
The maps $Z \to N_{s}$ and $((H \times H) \times Z)//(H \times H)_g \to G$
are strongly \'etale. Therefore $Z$ is an \'etale Luna slice at $g$.
\end{lemma}

\begin{proof}
We first recall that $(H \times H)_g = H_s$. Checking that $((H \times H)
\times Z)//(H \times H)_g \to G$ is strongly \'etale is equivalent to checking
that $(H \times Z)//H_s \to S$ is strongly \'etale. This is a geometric
statement so we may assume that $F$ is algebraically closed. Then the action
of $H$ on $S$ is reduced to two copies of $\GL_{n, F} \times \GL_{n, F}$
acting on $\GL_{2n, F}$. The lemma again reduces to the description
in~\cite{CZhang}*{Section~5.3}.
\end{proof}

We now give more concrete descriptions of the sliced representation at $g =
g(\beta, n_1, n_2, n_3)$. Let $\fs$ be the space consisting of matrices of
the form
    \[
    \begin{pmatrix} & \epsilon Y \\ \overline{Y} \end{pmatrix}, \quad
    Y \in M_n(E).
    \]
It is identified with the tangent space of $S$ at $1$, and the group $H$ acts
on $\fs$ by conjugation. We also write $\fs_n$ to indicate the size of the
space $\fs$. The stabilizer $(H\times H)_g$ is isomorphic to $H_1 \times H_2
\times H_3$ where
    \[
    H_1 = (\GL_{n_1, E})_{\beta, \mathrm{twisted}}, \quad H_2 = \GL_{n_2, E},
    \quad H_3 = \GL_{n_3, E}.
    \]
The norm space $N_g = V_1 \oplus V_2 \oplus V_3$ with
    \[
    V_1 = \left\{ \begin{pmatrix} &\epsilon Y \\ \overline{Y} \end{pmatrix}
    \in \fs_{n_1} \ \Big|\ \beta \overline{Y} = Y \overline{\beta} \right\},
    \quad V_2 = \fs_{n_2},
    \quad V_3 = \fs_{n_3}.
    \]
We also have $Z = Z_1 \times Z_2 \times Z_3$ according to the decomposition
$N_{g} = V_1 \times V_2 \times V_3$, and the $H_{s}$-equivariant morphism
$\iota = \iota_1 \times \iota_2 \times \iota_3$.

\section{Transfer of test functions}    \label{sec:transfer}

\subsection{Matching}   \label{subsec:matching_def}
The goal of this section is to prove Theorem~\ref{thm:transfer_group}.
Throughout this section, we suppress the subscript $v$ and assume that $E/F$
is a quadratic extension of local fields of characteristic zero.

Recall that we have the categorical quotients
    \[
    q': S' \to \mathbf{A}^n, \quad q: G \to \mathbf{A}^n.
    \]
Two semisimple elements $s' \in S'(F)$ and $g \in G(F)$ match if they are
mapped to the same point in $\mathbf{A}^n$. More concretely we have

\begin{lemma}   \label{lemma:matching_concrete_condition}
Two regular semisimple elements $s'(\alpha) \in S'(F)$ and $g(\beta) \in
G(F)$ match if and only if $-(1-\alpha \overline{\alpha})(\alpha
\overline{\alpha})^{-1}$ and $\epsilon \beta\overline{\beta}$ have the same
characteristic polynomial.
\end{lemma}

\begin{proof}
Let $s'(\alpha) \in G'(F)$ and $g(\beta) \in G(F)$ be regular semisimple
elements. The upper left $n \times n$ block of its image in $S(F)$ is
    \[
    (1 - \epsilon \beta \overline{\beta})^{-1}
    (1+\epsilon \beta \overline{\beta}).
    \]
By definition, that $s'(\alpha)$ and $g(\beta)$ match is equivalent to $2
\alpha \overline{\alpha}-1$ and $(1 - \epsilon \beta \overline{\beta})^{-1}
(1+\epsilon \beta \overline{\beta})$ have the same characteristic polynomial.
This is further equivalent to that $-(1-\alpha \overline{\alpha})(\alpha
\overline{\alpha})^{-1}$ and $\epsilon \beta\overline{\beta}$ have the same
characteristic polynomial.
\end{proof}

We put
    \[
    G(F)_{\reg, 0} = \left\{g \in G(F)_{\reg} \ \Big|\
    g \theta(g)^{-1} = \begin{pmatrix} A & \epsilon B \\
    \overline{B}& \overline{A} \end{pmatrix}
    \in S(F), \
    \frac{1}{2}(A +1) \in N \GL_n(E)\right\},
    \]
and
    \[
    S'(F)_{\reg, 0} = \left\{ s' =
    \begin{pmatrix} A & B \\ C &D \end{pmatrix} \in S'(F)_{\reg}\
    \Big|\ 1 - (A \overline{A})^{-1}
    \in \epsilon N\GL_n(E) \right\}.
    \]

\begin{lemma}   \label{lemma:matching_orbits}
The matching defines a bijection between the $H$-orbits in $G_{\reg, 0}$ to
$S'_{\reg, 0}$.
\end{lemma}

\begin{proof}
Take $s'(\alpha) \in S'(F)_{\reg, 0}$. By definition we can find a $\beta \in
\GL_n(E)$ such that
    \[
    1- (\alpha \overline{\alpha})^{-1} = \epsilon \beta \overline{\beta}.
    \]
Since $\alpha \overline{\alpha} \in \GL_n(F)$ is regular semisimple in the
usual sense, so is $\beta \overline{\beta}$. Then $g(\beta)$ is regular
semisimple in $G(F)$ and matches $s'(\alpha)$ by definition. The other
direction is proved similarly.
\end{proof}

Recall from \eqref{eq:tran fact} that we have defined transfer factors on $G$ and $G'$.  Fixed a character
$\widetilde{\eta}: E^\times \to \C^\times$ whose restriction to $F^\times$
equals $\eta$. Fix a purely imaginary element $\tau \in E^\times$. If $s' =
\begin{pmatrix} A' &B' \\C' &D' \end{pmatrix} \in S'(F)$
be a regular semisimple element, then by Lemma~\ref{lemma:semisimple_S'},
$A', B', C', D'$ are all invertible and we define  (cf.  \eqref{eq:tran fact})
\begin{align} \label{eq:tran fact S'}
\kappa^{S'}(s') = \chi(\tau
D') \widetilde{\eta}(B').
\end{align}
If $g \in G(F)$ and $g^{-1} =
\begin{pmatrix} A & \epsilon B \\ \overline{B} & \overline{A}
\end{pmatrix} \in G(F)$ then by Lemma~\ref{lemma:semisimple_S}, $A, B$
are both invertible and in  \eqref{eq:tran fact} we put $\kappa^G(s) = \chi(A)$.

We also put
    \[
    \cS(G(F))_0 = \{f \in \cS(G(F)) \mid O^G(g, f) = 0
    \text{ for all $g \not\in G(F)_{\reg, 0}$}\},
    \]
and
    \[
    \cS(S'(F))_0 = \{f' \in \cS(S'(F)) \mid O^{G'}(s', f') = 0
    \text{ for all $s'\not\in S'(F)_{\reg, 0}$}\}.
    \]
Then we say that $f \in \cS(G(F))_0$ and $f' \in \cS(G'(F))_0$ match if
    \[
    \kappa^G(g) O^G(g, f) = \kappa^{G'}(s') O^{G'}(s', f'_{s'}),
    \]
for all matching regular semisimple $g \in G(F)$ and $s' \in S'(F)$.

\begin{theorem} \label{thm:transfer}
For all $f \in \cS(G(F))_0$ there is an $f' \in \cS(S'(F))_0$ that matches
it, and vice versa.
\end{theorem}

By~\eqref{eq:f'tilde} and~\eqref{eq:simplify_orbital_integral}, this theorem
and Theorem~\ref{thm:transfer_group} are equivalent. The rest of this section
is devoted to the proof of Theorem~\ref{thm:transfer}. The basic idea is that
matching of test functions is a local property, and it is enough to prove the
existence of ``local transfer'' near each semisimple point. The orbital
integrals near a semisimple point can be related to the orbital integrals on
the corresponding sliced representation via semisimple descent. In the
current setting the orbital integrals on the sliced representations turn out
to be familiar ones. For readers' convenience we record this in the following
lemma, whose proof can be found in~\cite{Zhang1}*{Proposition~3.8}.

\begin{lemma}   \label{lemma:local_transfer}
Let $f \in \cS(G(F))_0$. If for all matching semisimple $g \in G(F)$ and $s'
\in S'(F)$, there is an $(H \times H)(F)$-invariant neighbourhood $U$ of $g$
in $G(F)$ and an $H'(F)$-invariant neighbourhood $U'$ of $s'$ in $S'(F)$, and
a function $f'_{s'} \in \cS(S'(F))_0$, such that for all matching regular
semisimple $g \in U$ and $s' \in U'$ we have
    \[
    \kappa^G(g) O^G(g, f) = \kappa^{S'}(s') O^{S'}(s', f'_{s'}),
    \]
then there is an $f' \in \cS(S'(F))_0$ that matches $f$. Similar statement
holds in the other direction.
\end{lemma}

The condition in this lemma will be referred to as the existence of local
transfer at $g$ and $s'$. The existence of local transfer is what we will
prove later in this section.

\subsection{Transfer on the sliced representations}
Near each semisimple point, the orbital integral can be connected to the
orbital integral on the sliced representation of that point. As the
preparation for the proof of Theorem~\ref{thm:transfer}, we explain the
transfer of test functions on the sliced representations in this subsection.

The transfer of test functions on the sliced representations are closely
related to the transfer of test functions in the relative trace formulae of
Guo and Jacquet, which has been established in~\cite{CZhang}. We recall it
first. Put $\mathsf{V} = M_n(F) \times M_n(F)$ and $\mathsf{W} = M_n(E)$. The
group $\mathsf{H}' = \GL_{n,F} \times \GL_{n,F}$ acts on $\mathsf{V}$ by
    \[
    (h_1, h_2) \cdot (X_1, X_2) = (h_1 X_1 h_2^{-1}, h_2 X_2 h_1^{-1}),
    \]
and the group $\mathsf{H} = \Res_{E/F} \GL_{n, E}$ acts on $\mathsf{W}$ by
twisted conjugation. The element $(X_1, X_2) \in \mathsf{V}$ is regular
semisimple if $X_1X_2 \in \GL_n(F)$ and is regular semisimple in the usual
sense. The element $Y \in \mathsf{W}$ is regular semisimple if $Y
\overline{Y} \in \GL_n(E)$ and is regular semisimple in the usual sense. Two
regular semisimple elements $(X_1, X_2) \in \mathsf{V}$ and $Y \in
\mathsf{W}$ match if $X_1X_2$ and $\epsilon Y \overline{Y}$ have the same
characteristic polynomial. For $f' \in \cS(\mathsf{V})$ and $f \in
\cS(\mathsf{W})$, define the orbital integrals
    \[
    O^{\mathsf{V}}((X_1, X_2), f') =
    \int_{\mathsf{H}_{(X_1, X_2)}(F) \bs \mathsf{H}(F)}
    f'(h_1 X_1 h_2^{-1}, h_2 X_2 h_1^{-1}) \eta(h_1h_2) \rd h_1 \rd h_2,
    \]
and
    \[
    O^{\mathsf{W}}(Y, f) = \int_{\mathsf{H}_Y(F) \bs \mathsf{H}(F)}
    f(h Y \overline{h}^{-1}) \rd h.
    \]
Put
    \[
    \cS(\mathsf{V})_0 = \{ O^{\mathsf{V}}((X_1, X_2), f') = 0 \text{ for all
    regular semisimple $(X_1, X_2) \in \mathsf{V}$
    with $X_1 X_2 \not\in \epsilon N \GL_n(E)$}\}.
    \]

Put also for $(X_1, X_2) \in \mathsf{V}$,
    \[
    \kappa^{\mathsf{V}}(X_1, X_2) = \eta(X_1).
    \]
Two functions $f' \in \cS(\mathsf{V})_0$ and $f \in \cS(\mathsf{W})$ match if
    \[
    \kappa^{\mathsf{V}}(X_1, X_2) O^{\mathsf{V}}((X_1, X_2), f')=
    O^{\mathsf{W}}(Y, f)
    \]
for all matching $(X_1, X_2)$ and $Y$. The following
is~\cite{CZhang}*{Theorem~5.14}.

\begin{prop}    \label{prop:GJ_transfer}
For any $f' \in \cS(\mathsf{V})_0$ there is an $f \in \cS(\mathsf{W})$ that
matches it, and vice versa.
\end{prop}

Let us get back to the sliced representations. Let
    \[
    s' = s'(\alpha, n_1, n_2, n_3) \in S'(F), \quad
    g = g(\beta, n_1, n_2, n_3) \in G(F)
    \]
be matching semisimple elements. Recall that the sliced representation at
$s'$ is isomorphic to
    \[
    (H_1', V_1') \times (H_2', V_2') \times (H_3', V_3'),
    \]
where
    \[
    H_1' = (\GL_{n_1, E})_{\alpha, \text{twisted}},
    \quad H_2' = \GL_{n_2, E}, \quad
    H_3' = \GL_{n_3, F} \times \GL_{n_3, F},
    \]
and
    \[
    V_1' = \{X \in M_{n_1}(E) \mid \alpha \overline{X} =  X \overline{\alpha} \},
    \quad V_2' = M_{n_2}(E), \quad V_3' = M_{n_3}(E)^- \oplus M_{n_3}(E)^-,
    \]
The sliced representation at $g \in G(F)$ is isomorphic to
    \[
    (H_1, V_1) \times (H_2, V_2) \times (H_3, V_3),
    \]
where
    \[
    H_1 = (\GL_{n, E})_{\beta, \text{twisted}}, \quad
    H_2 = \GL_{n_2, E},\quad
    H_3 = \GL_{n_3, E},
    \]
and
    \[
    V_1 = \left\{ Y \in M_{n_1}(E) \mid \beta \overline{Y} =
    Y \overline{\beta} \right\},
    \quad V_2 = M_{n_2}(E),
    \quad V_3 = M_{n_3}(E).
    \]

We can speak of the transfer of functions on each component.

\begin{enumerate}
\item The transfer on $V_1'$ and $V_1$. The group $H_1'$ and $H_1$ act on
    $V_1'$ and $V_1$ respectively by twisted conjugation. Let $L$ be the
    centralizer of $\alpha \overline{\alpha}$ in $\GL_{n_1}(F)$. Since
    $-(1- \alpha \overline{\alpha})( \alpha \overline{\alpha})^{-1}$ and
    $\epsilon \beta \overline{\beta}$ have the same characteristic
    polynomial, it is also the centralize of $\beta \overline{\beta}$. Then
    $H_1'$ and $H_1$ are both inner forms of $L$. The map
        \[
        V_1' \to \fh_1', \quad X \to X \overline{\alpha}
        \]
    is an isomorphism of representations of $H_1'$ where $H_1'$ acts on
    $\fh_1'$ by conjugation. The map
        \[
        V_1 \to \fh_1, \quad Y \mapsto  Y \overline{\beta}
        \]
    is an isomorphism of representations of $H_1$ where $H_1$ acts on
    $\fh_1$ by conjugation. Thus we may speak of semisimple and regular
    semisimple elements in $V_1'$ and $V_1$, and matching between these
    elements. Two semisimple elements $X \in V_1'$ and $Y \in V_1$ match if
    $X \overline{\alpha}$ and $Y \overline{\beta}$ have the same (reduced)
    characteristic polynomial. We define
        \[
        V'_{1, \reg, 0} = \{ X \in V'_{1, \reg} \mid \text{ $X$
        matches some $Y \in V_{1, \reg}$}\},
        \]
    and
        \[
        V_{1, \reg, 0} = \{ Y \in V_{1, \reg} \mid \text{ $Y$
        matches some $X \in V_{1, \reg}'$}\}.
        \]
    For regular semisimple $X \in V'_1$, $Y \in V_1$, and test functions
    $f' \in \cS(V_1')$, $f \in \cS(V_1)$, we may define orbital integral
    $O^{V_1'}(X, f')$ and $O^{V_1}(Y, f)$ as usual. Define
        \[
        \cS(V_{1}')_0 = \{ f' \in \cS(V_1') \mid O^{V_1'}(X, f') =0
        \text{ for all $X \not \in V'_{1, \reg, 0}$}\}
        \]
    and $\cS(V_1)_0$ similarly. Two test functions $f' \in \cS(V_{1}')_0$
    and $f \in \cS(V_1)_0$ match if
        \[
        O^{V_1'}(X, f') = O^{V_1}(Y, f)
        \]
    for all matching regular semisimple $X$ and $Y$. The transfer of test
    functions between inner forms of $\GL_n$ implies that for any $f' \in
    \cS(V_{1}')_0$ there is an $f \in \cS(V_1)_0$ that matches it and vice
    versa.

\item The transfer on $V_2'$ and $V_2$. Both $V_2'$ and $V_2$ are
    isomorphic to $M_{n_2}(E)$. Regular semisimple elements $X \in V_2'$
    and $Y \in V_2$ mean $X \overline{X}$ and $Y \overline{Y}$ are in
    $\GL_{n_2}(E)$ and are regular semisimple in the usual sense. We define
    that they match if $X \overline{X}$ and $-\epsilon Y \overline{Y}$ have
    the same characteristic polynomial. We define
        \[
        V'_{2, \reg, 0} = \{ X \in V'_{2, \reg} \mid \text{ $X$
        matches some $Y \in V_{2, \reg}$}\},
        \]
    and
        \[
        V_{2, \reg, 0} = \{ Y \in V_{2, \reg} \mid \text{ $Y$
        matches some $X \in V_{2, \reg}'$}\}.
        \]
    For regular semisimple $X \in V'_2$, $Y \in V_2$, and test functions
    $f' \in \cS(V_2')$, $f \in \cS(V_2)$, we  define orbital integral
    $O^{V_2'}(X, f')$ and $O^{V_2}(Y, f)$ by
        \[
        O^{V_2'}(X, f') = \int_{H'_{2, X} \bs H_2} f' (h^{-1} X \overline{h})
        \chi(h^{-1} \overline{h}) \rd h,
        \]
    and
        \[
        O^{V_2}(Y, f) = \int_{H_{2, Y} \bs H_2} f (h^{-1} Y
        \overline{h}) \chi(h^{-1} \overline{h}) \rd h.
        \]
    Define
        \[
        \cS(V_{2}')_0 = \{ f' \in \cS(V_2') \mid O^{V_2'}(X, f') =0
        \text{ for all $X \not \in V'_{2, \reg, 0}$}\}
        \]
    and $\cS(V_2)_0$ similarly. Define for regular semisimple $X$ and $Y$
    the transfer factor
        \[
        \kappa^{V_2'}(X) = \chi(X), \quad \kappa^{V_2}(Y) = \chi(Y).
        \]
    Two test functions $f' \in \cS(V_{2}')_0$ and $f \in \cS(V_2)_0$ match
    if
        \[
        \kappa^{V_2'}(X) O^{V_2'}(X, f') = \kappa^{V_2}(Y) O^{V_2}(Y, f)
        \]
    for all matching regular semisimple $X$ and $Y$.

    Let us explain that for any $f' \in \cS(V_{2}')_0$ there is an $f \in
    \cS(V_2)_0$ that matches it and vice versa. First by replacing $f'$ and
    $f$ by $f' \chi$ and $f\chi$, we may assume that $\chi$ is trivial. Let
    $f' \in \cS(V_2')_0$. Then $O^{V_2'}(X, f')$ is the orbital integral
    appearing on the $\mathsf{W}$-side of the transfer problem of Guo and
    Jacquet. Then by Proposition~\ref{prop:GJ_transfer}, applied to
    $\mathsf{W} = V_2'$, there is an $\widetilde{f} \in \cS(\mathsf{V})$
    such that
        \[
        O^{V_2}(X, f') = \kappa^{\mathsf{V}}(Z_1, Z_2)
        O^{\mathsf{V}}((Z_1, Z_2), \widetilde{f}),
        \]
    for all regular semisimple $X \in V_2'$ and $(Z_1, Z_2)$ with $Z_1Z_2 =
    X \overline{X}$. Since $f' \in \cS(V_2')_0$, the function
    $\widetilde{f}$ satisfies the property that $O^{\mathsf{V}}((Z_1, Z_2),
    \widetilde{f}) = 0$ if $Z_1Z_2 \not\in N \GL_n(E)$ or $Z_1Z_2 \not\in
    -\epsilon N \GL_n(E)$.  Now apply Proposition~\ref{prop:GJ_transfer}
    again to $\mathsf{W} = V_2$, we conclude that there is a function $f
    \in \cS(V_2)$ such that
        \[
        \kappa^{\mathsf{V}}(Z_1, Z_2) O^{\mathsf{V}}((Z_1, Z_2),
        \widetilde{f}) = O^{V_2}(Y, f)
        \]
    for all regular semisimple $X \in V_2'$ and $(Z_1, Z_2)$ with $Z_1Z_2 =
    -\epsilon Y \overline{Y}$. The function $f$ satisfies the property that
    $O^{V_2}(Y, f) = 0$ if $-\epsilon Y \overline{Y} \not\in N \GL_n(E)$,
    and hence $f \in \cS(V_2)_0$. It follows that if $X \in V_2'$ matches
    $Y \in V_2$, i.e. $X \overline{X}$ and $-\epsilon Y \overline{Y}$ have
    the same characteristic polynomial, then
        \[
        O^{V_2}(X, f') = O^{V_2}(Y, f).
        \]
    The proof of the converse direction from $f$ to $f'$ is the same.

\item The transfer on $V_3'$ and $V_3$ is exactly the transfer problem of
    Guo and Jacquet. Let
        \[
        V'_{3, \reg, 0} = \{(X_1, X_2) \mid
        X_1 X_2 \in \epsilon N\GL_{n_3}(E)\}.
        \]
    For any $f' \in \cS(V_3')$ have the orbital integral
        \[
        O^{V_3'}((X_1, X_2), f') = \int_{H'_{3, (X_1, X_2)} \bs H'_3}
        f'(h_1^{-1} X_1 h_2, h_2^{-1} X_2 h_1) \eta(h_1 h_2) \rd h_1 \rd h_2.
        \]
    For $f \in \cS(V_3)$ we have the orbital integral
        \[
        O^{V_3}(Y, f) = \int_{H_{3, Y}
        \bs H_3}
        f(h^{-1} Y \overline{h}) \rd h.
        \]
    We also have
        \[
        \cS(V_3')_0 = \{ f' \in \cS(V_3') \mid O((X_1, X_2), f') = 0\
        \text{for all $(X_1, X_2) \not\in V'_{3, \reg, 0}$}\}.
        \]

    The transfer factor on $V_3'$ is given by
        \[
        \kappa^{V_3'}((X_1, X_2)) = \widetilde{\eta}(X_1).
        \]
    By Proposition~\ref{prop:GJ_transfer}, for any $f' \in \cS(V_3')_0$
    there is an $f \in \cS(V_3)$ such that
        \[
        \kappa^{V_3'}((X_1, X_2)) O^{V_3'}((X_1, X_2), f') = O^{V_3}(Y, f)
        \]
    for all matching $(X_1, X_2)$ and $Y$, and vice versa.
\end{enumerate}

We put
    \[
    \cS(N_{s'})_0 = \cS(V_1')_0 \otimes \cS(V_2')_0 \otimes \cS(V_3')_0,
    \quad
    \cS(N_{g})_0 = \cS(V_1)_0 \otimes \cS(V_2)_0 \otimes \cS(V_3).
    \]
Let us define transfer factors on $N_{s'}$ and on $N_g$ by setting
    \[
    \kappa^{N_{s'}}(X_1, X_2, (X_3, X_4)) =
    \chi(X_2) \widetilde{\eta}(X_3),\quad
    \kappa^{N_g}(Y_1, Y_2, Y_3) = \chi(Y_2),
    \]
and define orbital integrals and matching of test functions in $N_{s'}$ and
$N_g$ in the obvious way.

The above discussion can then be summarized as the following proposition.

\begin{prop}    \label{prop:transfer_slice}
For all $f' \in \cS(N_{s'})_0$ there is an $f \in \cS(N_g)_0$ such that for
all matching $X \in N_{s', \reg, 0}$ and $Y \in N_{g, \reg, 0}$, we have
    \[
    \kappa^{N_{s'}}(X) O^{N_{s'}}(X, f') = \kappa^{N_g}(Y) O^{N_g}(Y, f).
    \]
\end{prop}

\subsection{Semisimple descent} \label{subsec:ss_descent}
We recall the semisimple descent of orbital integrals, which is our main tool
in proving Theorem~\ref{thm:transfer}. This is a very general procedure, so
we temporarily consider in this subsection a reductive group $G$ acting on an
affine variety $X$ and $x \in X(F)$ is a semisimple point (i.e. the $G$-orbit
of $x$ is Zariski closed, or equivalently $G(F)x$ is closed in $X(F)$ in the
analytic topology). We will need the notion of an analytic Luna slice,
cf.~\cite{AG2}. The analytic Luna slice at $x$ is denoted by $(U, p, \psi, M,
N_x)$, where
\begin{itemize}
\item $U$ is an $G(F)$-invariant analytic neighbourhood of $x$ in $X(F)$.

\item $p$ is an $G(F)$-equivariant analytic retraction $p: U \to G(F)x$ and
    $M = p^{-1}(x)$.

\item $\psi$ is an $G(F)_x$-equivariant analytic embedding $M \to N_x$ with
    saturated image and $\psi(x) = 0$.
\end{itemize}
Here saturated means that $M = \psi^{-1} (\psi(M))$. Note that since $p$ is
$G(F)$-equivariant, if $y \in M$ and $gy =y$ where $g \in G(F)$, then $gx =
x$, i.e. $G(F)_y$ is a subgroup of $G(F)_x$.

For semisimple $x \in X(F)$, we have an \'etale Lune slice $Z$ and strongly
etale morphisms $\iota: Z \to N_x$ and $\phi: G\times_{G_x} Z \to X$. We can
construct an analytic slice $(U, p, \psi, M, N_{x})$ from it,
cf.~\cite{AG2}*{Corollary~A.2.4}. Let $\pi_Z: Z \to Z//G_x$ be the
categorical quotient. By definition, the morphisms $Z//G_x \to X//G$ and
$Z//G_x \to N_x//G_x$ are both \'{e}tale. Therefore we may choose a
sufficiently small analytic neighbourhood $Z'$ of $\pi_Z(x)$ in
$(Z//G_x)(F)$, so that the above two morphisms are (analytic) isomorphisms
from $Z'$ to its image. Let $M$ be the inverse image of $Z'$ under the
natural map $Z(F) \to (Z//G_x)(F)$. Let $\psi = \iota|_M$. Let $U'$ be the
inverse image of $Z'$ in $(G \times_{G_x} Z)(F)$. Let $p': U' \to
(G//G_x)(F)$ be the natural $G(F)$-invariant map. Let $U'' = U'\cap
p'^{-1}(G(F)//G(F)_x)$ and $U = \phi(U'') \subset X(F)$. Note that
$\phi|_{U''}$ is an analytic isomorphism from $U''$ to $U$. Let $p = p' \circ
(\phi|_{U''})^{-1}$. Then $(U, p, \psi, M, N_x)$ is the desired analytic Luna
slice at $x$.

The following is~\cite{Zhang1}*{Proposition~3.11}. It describes the relation
between the orbital integrals on $X$ near the semisimple point $x$ and the
orbital integrals on the sliced representation at $x$.

\begin{prop}    \label{prop:ss_descent}
Let $\chi$ be a character of $G(F)$. There is an $H_x$-invariant
neighbourhood $\cV \subset M$ of $x$ such that the following holds.
\begin{enumerate}
\item For any $f \in \cS(X(F))$ there is an $f_x \in \cS(N_x)$ such that if
    $y \in \cV$, $z = \psi(y)$ and $\chi$ is trivial on $H(F)_y$, we have
        \begin{equation}    \label{eq:ss_descent}
        \int_{G(F)/G_y(F)} f(gy) \chi(g) \rd g  =
        \int_{G_x(F)/G_y(F)} f_x(gz) \chi(g) \rd g.
        \end{equation}

\item Conversely for any $f_x \in \cS(N_x)$ there is an $f \in \cS(X(F))$
    such that the equality~\eqref{eq:ss_descent} holds if $y \in \cV$, $z =
    \psi(y)$ and $\chi$ is trivial on $H(F)_y$.
\end{enumerate}
\end{prop}

\subsection{Proof of Theorem~\ref{thm:transfer}}
Let us retain the setup of Theorem~\ref{thm:transfer}. Let $s' = s'(\alpha,
n_1, n_2, n_3) \in S'(F)$ and $g = g(\beta, n_1, n_2, n_3) \in G(F)$ be
semisimple elements.

\begin{lemma}
If $g$ does not match any semisimple element in $S'(F)$, then there is an $(H
\times H)(F)$-invariant neighbourhood $U$ of $g$ such that $U \cap
S(F)_{\reg, 0} = \emptyset$. If $s'$ does not match any semisimple element in
$G(F)$, then there is an $H'(F)$-invariant neighbourhood $U'$ of $s'$ such
that $U' \cap S'(F)_{\reg, 0} = \emptyset$.
\end{lemma}

\begin{proof}
That $g$ does not match any semisimple element in $S'(F)$ is equivalent to
    \[
    1 - \epsilon \beta \overline{\beta} \not\in N \GL_{n_1}(E).
    \]
Any regular semisimple element in a neighbourhood of $g$ is in the $(H \times
H)(F)$ orbit of an element $y = (y_1, y_2, y_3) \in U$. We can choose the
neighbourhood to be small enough such that if $y_1 \theta(y_1)^{-1} =
\begin{pmatrix} A & \epsilon B \\ \overline{B} & \overline{A} \end{pmatrix}$
then $\frac{1}{2}(A+1)$ is in a small neighbourhood of $(1-\epsilon \beta
\overline{\beta})^{-1}$ in $\GL_{n_1}(E)$. When this neighbourhood is small
enough, no regular semisimple elements (in the usual sense) in it is a norm
(see below for an explanation). This shows that $y = (y_1, y_2, y_3)$ does
not match any element in $S'(F)$.

The case of $s'$ can be proved by the same argument.

It remains to explain that if $\gamma \in \GL_{n_1}(F)$ is semisimple in the
usual sense and $\gamma \not\in N\GL_{n_1}(E)$, then there is a neighbourhood
of $\gamma$ such that any regular semisimple $\delta$ in the neighbourhood is
not a norm. This can be seen as follows. Suppose that there is a sequence of
regular semisimple elements $\delta_k$, such that $\delta_k$ converges to
$\gamma$ and $\delta_k = N g_k$ for some $g_k \in\GL_{n_1}(E)$. As
$\delta_k$'s are all conjugate to an element in $\GL_{n_1}(F)$, we may assume
that $\delta_k \in \GL_{n_1}(F)$. The conjugation action map $\GL_{n_1}(F)
\times \GL_{n_1}(F)_\gamma \to \GL_{n_1}(F)$, $(x, y) \mapsto xyx^{-1}$, is a
fibration in a neighbourhood of $\gamma$. Therefore we may further assume
that $\delta_k \in \GL_{n_1}(F)_{\gamma}$. Moreover as there are only
finitely maximal tori in $\GL_{n_1}(F)_{\gamma}$ up to conjugation, we may
assume, by taking a subsequence, that there is a maximal torus $T$ of
$\GL_{n_1}(F)_{\gamma}$, such that $\delta_k \in T(F)$. As $\delta_k$ are all
regular semisimple, we conclude that $g_k \in T(E)$. Since the norm map $T(E)
\to T(F)$ is locally a fibration, we may assume, again by taking a
subsequence, that $g_k$ is convergent to an element $g$. Then $\gamma = N g$.
\end{proof}

From now on let us assume that $g$ and $s'$ match. This in particular implies
that $-(1-\alpha \overline{\alpha})(\alpha \overline{\alpha})^{-1}$ and
$\epsilon \beta \overline{\beta}$ have the same characteristic polynomial. We
may and will assume that they in fact equal. Put
    \[
    s_1' = \begin{pmatrix} \alpha & 1_{n_1}\\ 1_{n_1}- \alpha \overline{\alpha}
    & -\overline{\alpha} \end{pmatrix}, \quad
    s_2' = \begin{pmatrix} & 1_{n_2} \\ 1_{n_2} \end{pmatrix},\quad
    s_3' = 1_{2n_3},
    \]
and
    \[
    g_1 = \begin{pmatrix} 1_{n_1} & \epsilon \beta \\
    \overline{\beta} & 1_{n_1} \end{pmatrix}, \quad
    g_2 = \begin{pmatrix} & 1_{n_2} \\ 1_{n_2} \end{pmatrix}, \quad
    g_3 = 1_{2n_3}.
    \]

The stabilizer of $s$ (resp. $s'$) has the form $H_1 \times H_2 \times H_3$
(resp. $H_1' \times H_2' \times H_3'$), and the sliced representation $N_s$
(resp. $N_{s'}$) has the form $V_1 \times V_2 \times V_3$ (resp. $V_1' \times
V_2' \times V_3'$). We have $V_i//H_i \simeq V_i'//H_i'$. By the explicit
construction in Subsection~\ref{subsec:etale_Luna}, the \'etale Luna slice at
$g$ (resp. $s'$) takes the form $Z = Z_1 \times Z_2 \times Z_3$ (resp. $Z' =
Z_1' \times Z_2' \times Z_3'$). The analytic slice at $g$ (resp. $s'$) is
denoted by $(U, p, \psi, M, N_g)$ (resp. $(U', p', \psi', M', N_{s'})$).
According to the construction of the analytic slice recalled in
Subsection~\ref{subsec:ss_descent}, $M$ (resp. $M'$) takes the form $M_1
\times M_2 \times M_3$ (resp. $M_1' \times M_2' \times M_3'$) where $M_i
\subset Z_i$ (resp. $M_i' \subset Z_i'$), where $\psi'(M_i)$ (resp.
$\psi'(M_i')$) is a saturated open subset of $V_i$ (resp. $V_i'$). We may
assume that the image of $M_i$ (resp. $M_i'$) in $V_i
//H_i = V_i'//H_i'$ are identified.

Let us assume that $xs' = (x_1 s_1', x_2 s_2', x_3 s_3') \in M'$ be a regular
semisimple element, $x_i s_i' \in M_i$, and
    \[
    \psi'(xs') = \left( \begin{pmatrix} & X_1 \\ -\overline{X_1}
    (1- \alpha \overline{\alpha})\end{pmatrix}, \
    \begin{pmatrix} & X_2 \\ -\overline{X_2} \end{pmatrix},\
    \begin{pmatrix} &X_3 \\ X_4 \end{pmatrix} \right),
    \]
where $X_1 \in V_1'$, $X_2 \in V_2'$ and $(X_3, X_4) \in V_3'$. Write $X =
(X_1, X_2, (X_3, X_4)) \in N_{s'}$. Let $yg = (y_1 g_1, y_2 g_2, y_3 g_3) \in
M$ be a regular semisimple elements, $y_i g_i \in M_i$, and
    \[
    \psi(yg) = \left(
    \begin{pmatrix} & \epsilon Y_1 \\ \overline{Y_1} \end{pmatrix},\
    \begin{pmatrix} & \epsilon Y_2 \\ \overline{Y_2} \end{pmatrix},\
    \begin{pmatrix} & Y_3 \\ \overline{Y_3} \end{pmatrix} \right),
    \]
where $Y_i \in V_i$. Write $Y = (Y_1, Y_2, Y_3) \in N_g$.

The next lemma connects the notion of matching of $xs'$ and $yg$ with the
matching of $X$ and $Y$ in the sliced representation. We use $a \sim b$ to
indicate that $a$ and $b$ have the same characteristic polynomial. We also
note that $X_1 \overline{\alpha}$ and $\alpha \overline{\alpha}$ commute, and
hence it makes sense to evaluate a power series with coefficients being
rational functions of $\alpha \overline{\alpha}$ at $X_1 \overline{\alpha}$.
Similarly it makes sense to evaluate a power series with coefficients being
rational functions of $\epsilon \beta \overline{\beta}$ at $Y_1
\overline{\beta}$.

\begin{lemma}   \label{lemma:matching_sliced}
We can find an $H'_i$-invariant neighbourhood $\cV_i'$ of $s_i'$ in $M_i'$
and a neighbourhood $\cV_i$ of $g_i$ in $M_i$, $i = 1, 2,3$, and a power
series $\xi_1$ with coefficients being rational functions of $\alpha
\overline{\alpha}$, and the leading term
    \[
    \xi_1(t) =t (1-\alpha \overline{\alpha})(\alpha \overline{\alpha})^{-1}
    + \cdots,
    \]
such that if $xs'$ and $yg$ match, and $x_is_i' \in\cV_i'$, $y_i g_i \in
\cV_i$, $i = 1, 2, 3$, then
    \[
    \xi_1(X_1 \overline{\alpha}) \sim Y_1 \overline{\beta},
    \quad X_2 \overline{X_2} \sim -\epsilon Y_2 \overline{Y_2},
    \quad X_3 \overline{X_3} \sim \epsilon Y_3 \overline{Y_3}.
    \]
\end{lemma}

\begin{proof}
We always assume that $x$ and $y$ lie in a small neighbourhoods of $1$, or
equivalently $X_i$'s and $Y_i$'s are sufficiently close to $0$. Elementary
but tedious calculations gives
    \[
    x_1 s_1' = \begin{pmatrix} A &B \\ -\overline{B}(1-\alpha \overline{\alpha})
    & \overline{A} \end{pmatrix}
    \begin{pmatrix} \alpha & 1 \\ 1- \alpha \overline{\alpha}
    & - \overline{\alpha} \end{pmatrix}
    \]
where
    \[
    A = (1+X_1 \overline{X_1}(1-\alpha \overline{\alpha}))^{-1}
    (1 - X_1 \overline{X_1}(1-\alpha \overline{\alpha})),\quad
    B = -2(1+X_1 \overline{X_1}(1-\alpha \overline{\alpha}))^{-1} X_1
    \]
and
    \[
    x_2 s_2' = \begin{pmatrix}
    -2 (1+X_2\overline{X_2})^{-1} X_2&
    (1+X_2\overline{X_2})^{-1} (1-X_2\overline{X_2}) \\
    (1+X_2\overline{X_2})^{-1} (1-X_2\overline{X_2}) &
    2 (1+X_2\overline{X_2})^{-1} X_2 \end{pmatrix},
    \]
and
    \[
    x_3 s_3' = \begin{pmatrix} (1-X_3X_4)^{-1}(1+X_3X_4) &
    -2(1-X_3X_4)^{-1}X_3 \\
    -2 (1-X_3X_4)^{-1} X_4 &
    (1-X_3X_4)^{-1}(1+X_3X_4)
    \end{pmatrix}.
    \]

Similarly we have
    \[
    y_1 g_1 = \begin{pmatrix}
    (1-\epsilon Y_1 \overline{Y_1})^{-1}(1+\epsilon Y_1 \overline{Y_1}) &
    - 2\epsilon (1-\epsilon Y_1 \overline{Y_1})^{-1} Y_1\\
    -2 \overline{Y_1} (1 - \epsilon Y_1 \overline{Y_1})^{-1} &
    (1-\epsilon \overline{Y_1} Y_1)^{-1}(1+\epsilon \overline{Y_1} Y_1)
    \end{pmatrix}
    \begin{pmatrix} 1 & \epsilon \beta \\
    \overline{\beta} & 1 \end{pmatrix}
    \]
and
    \[
    y_2 g_2 = \begin{pmatrix}
    -2\epsilon (1-\epsilon Y_2 \overline{Y_2})^{-1} Y_2
    & \epsilon
    (1-\epsilon Y_2 \overline{Y_2})^{-1}(1+\epsilon Y_2 \overline{Y_2})
    \\
    (1-\epsilon \overline{Y_2} Y_2 )^{-1}(1+\epsilon \overline{Y_2}Y_2)
    & -2\epsilon (1- \epsilon \overline{Y_2}Y_2 )^{-1}\overline{Y_2}

    \end{pmatrix}
    \]
and
    \[
    y_3 g_3 = \begin{pmatrix}
    (1-\epsilon Y_3 \overline{Y_3})^{-1}(1+\epsilon Y_3 \overline{Y_3})&
    -2\epsilon (1-\epsilon Y_3 \overline{Y_3})^{-1} Y_3\\
    -2 \overline{Y_3}(1-\epsilon Y_3 \overline{Y_3})^{-1} &
    (1-\epsilon \overline{Y_3} Y_3 )^{-1}(1+\epsilon \overline{Y_3}Y_3 )
    \end{pmatrix}
    \]

From these expression it is straightforward to check that if $xs'$ and $yg$
match, and when $X_2, X_3, X_4$ and $Y_2$, $Y_3$ are close to $0$, then
    \[
    X_2 \overline{X_2} \sim -\epsilon Y_2 \overline{Y_2},
    \quad X_3 X_4 \sim \epsilon Y_3 \overline{Y_3}.
    \]

It remains to treat $X_1$ and $Y_1$. Write
    \[
    x_1 s_1' = \begin{pmatrix} A_1' & * \\ * & * \end{pmatrix}.
    \]
A little computation gives $2 A_1 \overline{A_1} - 1$ is a rational function
in $X_1 \overline{\alpha}$ and $\alpha \overline{\alpha}$. Its power series
expansion (in the variable $t = X_1 \overline{\alpha}$) is $\xi_1'(X_1
\overline{\alpha})$ where $\xi_1'$  with coefficients in rational functions
of $\alpha \overline{\alpha}$ and the first few terms are
    \[
    2\alpha \overline{\alpha} - 1
    - 8 t (1- \alpha \overline{\alpha}) + \cdots
    \]
The power series is convergent when $X_1 \overline{\alpha}$ is sufficiently
close to zero.

Write
    \[
    y_1 g_1 \theta(y_1 g_1)^{-1} = \begin{pmatrix}
    A_1 & * \\ * &* \end{pmatrix}.
    \]
A little computation gives $A_1$ is a rational function in $\epsilon \beta
\overline{\beta}$ and $Y_1 \overline{\beta}$. Its power series expansion (in
the variable $t = Y_1 \overline{\beta}$ is $\widetilde{\xi_1}(Y_1
\overline{\beta})$ where $\xi_1$ is a power series whose coefficients are
rational functions in $\epsilon \beta \overline{\beta}$ and the first few
terms are
    \[
    (1+ \epsilon \beta \overline{\beta})
    (1 - \epsilon \beta \overline{\beta})^{-1} -
    8 t (1 - \epsilon \beta \overline{\beta})^{-1}
     + \cdots
    \]
The power series is convergent when $Y_1 \overline{\beta}$ is sufficiently
close to zero.

That $xs'$ and $yg$ match implies that $\xi_1'(X_1 \overline{\alpha}) \sim
\widetilde{\xi_1}(Y_1 \overline{\beta})$. By assumption we have $-(1- \alpha
\overline{\alpha}) (\alpha \overline{\alpha})^{-1} = \epsilon \beta
\overline{\beta}$ and $\det (1 - \alpha \overline{\alpha}) \not=0$. This
implies that the constant terms of $\xi_1'$ and $\widetilde{\xi_1}$ equal. So
there is a power series $\xi_1$ with coefficients in $\alpha
\overline{\alpha}$ of the form
    \[
    t (1 - \alpha \overline{\alpha})^{-1}(\alpha \overline{\alpha}) + \cdots,
    \]
such that $\xi_1(X_1 \overline{\alpha}) \sim Y_1 \overline{\beta}$.
\end{proof}

\begin{proof}[Proof of Theorem~\ref{thm:transfer}]
Let $f' \in \cS(S'(F))_0$. We prove that there is an $f \in \cS(G(F))_0$ that
matches it. The other direction is similar. By
lemma~\ref{lemma:local_transfer}, it is enough to prove such an $f$ exists
locally near every semisimple point $s'$. Moreover precisely, we will prove
that for any semisimple point $s' \in S'(F)$, there is an $f \in \cS(G(F))_0$
such that
    \begin{equation}    \label{eq:local_transfer}
    \kappa^{S'}(xs') O^{S'}(xs', f') = \kappa^G(yg) O^G(yg, f)
    \end{equation}
for all matching regular semisimple $xs' \in S'(F)$ and $yg \in G(F)$ when
$x$ and $y$ are sufficiently close to $1$.

We use semisimple descent at the point $s'$. First according to the
computation in the proof of Lemma~\ref{lemma:matching_sliced} when $X =
\psi'(xs') \in N_{s'}$ is sufficiently close to $0$, we have that
$\kappa^{S'}(xs')$ is a constant multiple of $\kappa^{N_{s'}}(X)$.

Let us now apply Proposition~\ref{prop:ss_descent} to the orbital integrals
$O^{S'}(xs', f')$. There is a enough $H_i'$-invariant neighbourhood $\cV_i'
\subset M_i'$ of $0$, and a function $f'_{s'} \in \cS(\psi'(\cV'_1 \times
\cV'_2 \times \cV'_3))$ such that for all $xs' \in \cV'_1 \times \cV'_2
\times \cV'_3$, we have
    \[
    \kappa^{S'}(xs') O^{S'}(xs', f') =
    \kappa^{N_{s'}}(X) O^{N_{s'}}(X, f'_{s'}).
    \]
We put $\cU'_i = \psi_i'(\cV_i')$, $i = 1, 2, 3$, $\cV' = \cV'_1 \times
\cV'_2 \times \cV'_3$, $\cU'  = \psi'(\cV') = \cU'_1 \times \cU'_2 \times
\cU'_3$.

For $i = 1, 2, 3$, by shrinking $\cV_i$, we may assume that the neighbourhood
$\cV'_i$ such that Lemma~\ref{lemma:matching_sliced} holds. We can find a
power series $\xi_1$ with coefficients in rational functions of $\alpha
\overline{\alpha}$ such that if $xs'$ and $yg$ match and $x_i s_i' \in
\cV'_i$, $i = 1, 2, 3$, then
    \[
    \xi_1(X_1 \overline{\alpha}) \sim Y_1 \overline{\beta}, \quad
    X_2 \overline{X_2} \sim -\epsilon Y_2 \overline{Y_2}, \quad
    X_3 X_4 \sim \epsilon Y_3 \overline{Y_3}.
    \]

Shrinking $\cV_1'$ further if necessary, we may assume that $X_1 \mapsto
\xi_1(X_1 \overline{\alpha}) \overline{\alpha}^{-1}$ defines a bijection from
$\cU_1$ to its image in $V'_1$. Let $\xi: \cU' \to N_{s'}$ be the map
    \[
    \xi(X_1, X_2, (X_3, X_4)) = (\xi_1(X_1 \overline{\alpha}) \overline{\alpha}^{-1},
    X_2, (X_3, X_4)).
    \]
Note that this map is $H'_{s'}$-equivariant. Then $\xi(\cU')$ is a
neighbourhood of $0 \in N_{s'}$ and $\xi$ is a bijection from $\cU'$ to its
image. Moreover $\kappa^{N_{s'}}(X) = a \kappa^{N_{s'}}(\xi(X))$ if $X \in
\cU'$, where $a$ is a nonzero constant.

Define a function $\widetilde{f'_{s'}}$
    \[
    \widetilde{f'_{s'}}(X) = \left\{ \begin{aligned} &
    a f'_{s'}(\xi^{-1}(X)),
    && X \in \xi(\cU'), \\
    & 0, && X \not\in \xi(\cU'). \end{aligned} \right.
    \]
Then we have
    \[
    \kappa^{N_{s'}}(X) O^{N_{s'}}(X, f'_{s'}) =
    \kappa^{N_{s'}}(\xi'(X))
    O^{N_{s'}}(\xi'(X), \widetilde{f'_{s'}}),
    \]
for all $X \in \xi(\cU')$. We view $f_{s'}'$ as an function on $N_{s'}$ via
extension by zero. Since $f' \in \cS(S'(F))_{0}$, we conclude that
$\widetilde{f'_{s'}} \in \cS(N_{s'})_0$.

By Proposition~\ref{prop:transfer_slice} there is an $f_g \in \cS(N_{g})_0$
that matches $\widetilde{f'_{s'}}$, i.e.
    \[
    \kappa^{N_{s'}}(\xi(X))
    O^{N_{s'}}(\xi(X), \widetilde{f'_{s'}}) =
    \kappa^{N_g}(Y)
    O^{N_g}(Y, f_g),
    \]
for all regular semisimple matching $\xi(X) \in \xi(\cU')$ and $Y \in N_g$.

As in the case of $S'$, when $Y$ is sufficiently close to $0$ in $N_g$, we
have $\kappa^{G}(yg)$ equals a constant multiple of $\kappa^{N_g}(Y)$.
Applying Proposition~\ref{prop:ss_descent} in the converse direction, we
conclude that there is an $H_{s}$-invariant neighbourhood $\cV$ of $g$ in
$N_g$, and an $f \in \cS(G(F))$ such that
    \[
    \kappa^{N_g}(Y) O^{N_s}(Y, f_g) =
    \kappa^G(yg) O^S(yg, f)
    \]
for all $yg \in \cV$. Since $f_g \in \cS(N_g)_0$ we conclude that $f \in
\cS(G(F))_0$. Shrinking $\cV$ if necessary, we may assume that if $yg \in
\cV$ and $xs' \in S'(F)$ match then $xs' \in \cV'$. Then we conclude a chain
of equalities
    \[
    \begin{aligned}
    \kappa^G(yg) O^S(yg, f) &= \kappa^{N_g}(Y) O^{N_s}(Y, f_g)
    = \kappa^{N_{s'}}(\xi(X))
    O^{N_{s'}}(\xi(X), \widetilde{f'_{s'}})\\
    &= \kappa^{N_{s'}}(X) O^{N_{s'}}(X, f'_{s'})
    = \kappa^{S'}(xs') O^{S'}(xs', f')
    \end{aligned}
    \]
when $xs' \in \cV'$ and $yg \in \cV$ match. This
proves~\eqref{eq:local_transfer} and thus Theorem~\ref{thm:transfer}.
\end{proof}

\section{The fundamental lemma} \label{sec:FL}

\subsection{The fundamental lemma}
Assume that $E/F$ is unramified with odd residue characteristic. Let
$\widetilde{\eta}(x) = (-1)^{\val_E(x)}$ be the unique unramified character
that extends $\eta$ to $E^\times$. Let $\tau \in \fo_E^\times$ be a purely
imaginary element which is used in the definition of the transfer factor
$\kappa^{S'}$ (cf.  \eqref{eq:tran fact S'}).

\begin{theorem} \label{thm:FL}
Let $s' = s'(\alpha) \in S'(F)$ and $g = g(\beta) \in G(F)$ be matching
regular semisimple elements. Then
    \begin{equation}    \label{eq:FL}
    \kappa^{S'}(s') O^{S'}(s', \id_{S'(\fo_F)}) =
    \kappa^G(g) O^{G}(g, \id_{G(\fo_F)}).
    \end{equation}
\end{theorem}

It is straightforward to see that this theorem implies
Theorem~\ref{thm:FL_group}. The rest of this section is devoted to prove this
theorem. We distinguish two cases depending on $\overline{\alpha}\alpha$
being elliptic or not. The elliptic case will be deduced from the base change
fundamental lemma and the calculations in~\cite{Guo2}. The nonelliptic case
will be treated using the parabolic descent and induction on $n$.

Before we prove the theorem, let us first explain that the validity
of~\eqref{eq:FL} is independent of the character $\chi$. In fact, the left
hand side equals
    \[
    \chi(-\overline{\alpha}) \int \id_{S'(\fo_F)} \left(
    \begin{pmatrix} h_1^{-1} \\ & h_2^{-1} \end{pmatrix}
    \begin{pmatrix} \alpha & 1 \\ 1 - \alpha \overline{\alpha}
    & - \overline{\alpha} \end{pmatrix}
    \begin{pmatrix} \overline{h_1} \\ & \overline{h_2} \end{pmatrix} \right)
    \chi(\overline{h_2} h_2^{-1}) \widetilde{\eta}(h_1 h_2) \rd h_1 \rd h_2.
    \]
Since $\det \overline{h_2} h_2^{-1} \in \fo_E^\times$ and $\chi$ is
unramified, the above expression equals
    \[
    \chi(-\overline{\alpha}) \int \id_{S'(\fo_F)} \left(
    \begin{pmatrix} h_1^{-1} \\ & h_2^{-1} \end{pmatrix}
    \begin{pmatrix} \alpha & 1 \\ 1 - \alpha \overline{\alpha}
    & - \overline{\alpha} \end{pmatrix}
    \begin{pmatrix} \overline{h_1} \\ & \overline{h_2} \end{pmatrix} \right)
    \widetilde{\eta}(h_1 h_2) \rd h_1 \rd h_2.
    \]
The right hand side of~\eqref{eq:FL} equals
    \[
    \int \id_{G(\fo_F)} \left(
    \begin{pmatrix} g^{-1} \\ & \overline{g}^{-1} \end{pmatrix}
    \begin{pmatrix} 1 & \beta \\ \overline{\beta} & 1 \end{pmatrix}
    \begin{pmatrix} h \\ & \overline{h} \end{pmatrix}
    \right) \chi(g^{-1} h) \rd g \rd h.
    \]
Taking determinant we see that
    \[
    \det(g \overline{g})^{-1} \det (1- \beta\overline{\beta})
    \det (h \overline{h}) \in \fo_F^\times.
    \]
Recall that $s'(\alpha)$ and $g(\beta)$ match if $-(1-\alpha
\overline{\alpha})(\alpha \overline{\alpha})^{-1}$ and $\beta
\overline{\beta}$ have the same characteristic polynomial, which implies that
$(1 - \beta \overline{\beta})^{-1}$ and $\alpha \overline{\alpha}$ have the
same characteristic polynomial. As $\chi$ is unramified we obtain
$\chi(g^{-1} h) = \chi(\alpha) = \chi(-\overline{\alpha})$. Thus the right
hand side of~\eqref{eq:FL} equals
    \[
    \chi(- \overline{\alpha}) \int \id_{G(\fo_F)} \left(
    \begin{pmatrix} g^{-1} \\ & \overline{g}^{-1} \end{pmatrix}
    \begin{pmatrix} 1 & \beta \\ \overline{\beta} & 1 \end{pmatrix}
    \begin{pmatrix} h \\ & \overline{h} \end{pmatrix}
    \right) \rd g \rd h.
    \]

Thus from now on we assume that $\chi$ is trivial. Under this assumption the
transfer factors on both sides of~\eqref{eq:FL} are trivial.

\subsection{The elliptic case}
Put $r =-(1-  \alpha \overline{\alpha}) (\alpha \overline{\alpha})^{-1}$ and
    \[
    x_r = \abs{\det (1-r)} = \abs{\det \alpha \overline{\alpha}}^{-1}, \quad
    y_r = \abs{\det r} = \abs{ \det (1-\alpha \overline{\alpha})
    (\alpha \overline{\alpha})^{-1}}.
    \]
In this subsection we always assume that $\alpha \overline{\alpha}$ and hence
$r$ are elliptic regular semisimple in $\GL_n(F)$ in the usual sense. Note
that this in particular covers the base case of the induction $n = 1$.

Make a change of variable $h_1 \mapsto \overline{h_2} h_1$ on the left hand
side of~\eqref{eq:FL} we obtain
    \begin{equation}    \label{eq:LHS_FL}
    \int
    \id_{S'(\fo_F)} \left( \begin{pmatrix} h_1^{-1} \\ & 1  \end{pmatrix}
    \begin{pmatrix} \overline{h_2}^{-1} \alpha h_2 & 1 \\
    h_2^{-1} (1 - \overline{\alpha} \alpha)h_2
    & - h_2^{-1} \overline{\alpha} \overline{h_2} \end{pmatrix}
    \begin{pmatrix} \overline{h_1} \\ & 1 \end{pmatrix}
    \right) \widetilde{\eta}(\det h_1)  \rd h_1  \rd h_2.
    \end{equation}
Here the integration is over $h_2 \in \GL_n(E)_{\alpha, \mathrm{twisted}} \bs
\GL_n(E)$ and $h_1 \in \GL_n(E)$.

The first observation is that if $x_r < 1$, i.e. $\abs{\det \alpha}
>1$, then $h_2^{-1} \alpha \overline{h_2} \not\in M_n(\fo_E)$ for any $h_2$
and therefore the
integral vanishes. Moreover under this condition, the right hand side
of~\eqref{eq:FL} also vanishes by~\cite{Guo2}*{(3.9)}. Thus
Theorem~\ref{thm:FL} holds in this case.

Assume from now on that $x_r \geq 1$. By~\cite{Guo2}*{Lemma~3.4}, either
$y_r\leq x_r = 1$ or $x_r = y_r >1$. We distinguish two cases.

\begin{enumerate}
\item $y_r \leq x_r = 1$, i.e. $\det \alpha \in \fo_E^\times$. Then $y_r =
    \abs{\det (1- \overline{\alpha}\alpha)}$.

    We need to make use of a lemma of Kottwitz's,
    cf.~\cite{Kottwitz}*{Lemma~8.8}, which we recall in a special case for
    readers' convenience. The lemma says that if $\gamma \in \GL_n(E)$ is
    conjugate to an element in $\GL_n(\fo_E)$ and is regular semisimple,
    then we can find an element $\delta \in \GL_n(E)$ such that $\gamma =
    \delta \overline{\delta} = \overline{\delta} \delta$ and $x^{-1} \gamma
    x \in \GL_n(\fo_E)$ implies $x^{-1} \delta x \in \GL_n(\fo_E)$.
    Moreover if $\delta$ satisfies the above property and $h \in \GL_n(E)$,
    then $h^{-1} \overline{\delta} \overline{h} \in \GL_n(\fo_E)$ if and
    only if $h^{-1} \gamma h \in \GL_n(\fo_E)$ and $h \in \GL_n(F)
    \GL_n(\fo_E)$.

    Let us come back to our setup. If $\alpha \overline{\alpha}$ is not
    conjugate to an element in $\GL_{n}(\fo_F)$, by considering the lower
    right block of the matrix, we see that~\eqref{eq:LHS_FL} equals zero.
    Let us assume that $\alpha \overline{\alpha}$ is conjugate to an
    element in $\GL_{n}(\fo_F)$. We first apply this lemma to $\gamma =
    \alpha\overline{\alpha}$. We may further assume that $\alpha = \delta$
    has the property described in the lemma of Kottwitz's. We then conclude
    that $h_2 \in \GL_n(F) \GL(\fo_E)$. Thus we may replace the outer
    integral in~\eqref{eq:LHS_FL} by $h_2 \in \GL_n(F)_\alpha \bs
    \GL_n(F)$. Here though $\alpha$ might not be in $\GL_n(F)$ but in
    $\GL_n(E)$, the group $\GL_n(F)_{\alpha}$ stands for all $h \in
    \GL_n(F)$ that commutes with $\alpha$. We note that $\GL_n(F)_{\alpha}
    = \GL_{n}(F)_{\alpha \overline{\alpha}}$. Indeed as explained
    in~\cite{AC}*{Chapter~1, Proof of Lemma~1.1}, as algebraic groups over
    $F$, the group $\GL_n(E)_{\alpha, \mathrm{twisted}}$ is an inner form
    of $\GL_n(F)_{\alpha \overline{\alpha}}$. Since $\alpha
    \overline{\alpha}$ is regular semisimple, both are tori over $F$ and
    hence are canonically isomorphic, which implies that $\GL_n(E)_{\alpha,
    \mathrm{twisted}} = \GL_n(F)_{\alpha \overline{\alpha}}$. This in
    particular implies that if $h \in \GL_n(E)_{\alpha, \mathrm{twisted}}$
    then $h \in \GL_n(F)$ and thus $\GL_n(F)_{\alpha} = \GL_n(E)_{\alpha,
    \mathrm{twisted}} = \GL_n(F)_{\alpha \overline{\alpha}}$. Therefore we
    may replace the outer integral by $h_2 \in \GL_n(F)_{\alpha
    \overline{\alpha}} \bs \GL_n(F)$.

    Now $h_2 \in \GL_n(F)$ and we apply the lemma again to $\gamma =
    h_2^{-1} \alpha \overline{\alpha}  h_2$, and $\delta = h_2^{-1}
    \overline{\alpha} h_2$. Clearly this $\gamma$ and $\delta$ again
    satisfy the conditions in the lemma of Kottwitz's. By consider the
    upper left corner of the matrix in the integrand, we conclude that $h_1
    \in \GL_n(F) \GL(\fo_E)$ and we may replace the outer integral
    in~\eqref{eq:LHS_FL} by $h_1 \in \GL_n(F)$.

    The integrand in~\eqref{eq:LHS_FL} is thus equivalent to the four
    conditions
        \[
        h_1^{-1} h_2^{-1} \alpha h_2 h_1,\ h_2^{-1} \alpha h_2 \in \GL_n(\fo_E),
        \quad h_1^{-1},\
        h_2^{-1}(1-\alpha \overline{\alpha}) h_2 h_1 \in M_n(\fo_F).
        \]
    As $\alpha$ satisfies the conditions in Kottwitz's lemma, the first two
    conditions are implied by the other two. To see this, observe that
    $h_1^{-1}$ and $h_2^{-1}(1-\overline{\alpha} \alpha) h_2 h_1$ being in
    $M_n(\fo_F)$ implies both
        \[
        h_2^{-1}(1-\alpha\overline{\alpha} ) h_2,
        \quad h_1^{-1} h_2^{-1}(1- \alpha\overline{\alpha}) h_2 h_1
        \]
    are in $M_n(\fo_F)$. As we have assume that $\det \alpha \in
    \fo_E^\times$, we conclude that
        \[
        h_2^{-1}  \alpha \overline{\alpha} h_2,
        \quad h_1^{-1} h_2^{-1} \alpha \overline{\alpha} h_2 h_1
        \]
    are in $\GL_n(\fo_F)$. That $\alpha$ satisfies the conditions in the
    lemma of Kottwitz's implies
        \[
        h_1^{-1} h_2^{-1} \alpha h_2 h_1,\quad h_2^{-1} \alpha h_2
        \]
    are both in $\GL_n(\fo_E)$.

    The integral~\eqref{eq:LHS_FL} thus simplifies to
        \[
        \int
        \id_{M_n(\fo_F) \times M_n(\fo_F)}
        (h_2^{-1} (1-\alpha\overline{\alpha} )
        (\alpha\overline{\alpha} )^{-1} h_2 h_1, h_1^{-1})
        \eta (\det h_1)  \rd h_1 \rd h_2,
        \]
    where the domain of integration is $h_1 \in \GL_n(F)$ and $h_2 \in
    \GL_n(F)_{\alpha \overline{\alpha}}  \bs \GL_n(F)$.

    In~\cite{Guo2}*{Lemma~3.5}, a Hecke function $\Psi_r$ on $\GL_n(F)$ is
    defined. By the calculation in~\cite{Guo2}*{p.~137}, under the
    assumption that $y_r \leq x_r = 1$, this function equals
        \[
        g \mapsto \int_{\GL_n(F)}
        \id_{\{ X, Y \in M_n(\fo_F), \  \abs{\det XY} = y_r \}}
        (g h_1, h_1^{-1}) \eta(h_1) \rd h_1.
        \]
    Therefore~\eqref{eq:LHS_FL} equals
        \begin{equation}    \label{eq:FL_LHS}
        \int_{\GL_n(F)_{\alpha \overline{\alpha}} \bs \GL_n(F)}
        \Psi_r(h_2^{-1} r h_2) \rd h_2.
        \end{equation}

    The right hand side of~\eqref{eq:FL} is exactly the orbital integral
    appearing in~\cite{Guo2}. For any (twisted) elliptic regular semisimple
    $\beta \in \GL_n(E)$, a Heck function $\Phi_\beta$ on $\GL_n(E)$ is
    defined in~\cite{Guo2}*{Lemma~3.6}. Under the assumption $y_r \leq x_r
    = 1$, by the calculation in~\cite{Guo2}*{p.~139} the characteristic
    function of
        \[
        \{ X \in M_n(\fo_E), \abs{\det X} = y_r\}.
        \]
    By~\cite{Guo2}*{Lemma~3.6} the right hand side of~\eqref{eq:FL} equals
        \begin{equation}    \label{eq:FL_RHS}
        \int_{\GL_n(E)_{\beta, \mathrm{twisted}} \bs \GL_n(E)}
        \Phi_\beta (h^{-1}\beta \overline{h}) \rd h.
        \end{equation}

    If $\alpha\overline{\alpha}$ is not conjugate to an element in
    $\GL_n(\fo_F)$, then neither is $1- \beta\overline{\beta}$. In this
    case the above integral of $\Phi_\beta$ equals zero so both sides
    of~\eqref{eq:FL} equal zero. Now assume that $ \alpha\overline{\alpha}$
    is conjugate to an element in $\GL_n(\fo_F)$. Let $\cH(\GL_n(E))$ and
    $\cH(\GL_n(F))$ be the spherical Hecke algebra of $\GL_n(E)$ and
    $\GL_n(F)$ respectively, and
        \[
        \mathrm{bc}: \cH(\GL_n(E)) \to \cH(\GL_n(F))
        \]
    the usual base change map. By~\cite{Guo2}*{Corollary~3.8}, we have
    $\Psi_r = \mathrm{bc}(\Phi_\beta)$. Thus the desired
    equality~\eqref{eq:FL} is a consequence of the
    identities~\eqref{eq:FL_LHS}~\eqref{eq:FL_RHS}, and the base change
    fundamental lemma~\cite{AC}*{Chapter~1, Theorem~4.5}.

\item Let us now assume that $x_r = y_r > 1$, i.e. $\abs{\det \alpha}<1$
    and hence $\abs{\det (1- \overline{\alpha} \alpha)} = 1$. The integrand
    of~\eqref{eq:LHS_FL} is equivalent to the condition that
        \[
        h_2^{-1} \overline{\alpha} \overline{h_2}, \quad
        h_1^{-1} \overline{h_2}^{-1} \alpha h_2 \overline{h_1}, \quad
        h_2^{-1} (1 - \overline{\alpha} \alpha)h_2 \overline{h_1},
        \quad h_1^{-1}
        \]
    are all in $M_n(\fo_E)$. Note that
        \[
        \abs{\det h_1^{-1}} \leq 1, \quad
        \abs{\det h_2^{-1} (1 - \overline{\alpha} \alpha)h_2 \overline{h_1}}
        \leq 1,
        \]
    but
        \[
        \abs{\det (1 - \overline{\alpha} \alpha)} = 1.
        \]
    It follows that $\abs{\det h_1} = 1$ and hence $h_1 \in \GL_n(\fo_E)$.
    Moreover since $\abs{\det \alpha}<1$, we have $h_2^{-1}
    \overline{\alpha} \overline{h_2} \in M_n(\fo_E)$ implies $h_2^{-1} (1 -
    \overline{\alpha} \alpha)h_2 \in \GL_n(\fo_E)$. It follows that the
    integral~\eqref{eq:LHS_FL} reduces to
        \begin{equation}    \label{eq:easy_case_simplified}
        \int \id_{M_n(\fo_E)} ( h_2^{-1} \alpha \overline{h_2} ) \rd h_2.
        \end{equation}

    By~\cite{Guo2}*{Lemma~3.6}, and the calculation
    in~\cite{Guo2}*{p.~139-140}, under the assumption $x_r = y_r >1$, the
    right hand side of~\eqref{eq:FL} equals
        \begin{equation}    \label{eq:LHS_FL_case2}
        \int \id_{\{ X^{-1} \in M_n(\fo_E), \abs{\det X}_E = x_r\}}
        ( h_2^{-1} \beta \overline{h_2} ) \rd h_2.
        \end{equation}

    It remains to explain that the
    integrals~\eqref{eq:easy_case_simplified} and~\eqref{eq:LHS_FL_case2}
    equal. First the condition $\abs{\det X}_E = x_r$
    in~\eqref{eq:LHS_FL_case2} is redundant as twisted conjugation does not
    change the absolute value of the determinant. As $E/F$ is unramified,
    $-1 \in NE^\times$ and hence there is a $\delta \in \fo_E^\times$ such
    that $\delta \overline{\delta} = -1$. Since $\alpha \overline{\alpha}$
    is elliptic and $\abs{\det \alpha \overline{\alpha}}<1$, we conclude
    the absolute values of its eigenvalues are all strictly less than one
    (in a fixed splitting field of $F$). Thus
        \[
        (1 - \alpha \overline{\alpha})^{-\frac{1}{2}} = 1 -
        \binom{\frac{1}{2}}{1} \alpha \overline{\alpha} +
        \binom{\frac{1}{2}}{2} (\alpha \overline{\alpha})^2 + \cdots
        \]
    is convergent and gives a well-defined element in $\GL_n(F)$. Put
    $\gamma = \delta (1 - \alpha \overline{\alpha})^{-\frac{1}{2}}$. By
    assumption $\alpha \overline{\alpha} \in \GL_n(F)$ and therefore
    $\gamma$ commutes with $\alpha$ and thus $(\gamma \alpha
    \overline{\gamma \alpha})^{-1}$ and $\beta \overline{\beta}$ have the
    same characteristic polynomial. Replacing $\beta$ by its twisted
    conjugate, we may assume that $\beta^{-1} = \gamma \alpha$. Therefore
    we need to explain
        \begin{equation}    \label{eq:simple_case_final}
        \eqref{eq:easy_case_simplified} = \int \id_{M_n(\fo_E)}
        (h^{-1} \gamma \alpha \overline{h}) \rd h.
        \end{equation}
    Assume that $h^{-1} \alpha \overline{h} \in M_n(\fo_E)$. Then $h^{-1}
    \alpha \overline{\alpha} h \in M_n(\fo_E)$ and the absolute values of
    all eigenvalues of it are strictly less than one. This implies
        \[
        h^{-1} \gamma h = \delta \left( 1 -
        \binom{\frac{1}{2}}{1} h^{-1} \alpha \overline{\alpha} h +
        \binom{\frac{1}{2}}{2} (h^{-1} \alpha \overline{\alpha} h)^2 + \cdots
        \right)
        \]
    is convergent and is in $\GL_n(\fo_E)$ (note the only denominators in
    these binomial coefficients are powers of $2$). Thus $h^{-1} \gamma
    \alpha \overline{h} \in M_n(\fo_E)$. Conversely if $h^{-1} \gamma
    \alpha \overline{h} \in M_n(\fo_E)$, then $h^{-1} (1 - \alpha
    \overline{\alpha})^{-1} \alpha \overline{\alpha} h \in M_n(\fo_E)$.
    Since $1 + h^{-1} (1 - \alpha \overline{\alpha})^{-1} \alpha
    \overline{\alpha} h = h^{-1} (1 - \alpha \overline{\alpha})^{-1} h$ and
    the absolute values of $\alpha \overline{\alpha}$ are all strictly less
    than one, we have $h^{-1} (1 - \alpha \overline{\alpha}) h \in
    \GL_n(\fo_E)$ and hence $h^{-1} \alpha \overline{\alpha} h \in
    M_n(\fo_E)$. Then as before we conclude that $h^{-1} \gamma h \in
    \GL_n(\fo_E)$ and hence $h^{-1} \alpha \overline{h} \in M_n(\fo_E)$.
    This proves~\eqref{eq:simple_case_final}.
\end{enumerate}

This finishes the proof of Theorem~\ref{thm:FL} when $\alpha
\overline{\alpha}$ is elliptic.

\subsection{Parabolic descent}
To handle the nonelliptic case, we make use of the parabolic descent of the
orbital integrals. In this subsection, we deviate from the setup from
Theorem~\ref{thm:FL} and consider $O^{G'}(x, f')$ where $f' \in \cS(G'(F))$
and $x \in G'(F)$ is regular semisimple in general.

We fix integers $n_1, n_2$ with $n = n_1 + n_2$. Let $Q = LU$ be the
parabolic subgroup of $\GL_{2n}$ of the form
    \[
    L = \begin{pmatrix} m_1^{(1)} & & m_1^{(2)} \\ & m_2^{(1)} && m_2^{(2)} \\
    m_1^{(3)} & & m_1^{(4)} \\ & m_2^{(3)} && m_2^{(4)} \end{pmatrix},
    \quad
    U = \begin{pmatrix} 1  & u^{(1)} & & u^{(2)} \\
    & 1 \\ & u^{(3)} & 1 & u^{(4)} \\ &&& 1 \end{pmatrix},
    \]
where
    \[
    m_i = \begin{pmatrix}
    m_i^{(1)} &  m_i^{(2)} \\ m_i^{(3)} & m_i^{(4)} \end{pmatrix} \in \GL_{2n_i},
    \quad
    \begin{pmatrix} u^{(1)} & u^{(2)} \\
    u^{(3)} & u^{(4)} \end{pmatrix} \in M_{2n_1 \times 2n_2}.
    \]

By definition
    \[
    O^{G'}(x, f') = (\chi\widetilde{\eta})^{-1}(x)
    \int_{(H' \times H'')_x \bs (H' \times H'')}
    f'(h^{-1} x h'') \chi_{H'}(h)
    (\chi\eta)^{-1}(h'') \rd h \rd h''
    \]
Suppose $x = (x_1, x_2) \in L(E)$ and $x_i \in \GL_{2n_i}(E)$, $i = 1, 2$.
Assume that $x_i \overline{x_i}^{-1} = s'(\alpha_i)$ is regular semisimple in
$S'_{n_i}(F)$. Here we add the subscript to indicate the size of the various
groups and symmetric spaces. Let $\widetilde{r_1}, \cdots,
\widetilde{r_{n_1}}$ and $\widetilde{s_1}, \cdots, \widetilde{s_{n_2}}$ be
the eigenvalues of $ \alpha_1\overline{\alpha_1}$ and
$\alpha_2\overline{\alpha_2}$ respectively (in some fixed algebraic closure
of $F$). Put
    \[
    \lambda' = \prod_{1 \leq i \leq n_1, 1 \leq j \leq n_2}
    (\widetilde{r_i} - \widetilde{s_j})^{-1}.
    \]
Then one checks that $\lambda' \in F$.

Let $P = MN$ be the upper triangular parabolic subgroup of $\GL_n(E)$
corresponding to the partition $n = n_1 + n_2$. Here $N$ is the unipotent
radical, and $M$ is the standard diagonal block Levi subgroup. Write $h =
(h_1, h_2)$, $h_1, h_2 \in \GL_n(E)$. We make use of the Iwasawa
decomposition
    \[
    h_i = u_i m_i k_i, \ i = 1, 2, \quad h'' = u''m''k''
    \]
where $u_i \in N(E)$, $m_i \in M(E)$, $k_i \in \GL_n(\fo_E)$, $u' \in U(F)$,
and $k' \in \GL_{2n}(\fo_F)$. Then $O^{G'}(x, f')$ equals
    \[
    (\chi\widetilde{\eta})^{-1}(x) \int f'_K
    ( m^{-1} u^{-1} x u'' m'')  \delta_{P(E)}(m)^{-1}
    \delta_{Q(F)}(m'')^{-1}
    \chi_{H'}(m) (\chi\eta)^{-1}(m'')
    \rd u \rd u'' \rd m \rd m'',
    \]
where the domain of integration is $(m, m'') \in ((M\cap L)(E) \times L(F))_x
\bs ((M\cap L)(E) \times L(F))$, $u \in N(E) \cap U(E)$, $u'' \in U(F)$, and
    \[
    f'_K(g) = \int_{K_{H'}} \int_{K_{H''}} f(k_1^{-1} g k_2) \chi_{H'}(k_1)
    (\chi\eta)^{-1}(k_2) \rd k_1 \rd k_2,
    \]
$K_{H'} = \GL_n(\fo_E) \times \GL_n(\fo_E)$ is a maximal open compact
subgroup of $H'(F)$ and $K_{H''} = \GL_{2n}(\fo_F)$ is a maximal open compact
subgroup of $H''(F)$.

Let us now show that
    \[
    \delta_A: (N \cap U)(E) \times U(F) \to U(E),
    \quad (u, u'') \mapsto u^{-1}x u'' x^{-1}
    \]
is bijective and submersive everywhere. Direct computation gives that the
tangent map at $(u, u'')$ is given by
    \[
    (\fn \cap \fu)(E)  \times \fu(E) \to \fu(E),\quad
    (X, Y) \mapsto -u^{-1} X x u'' x^{-1} + u^{-1} x u'' Y x^{-1}.
    \]
Since $u$ and $u''$ are both unipotent, the determinant at any $(u, u'')$
equals the determinant at $(1, 1)$. At the point $(1, 1)$, the tangent can be
more explicitly written as
    \[
    M_{n_1 \times n_2}(E) \times M_{n_1 \times n_2}(E) \times M_{2n_1 \times 2n_2}(F)
    \to M_{2n_1 \times 2n_2}(E),
    \]
and
    \[
    (X_1, X_2, Y) \mapsto -  \begin{pmatrix} X_1 \\ & X_2  \end{pmatrix}
    +
    x_1 Y x_2^{-1}.
    \]

Let $\rho(x)$ be this map and we put
    \[
    \Delta(x) = \delta_{Q(E)}^{\frac{1}{2}}(x) \abs{\det \rho(x)}^{-1}.
    \]
This will be computed later. We make a change of variable $u^{-1} x u'
\mapsto w x$ where $w \in U(E)$. Then we conclude that the orbital integral
$O^{G'}(x, f')$ equals
    \[
    (\chi\widetilde{\eta})^{-1}(x)
    \int  \delta_{Q(E)}^{-\frac{1}{2}}(x) \Delta(x)
    f_{K}'( m^{-1} w x m'') \chi_{H'}(m)
    (\chi\eta)^{-1}(m'') \delta_{P(E)}(m)^{-1} \delta_{Q(F)}(m'')^{-1}
    \rd w \rd m \rd m''.
    \]
We note that $\delta_{P(E)}(m) = \delta_{Q(E)}(m)^{\frac{1}{2}}$. Therefore a
change of variable $w \mapsto mwm^{-1}$ yields that the above integral equals
    \[
    (\chi\widetilde{\eta})^{-1}(x)
    \int  \Delta(x)
    f'_K (w m^{-1} x m'')
    (\chi\eta)^{-1}(m') \delta_{Q(E)}(m^{-1} x m'')^{-\frac{1}{2}}
    \rd w \rd m \rd m''.
    \]
Put
    \[
    f'^{Q}(x) = \delta_{Q(E)}(x)^{-\frac{1}{2}}
    \int_{U''(E)} f'_K( w x) \rd w, \quad x \in L(E).
    \]
Then $f'^Q \in \cS(L(E))$. The map $f' \mapsto f'^{Q}$ is the well-known
parabolic descent map. Thus
    \[
    O^{G'}(x, f') = \Delta(x) O^{L}(x, f'^{Q}).
    \]

It remains to compute $\Delta(x)$. The determinant of the map $\rho(x)$ is
the same as the determinant of the map
    \[
    M_{2n_1 \times 2n_2}(F) \to
    M_{n_1\times n_2}(E) \times M_{n_1 \times n_2}(E),
    \quad
    Y \mapsto p (x_1 Y x_2^{-1})
    \]
where $p$ is the projection of a matrix in $M_{2n_1 \times 2 n_2}(E)$ to
$M_{n_1\times n_2}(E) \times M_{n_1 \times n_2}(E)$, the upper right and
lower left corner. As we are merely computing determinants, we may pass to
the algebraic closure and assume that $F$ is algebraically close. Then $E$ is
identified with $F \times F$ and the Galois conjugation exchanges two
components in $F \times F$.

Recall that $x \overline{x}^{-1} = s'(\alpha)$ and $\alpha = \begin{pmatrix}
\alpha_1 \\ & \alpha_2 \end{pmatrix}$. One checks readily that $\Delta(x)$
depends only on the conjugacy class of $\alpha_i \overline{\alpha_i}$.
Therefore we may assume that $\alpha_i$ is diagonal
    \[
    \alpha_i =
    \begin{pmatrix} a_1^{(i)}\\ & \ddots \\ && a_{n_i}^{(i)} \end{pmatrix}
    \in \GL_{n_i}(F \times F), \quad
    a_j^{(i)} = (b_j^{(i)}, c_j^{(i)}) \in F^\times \times F^\times.
    \]
and
    \[
    x_i = \left( \widetilde{x_i}, 1 \right) \in \GL_{2n_i}(F) \times \GL_{2n_i}(F),
    \quad \widetilde{x_i}  = \begin{pmatrix} B^{(i)} & 1 \\ 1 - B^{(i)}C^{(i)} &
    -C^{(i)} \end{pmatrix},
    \]
and
    \[
    B^{(i)} =
    \begin{pmatrix} b_1^{(i)}\\ & \ddots \\ && b_{n_i}^{(i)} \end{pmatrix}, \
    C^{(i)} =
    \begin{pmatrix} c_1^{(i)}\\ & \ddots \\ && c_{n_i}^{(i)} \end{pmatrix}
    \in \GL_{n_i}(F).
    \]
With these choices, the determinant we would like to compute is the product
of various determinants of the linear transforms of the form
    \[
    M_{2 \times 2}(F) \to F \times F \times F \times F, \quad
    Y \mapsto p_{st} \left( \begin{pmatrix} b_s^{(1)} & 1 \\
    1 - b_s^{(1)} c_s^{(1)} & -c_s^{(1)} \end{pmatrix}
    Y \begin{pmatrix} b_t^{(2)} & 1 \\
    1 - b_t^{(2)} c_t^{(2)} & -c_t^{(2)} \end{pmatrix}^{-1}, Y \right)
    \]
where $p_{st}$ is the projection to the upper right and lower left corner,
and the product ranges over all $1\leq s \leq n_1$ and $1 \leq t \leq n_2$.
Direct computation gives that the determinant is $b_s^{(1)} c_s^{(1)} -
b_t^{(2)} c_t^{(2)}$, which in term equals $a_s^{(1)} \overline{a_s^{(1)}} -
a_t^{(2)} \overline{a_t^{(2)}}$. According the special form of $\alpha_i$ we
took, we have
    \[
    a_i^{(1)} \overline{a_i^{(1)}} - a_j^{(2)} \overline{a_j^{(2)}} =
    \widetilde{r_i} - \widetilde{t_j}.
    \]
Moreover $\delta_{Q(E)}(x) = 1$. It follows that $\Delta(x) = \lambda'$.

We summarize the above computation in the following proposition.

\begin{prop}    \label{prop:parabolic_descent}
Let the notation be as above. Then
    \[
    O^{G'}(x, f') = \lambda' \cdot O^L(x, f'^{Q}).
    \]
\end{prop}

\subsection{Reduction to the elliptic case}

Let us come back to the setup of Theorem~\ref{thm:FL}. Assume that $\alpha
\overline{\alpha} \in \GL_n(E)$ is regular semisimple but not elliptic (in
the usual sense). Then we can find positive integers $n_1, n_2$ with $n_1+n_2
= n$, $P = MN$ be the standard blocked upper triangular parabolic subgroup of
$\GL_n$ corresponding to this partition, and $\alpha$ is twisted conjugate to
$\begin{pmatrix} \alpha_1 \\ & \alpha_2 \end{pmatrix} \in M(E)$. Since
$s'(\alpha)$ and $g(\beta)$ match, we may find $\beta$ is twisted conjugate
to $\begin{pmatrix} \beta_1 \\ & \beta_2 \end{pmatrix} \in M(E)$,
$s'(\alpha_i) \in S'_{2n_{i}}$ matches $g(\beta_i) \in \GL_{2n_i}(E)$, $i =
1, 2$.

Put $f' = \id_{\GL_{2n}(\fo_E)}$ in Proposition~\ref{prop:parabolic_descent}.
Since $f'^{Q} = \id_{\GL_{2n_1}(\fo_E)} \otimes \id_{\GL_{2n_2}(\fo_E)}$, and
    \[
    O^{G'}(x, \id_{\GL_{2n}(\fo_E)}) = O^{S'}(x \overline{x}^{-1},
    \id_{S'(\fo_F)}),
    \]
we conclude
    \[
    O^{S'}(s'(\alpha), \id_{S'(\fo_F)}) = \lambda'
    O^{S'_{2n_1}}(s'(\alpha_1), \id_{S'_{2n_1}(\fo_F)})
    O^{S'_{2n_2}}(s'(\alpha_2), \id_{S'_{2n_2}(\fo_F)}).
    \]

Let $r_1, \cdots, r_{n_1}$ be the eigenvalues of $\beta_1\overline{\beta_1}$
and $s_1, \cdots, s_{n_2}$ be the eigenvalues of $\beta_2 \overline{\beta_2}$
(in some fixed algebraic closure of $F$). Put
    \[
    \lambda = \frac{\abs{ \det(1 - \beta_1\overline{\beta_1})}^{n_2}
    \abs{\det (1 - \beta_2\overline{\beta_2})}^{n_1}}
    {\abs{\prod_{1 \leq i \leq n_1, 1 \leq j \leq n_2}(r_i - s_j)}}.
    \]
Then one checks that $\lambda \in F$. By~\cite{Guo2}*{Proposition~2.2} we
have
    \[
    O^{\GL_{2n}} (g(\beta), \id_{\GL_{2n}(\fo_F)}) =
    \lambda
    O^{\GL_{2n_1}} (g(\beta_1), \id_{\GL_{2n_1}(\fo_F)})
    O^{\GL_{2n_2}} (g(\beta_2), \id_{\GL_{2n_2}(\fo_F)}).
    \]

Since $s'(\alpha_i)$ and $g(\beta_i)$ match, $i= 1, 2$, the elements $\beta_i
\overline{\beta_i}$ and $-(1-\alpha_i \overline{\alpha_i})( \alpha_i
\overline{\alpha_i})^{-1}$ have the same characteristic polynomial. It
follows that
    \[
    \lambda' = \lambda.
    \]
Then Theorem~\ref{thm:FL} follows by induction on $n$. This finishes the
proof of Theorem~\ref{thm:FL}.

\appendix

\section{Convergence of the elliptic part}
\label{s:A1}

The goal of this appendix is to explain the absolute convergence of the
elliptic part of the relative trace formula. We will prove
    \begin{equation}    \label{eq:abs_convergent_elliptic}
    \int_{H'(F) \bs H'(\bA_F)^1}
    \sum_{x \in S'(F)_{\el}} \abs{f'(h^{-1} x \overline{h})} \rd h
    \end{equation}
is convergent for all $f \in \cS(S'(\bA_F))$, where
    \[
    H'(\bA_F)^1 = \{ (h_1, h_2) \in H'(\bA_F) \mid \abs{\det h_1 h_2} = 1\}.
    \]
This implies the absolute convergence of~\eqref{eq:geometric_split}. The
proof of the absolute convergence of~\eqref{eq:geometric_nonsplit} is
similar.

Let $P_0$ be the usual upper triangular Borel subgroup of $\GL_n$ and $P' =
\Res_{E/F} P_0 \times P_0$ be a minimal parabolic subgroup of $H'$. Let $c$
be a real number with $0<c<1$ and $T_c$ the subset of the diagonal torus in
$\GL_{2n}(\R)$ consisting of
    \[
    \{(a_1, \cdots, a_n, b_1, \cdots, b_n) \in (\R_{>0})^{2n}
    \mid a_i a_{i+1}^{-1} \geq c, \ b_i b_{i+1}^{-1} \geq c, \
    a_1\cdots a_n b_1 \cdots b_n = 1 \}.
    \]
Let $T_c$ be diagonally embedded in $H'(F_\infty)$ and identify it with its
image. Fix a maximal compact subgroup $\cK$ of $H'(\bA_F)$. Then reduction
theory gives that there is a compact subgroup $\omega \subset P'(\bA_F)$,
such that $H'(\bA_F)^1 = H'(F) \cG$ and
    \[
    \cG = \{ pak \mid p \in \omega, \ a \in T_c, \ k \in \cK\}.
    \]
Thus we only need to prove that
    \[
    \int_{\omega} \int_{T_c} \int_{\cK}
    \sum_{x \in S'(F)_{\el}}
    \abs{f'((pak)^{-1} x (\overline{pak}))}
    \delta_P(a)^{-1} \rd k \rd a \rd p
    \]
is absolutely convergent. By the definition of $T_c$, there is a compact
subset $\Omega$ of $H'(\bA_F)$ such that if $p \in \omega$, $a \in T_c$, $k
\in \cK$, then $a^{-1}pak \in \Omega$. It follows that we only need to prove
that
    \[
    \int_{T_c} \sum_{x \in S'(F)_{\el}} \abs{f'(a^{-1} x a)}
    \delta_P(a)^{-1} \rd a
    \]
is absolutely convergent for all Schwartz functions $f'$ on $S'(\bA_F)$. It
is enough to consider $f' = \otimes_v f'_v$, where $f'_{v}$ is a Schwartz
function on $S'(F_v)$. Since $f'_v$ is compactly supported if $v \nmid
\infty$ and $T_c \subset H'(F_v)$, we just need to prove that
    \begin{equation}    \label{eq:sum_over_lattice}
    \int_{T_c} \sum_{x \in S'(L)_{\el}} \abs{f'_\infty(a^{-1} x a)}
    \delta_P(a)^{-1} \rd a
    \end{equation}
is absolutely convergent for any Schwartz function $f'_\infty$ on
$S'(F_\infty)$ and any fractional ideal $L$ of $\fo_F$. Note that $L$ is
discrete in $F_\infty$.

We fix some notation. Let $v \mid \infty$ be an infinite place we write
$\abs{\cdot}$ for the usual absolute value. If $x = (x_v) \in F_\infty$ we
write $\abs{x}$ for $\max_{v \mid \infty} \abs{x_v}$. If $X = (x_{ij}) \in
M_n(F_\infty)$, then we write $\aabs{X} = \max_{ij} \abs{x_{ij}}$.

Let us divide the integral into two pieces depending on $a_1\cdots a_n >1$ or
not. We will treat the case $a_1 \cdots a_n>1$. The case $a_1 \cdots a_n<1$
can be handled in exactly the same way by noting that $b_1 \cdots b_n >1$
under this assumption.

From now on assume that $a_1 \cdots a_n >1$. Then $b_1 \cdots b_n = (a_1
\cdots a_n)^{-1}<1$.

Since $L$ is a fractional ideal, there is a constant $c_L>0$ such that if $x
\in S'(L)$ and $u$ is a nonzero entry of $x$ then $\abs{u} \geq c_L$. This is
where the discreteness of $L$ in $F_\infty$ is used.

We write $x \in S'(F_\infty)$ as $\begin{pmatrix} A & B \\ C & D
\end{pmatrix}$. Fix a positive polynomial $P_1$ such that
    \[
    P_1(x) \geq \max\{ \aabs{A}, \aabs{B}, \aabs{C}, \aabs{D}\}.
    \]
Here polynomial means that we view $S'(F_\infty)$ as a real manifold and a
$P_1$ is a real positive polynomial, in other worlds, if $a_{ij}$ is an entry
of $A$, then both $a_{ij}$ and $\overline{a_{ij}}$ might appear in the
polynomial $P_1$.

If $x = \begin{pmatrix} A & B \\ C & D \end{pmatrix} \in S'(F_\infty)_{\el}$,
we write $A  = (x_{ij})$. Since the characteristic polynomial of $A
\overline{A}$ is irreducible over $F$, for every $i_0 = 1, \cdots, n-1$,
there is a $j \geq i_0+1$ and $i \leq i_0$ such that $x_{ji} \not=0$ (for
otherwise $A$ is contained in a proper parabolic subgroup of $\GL_n(E)$).
Thus $\abs{x_{ji}} \geq c_L$. Something similar holds for the entries of $D$.
This is where the condition ``elliptic'' is used.

By the choice of $P_1$ we have
    \[
    P_1(a^{-1}xa) \geq \abs{a_j^{-1} x_{ji} a_i} \geq c_L a_i a_j^{-1}.
    \]
Since $a \in T_c$ we have
    \[
    a_i \geq c a_{i+1} \geq \cdots \geq c^{i_0-i} a_{i_0}, \quad
    a_j^{-1} \geq c a_{j-1}^{-1} \geq \cdots \geq c^{j-i_0-1} a_{i_0+1}^{-1}.
    \]
Therefore
    \begin{equation}    \label{eq:ratio_a}
    P_1(a^{-1} x a) \geq c_L c^{j-i-1} a_{i_0} a_{i_0+1}^{-1}
    \geq c_L c^{n-2} a_{i_0} a_{i_0+1}^{-1}.
    \end{equation}
Note that we used the fact that $0<c<1$. So we obtain
    \[
    a_n^{-1} = (a_1 \cdots a_n)^{-\frac{1}{n}}
    \prod_{i=1}^{n-1} (a_i a_{i+1}^{-1})^{\frac{i}{n}} \leq
    (a_1 \cdots a_n)^{-\frac{1}{n}} \left( c_L^{-1} c^{-(n-2)}
    P_1(a^{-1}xa) \right)^{\frac{n-2}{2}},
    \]
and
    \[
    a_1 = (a_1 \cdots a_n)^{\frac{1}{n}}
    \prod_{i=1}^{n-1} (a_i a_{i+1}^{-1})^{\frac{n-i}{n}} \leq
    (a_1 \cdots a_n)^{\frac{1}{n}} \left( c_L^{-1} c^{-(n-2)}
    P_1(a^{-1}xa) \right)^{\frac{n-2}{2}},
    \]
Similarly by considering $D$, we conclude
    \begin{equation}    \label{eq:ratio_b}
    P_1(a^{-1} x a)
    \geq c_L c^{n-2} b_{i_0} b_{i_0+1}^{-1},
    \end{equation}
and
    \[
    b_n^{-1} \leq
    (b_1 \cdots b_n)^{-\frac{1}{n}} \left( c_L^{-1} c^{-(n-2)}
    P_1(a^{-1}xa) \right)^{\frac{n-2}{2}},
    \quad
    b_1 \leq
    (b_1 \cdots b_n)^{\frac{1}{n}} \left( c_L^{-1} c^{-(n-2)}
    P_1(a^{-1}xa) \right)^{\frac{n-2}{2}}.
    \]

For any $i, j = 1, \cdots, n$ we also have
    \[
    P_1(a^{-1}xa) \geq \abs{a_i^{-1} x_{ij} a_j},
    \]
and thus
    \begin{equation}    \label{eq:x}
    \abs{x_{ij}} \leq  a_i a_j^{-1} P_1(a^{-1} xa)
    \leq c^{-(i-1) - (n-j)} a_1 a_n^{-1}P_1(a^{-1} x a)
    \leq c^{-(2n-2)} \left( c_L^{-1} c^{-(n-2)}
    P_1(a^{-1} x a) \right)^{n-2}.
    \end{equation}
Write $C = (z_{ij})$, $D = (w_{ij})$. Similar considerations also give
    \begin{equation}    \label{eq:z}
    \begin{aligned}
    \abs{z_{ij}} &\leq (a_1 \cdots a_n)^{-\frac{1}{n}}
    (b_1 \cdots b_n)^{\frac{1}{n}}
    c^{-(2n-2)} \left( c_L^{-1} c^{-(n-2)}
    P_1(a^{-1} x a) \right)^{n-2}\\
    &=
    (a_1 \cdots a_n)^{-\frac{2}{n}}
    c^{-(2n-2)} \left( c_L^{-1} c^{-(n-2)}
    P_1(a^{-1} x a) \right)^{n-2},
    \end{aligned}
    \end{equation}
and
    \begin{equation}    \label{eq:w}
    \abs{w_{ij}} \leq  a_i a_j^{-1} P_1(a^{-1} x a)
    \leq c^{-(2n-2)} \left( c_L^{-1} c^{-(n-2)}
    P_1(a^{-1} x a) \right)^{n-2}.
    \end{equation}

To summarize, multiplying the inequalities~\eqref{eq:ratio_a},
\eqref{eq:ratio_b}, \eqref{eq:x}, \eqref{eq:z} and~\eqref{eq:w}, we obtain a
positive polynomial function $P$ on $S'(F_\infty)$ such that
    \begin{equation}    \label{eq:separation}
    \prod_{i=1}^{n-1} \left( \frac{a_i}{a_{i+1}} \frac{b_i}{b_{i+1}} \right)
    \max\{\aabs{A}, \aabs{C}, \aabs{D}\}
    \leq (a_1\cdots a_n)^{-\frac{2}{n}} P(a^{-1} x a).
    \end{equation}

Let us now fix a positive homogeneous polynomial $Q$ in $M_n(F_\infty)$ of a
large degree $M$. Consider
    \[
    \phi(x) = f'_{\infty}(x) Q(B).
    \]
This is still a Schwartz function on $S'(F_{\infty})$. Then
    \[
    \eqref{eq:sum_over_lattice} = \int \sum_{x \in S'(L)_{\el}}
    \phi(a^{-1} x a) (a_1 \cdots a_n)^{2M} Q(B)^{-1}\delta_{P'}(a)^{-1} \rd a.
    \]
where as before we write each $x = \begin{pmatrix} A & B \\ C &D
\end{pmatrix}$.

Since $\phi$ is Schwartz, it is bounded by the reciprocal of any polynomial
and in particular by $P^{-N}$ when $N$ is large enough, thus
by~\eqref{eq:separation} we have
    \[
    \begin{aligned}
    \eqref{eq:sum_over_lattice} &\leq
    \int \sum_{x \in S'(L)_{\el}} \max\{\aabs{A}, \aabs{C}, \aabs{D}\}^{-N}
    (a_1 \cdots a_n)^{\frac{-2N}{n}}  (a_1 \cdots a_n)^{2M} Q(B)^{-1} \rd a\\
    & = \sum_{x \in S'(L)_{\el}}
    \max\{\aabs{A}, \aabs{C}, \aabs{D}\}^{-N} Q(B)^{-1} \times
    \int (a_1 \cdots a_n)^{\frac{-2N}{n}}  (a_1 \cdots a_n)^{2M} \rd a
    \end{aligned}
    \]
Here $N$ is a sufficiently large real number, and the integration is over $a
\in T_c$ and $a_1 \cdots a_n >1$. The point is that the variables in the
integral, i.e. $a_1, \cdots, a_n, b_1, \cdots, b_n$, and the variables in the
sum, i.e. $x = \begin{pmatrix} A & B \\ C & D \end{pmatrix}$, are separated.
Thus when $N>>M>>0$, both the sum and the integral are convergent. This
proves the convergence of~\eqref{eq:sum_over_lattice} and hence the absolute
convergence of~\eqref{eq:geometric_split}.

\section{Elliptic representations}
\label{s:A2}

The goal of this appendix is to sketch a proof of
Proposition~\ref{prop:elliptic_rep}. To simplify notation, we fix in this
subsection a nonarchimedean nonsplit place $v$ of $F$ and suppress it from
all notation. Thus $F$ stands for a nonarchimedean local field of
characteristic zero. To shorten notation we also write $H$ for its group of
$F$-points $H(F)$. The equalities in the proof usually depend on the choice
of the measures. But such choices are not essential to the final result. Thus
we should interpret the equalities in the proof as equalities up to a nonzero
constant depending only the choice of the measures.

\subsection{Results on orbital integrals}
First we need some results on the nilpotent orbital integrals and Shalika
germ. let $\fs$ be the tangent space of $S$ at $1$, with an action of $H$ by
conjugation. An element $x \in \fs$ is called regular semisimple if $H_x$ is
a torus of dimension $n$, and it is called elliptic if in addition $H_x$ is
an elliptic torus modulo the split center of $H$. Regular semisimple orbital
integrals has been defined and studied in~\cite{CZhang}. An $H$-orbit in
$\fs$ is called nilpotent if its closure contains $0$. Nilpotent orbital
integrals have been defined in~\cite{Guo3}. In particular if $\cO$ is an
nilpotent orbit in $\fs$, it is proved in~\cite{Guo3} that the integral
    \[
    \int_{\cO} f(x) \rd x, \quad f\in \cS(\fs),
    \]
is absolutely convergent, where $\rd x$ is an invariant Radon measure on
$\cO$. Moreover it is proved that the Fourier transform $\widehat{\mu_{\cO}}$
of the distribution $\mu_{\cO}$ is a locally integrable function on $\fs$. If
$\cO = \{0\}$ is the smallest nilpotent orbit, then obviously
$\widehat{\mu_{\cO}}(X)$ is a nonzero constant. More importantly $\mu_{\cO}$
and $\widehat{\mu_{\cO}}$ have the following homogeneity property. If $t \in
F^{\times}$, then
    \[
    \mu_{\cO}(f_{t}) = \abs{t}^{\dim \cO} \mu_{\cO}(f),
    \quad f_t(X) = f(t^{-1}X).
    \]
This follows from the explicit formula for $\mu_{\cO}$ given
in~\cite{Guo3}*{Proposition~5.1}. Taking Fourier transform we conclude that
    \[
    \widehat{\mu_{\cO}}(f_t) = \abs{t}^{2n^2 - \dim \cO}
    \widehat{\mu_{\cO}}(f).
    \]
The most important point is that $\dim \cO < n^2$ for all $\cO$, and thus
    \begin{equation}    \label{eq:homogeneity_inequality}
    \dim \cO < 2n^2 - \dim \cO'
    \end{equation}
for any two nilpotent orbits $\cO$ and $\cO'$.

As in the classical situation of Harish-Chandra, we have the Shalika germs.
Let $\exp: \fs \to S$ be the exponential map, defined in an $H$-invariant
neighbourhood $0 \in \fs$. For any $f \in \cS(G)$, we define in an
$H$-invariant neighbourhood of $0 \in \fs$ a function $f_{\natural}$ by
requiring that
    \[
    \int_H f(gh) \chi(gh)^{-1} \rd h = f_{\natural}(X)
    \]
if $g \theta(g)^{-1}  = \exp X$. There is a unique $H$-invariant real valued
function $\Gamma_\cO$ on the regular semisimple locus of $\fs$ for each
nilpotent orbit $\cO$ with the following properties.
\begin{enumerate}
\item For any $f \in \cS(\fs)$, there is an $H$-invariant neighbourhood
    $U_f$ of $0 \in \fs$ such that
        \begin{equation}    \label{eq:Shalika_germ_CSA}
        O^G(g, f) = \sum_{\cO} \Gamma_\cO(X) \mu_{\cO}(f_{\natural}).
        \end{equation}
    for all regular semisimple $g \in U_f$, such that $g\theta(g)^{-1} =
    \exp(X)$.

\item For all $t \in F^\times$ and all regular semisimple $X$, we have
    \[
    \Gamma_{\cO}(tX) =
    \abs{t}^{- \dim \cO}  \Gamma_{\cO}(X).
    \]
\end{enumerate}

\begin{lemma}   \label{lemma:OI}
The Shalika germs $\Gamma_{\cO}$ are linearly independent. They are not
identically zero in any neighbourhood of $0$. If $\cO = \{0\}$ the minimal
nilpotent orbit, then $\Gamma_0(X) = 0$ if $X$ is not elliptic in $\fs$.
\end{lemma}

\begin{proof}
The linear independence is proved by exactly the same argument as in the
classical case of Harish-Chandra. The key to this argument is the
inequality~\eqref{eq:homogeneity_inequality}, and the rest of the argument is
essentially formal, cf.~\cite{Kottwitz2}*{Section~27}
and~\cite{Xue3}*{Section~7}. The fact that $\Gamma_0(X) = 0$ if $X$ is not
elliptic is proved using parabolic descent~\cite{CZhang}*{Subsection~6.1} and
the homogeneity property of $\Gamma_{\cO}$'s.
\end{proof}

\subsection{Characters of supercuspidal representations}
Now we recall that by~\cite{BP2}*{Proposition~3.2.1}, in the case $\pi$ being
supercuspidal, up to some nonzero constant depending only on the choice of
the measures and the linear form $\ell$, we have
    \begin{equation}    \label{eq:integration_matrix_coeffcient}
    \ell (v) \overline{\ell(w)} = \int_{Z \bs H} \langle v, \pi(h) w \rangle
    \chi(h)^{-1} \rd h
    \end{equation}
for all $v, w \in \pi$. Thus if $\varphi \in \cS(G)$ then
    \[
    J_{\pi}(\varphi) = \sum_v
    \int_{Z \bs H} \langle \pi(\varphi) v, \pi(h) w \rangle
    \chi(h)^{-1} \rd h,
    \]
where the sum runs over an orthonormal basis of $\pi$.

Recall  from  \eqref{eq:tran fact} that we have defined a transfer factor
    \[
    \kappa^G(g) = \chi(A), \quad g^{-1} =
    \begin{pmatrix} A & * \\ * &* \end{pmatrix}
    \]
for all regular semisimple $g \in G$. We put $\widetilde{\Theta}_{\pi}(g) =
\kappa^G(g) \Theta_{\pi}(g)$. Then $\widetilde{\Theta}_{\pi}$ is left and
right $H$-invariant, and we can view it as a function on $S$ which is
$H$-conjugate invariant. By~\cite{RR}*{Theorem~7.11}, if $X$ is in a small
neighbourhood of $0 \in \fs$, $g\in G$, $g \theta(g)^{-1} = \exp X$, we have
    \begin{equation}    \label{eq:character_expansion}
    \widetilde{\Theta}_{\pi} (g) =
    \sum_{\cO} c_{\cO} \widehat{\mu_{\cO}}(X).
    \end{equation}
The case treated in~\cite{RR} does not involve the character, but the same
argument goes through without change in our setup.

\begin{lemma}   \label{lemma:integral_matrix_coefficients}
Let $v, w \in \pi$, and $f(g) = \langle v, \pi(g) w \rangle$ be the matrix
coefficient. Then we have
    \begin{equation}    \label{eq:explicit_character}
    \kappa^G(g) O^G(g, f) =
    \widetilde{\Theta}_\pi(g) \int_{Z \bs H} f(h) \chi(h)^{-1} \rd h.
    \end{equation}
\end{lemma}

\begin{proof}
It is enough to prove that for any $\varphi \in \cS(G)$ supported in the
elliptic locus, we have
    \begin{equation}    \label{eq:explict_character_2}
    \int_{G} \varphi(g) \kappa^G(g) O^G(g, f) \rd g
    = J_\pi(\varphi \kappa^G)
    \int_{Z \bs H} f(h) \chi(h)^{-1}\rd h.
    \end{equation}
Though $\kappa^G$ is not defined on all $G$, as $\varphi$ is locally constant
and compactly supported in the elliptic locus, $\varphi \kappa^G \in \cS(G)$
and $J_\pi(\varphi \kappa^G)$ makes sense.

Let us first note that because $(H \times H)_g$ is an elliptic torus modulo
the center of $G$, up to some nonzero constant depending only on the choice
of the measures, the orbital integral of $f$ equals
    \[
    \kappa^G(g)
    \int_{Z \bs H \times Z \bs H} f(h_1^{-1} g h_2) \chi(h_1^{-1}h_2)^{-1}
    \rd h_1 \rd h_2.
    \]
As $\varphi$ is supported on the elliptic locus, we have
    \[
    \int_{G} \varphi(g) O^G(g, f) \rd g = \int_G \int_{Z \bs H \times Z \bs H}
    \varphi(g) \kappa^G(g)
    f(h_1^{-1} g h_2) \chi(h_1^{-1}h_2)^{-1} \rd h_1 \rd h_2 \rd g.
    \]
The right hand of this integral is absolutely convergent. Thus we may change
the order of integration and conclude that this integral equals
    \begin{equation}    \label{eq:integral_matrix_coefficient}
    \int_{Z \bs H \times Z \bs H} \langle v,
    \pi(h_1^{-1})
    \pi(\varphi \kappa^G) \pi(h_2) w \rangle
    \chi(h_1^{-1} h_2)^{-1} \rd h_1 \rd h_2.
    \end{equation}
Since $\pi$ is admissible, we may find elements $v_1, \cdots, v_r$ and $w_1,
\cdots, w_r$ in $\pi$ such that
    \[
    \pi(\varphi \kappa^G) v = \sum_{i = 1}^r \langle v, v_i \rangle w_i.
    \]
It then follows by~\eqref{eq:integration_matrix_coeffcient} that
    \[
    \eqref{eq:integral_matrix_coefficient}=
    \sum_{i= 1}^r \ell(v) \overline{\ell(w_i)} \ell(v_i) \overline{\ell(w)}.
    \]
We also have
    \[
    J_{\pi}(\varphi \kappa^G) = \sum_u
    \ell(\pi(\varphi \kappa^G)u)
    \overline{\ell(u)} = \sum_u \sum_{i=1}^r
    \langle u, v_i \rangle \ell(w_i)
    \overline{\ell(u)} =\sum_{i=1}^r \ell(w_i) \overline{\ell(v_i)}.
    \]
Thus~\eqref{eq:explict_character_2} follows by another application
of~\eqref{eq:integration_matrix_coeffcient}.
\end{proof}

\begin{proof}[Proof of Proposition~\ref{prop:elliptic_rep}]
Let $X$ be in a small neighbourhood of $0 \in \fs$, $g\in G$, $g
\theta(g)^{-1} = \exp X$. The character
expansion~\eqref{eq:character_expansion} gives
    \[
    \widetilde{\Theta}_{\pi} (g) =
    \sum_{\cO} c_{\cO} \widehat{\mu_{\cO}}(X).
    \]
Note that $X$ is elliptic if and only if $g$ is elliptic. Since $\cO = \{0\}$
is the only nilpotent orbit with $\widehat{\mu_{\cO}}(tX) =
\widehat{\mu_{\cO}}(X)$ for all $X \in \fs$ and $t \in F^\times$, to show
that $\widetilde{\Theta}_{\pi}(g) \not=0$ for some elliptic $g \in G$ which
sufficiently close to $1$, we only need to show that $c_0 \not=0$.

We find a matrix coefficient $f$ such that
    \[
    \int_{Z \bs H} f(h) \chi(h)^{-1} \rd h \not=0.
    \]
For this $f$ we consider the expansion of both sides
of~\eqref{eq:explicit_character} when $g$ is close to $1$ and is elliptic. We
have
    \[
    \sum_{\cO} \Gamma_{\cO}(X) \mu_{\cO}(f_{\natural})
    = \sum_{\cO} c_{\cO} \widehat{\mu_{\cO}}(X) \times
    \int_{Z \bs H} f(h) \chi(h)^{-1} \rd h.
    \]
The only terms on both sides of the expansion that are invariant under the
scaling $X \to tX$ are those corresponding to $\cO = 0$. Thus by the
homogeneity property of $\Gamma_{\cO}$ and $\widehat{\mu_{\cO}}$, we conclude
that
    \[
    \Gamma_0(X) \mu_0(f_{\natural}) = c_{0}
    \widehat{\mu_{0}}(X) \times
    \int_{Z \bs H} f(h) \chi(h)^{-1} \rd h.
    \]
By our choice of $f$ we have
    \[
    \mu_0(f_{\natural}) = \int_{Z \bs H} f(h) \chi(h)^{-1} \rd h \not=0.
    \]
If $c_0 = 0$, then $\Gamma_0(X) = 0$ if $X$ is elliptic in a neighbourhood of
$0$. By Lemma~\ref{lemma:OI}, $\Gamma_0(Y) = 0$ if $Y$ is not elliptic and
hence is identically zero in a neighbourhood of $0$. This is impossible by
Lemma~\ref{lemma:OI}. Therefore $c_0 \not=0$.
\end{proof}

\end{document}